\documentclass[11 pt,reqno]{amsart}

\usepackage{amsmath,amsfonts,amssymb,amsthm,enumitem}
\usepackage{color, comment,bbm}
\usepackage{tikz}
\usepackage{ytableau}
\usepackage{float}
\usepackage{fullpage}
\usetikzlibrary{arrows.meta,chains,matrix,decorations.pathreplacing}
\usetikzlibrary{patterns}
\usepackage{mathabx}

\usepackage{graphicx}
\usepackage[all]{xy}

\definecolor{references}{rgb}{0,0,1}
\RequirePackage[pdftex,
colorlinks = true,
urlcolor = references, 
citecolor = references, 
linkcolor = references, 
]
{hyperref}
 
\newtheorem*{theorem*}{Theorem}

\newtheorem{theorem}{Theorem}[section]
\newtheorem{lemma}[theorem]{Lemma}
\newtheorem{proposition}[theorem]{Proposition}
\newtheorem{corollary}[theorem]{Corollary}
\newtheorem{conjecture}[theorem]{Conjecture}

\newtheorem{introthm}{Theorem}

\theoremstyle{definition}
\newtheorem{definition}[theorem]{Definition}
\newtheorem{example}[theorem]{Example}
\newtheorem{construction}[theorem]{Construction}

\theoremstyle{remark}
\newtheorem{remark}[theorem]{Remark}

\numberwithin{equation}{section}

\newcommand{\iinv}{^{-1}}
\newcommand{\inv}{\ensuremath\mathrm{inv}}

\newcommand{\codinv}{\ensuremath\mathtt{codinv}}
\newcommand{\dinv}{\ensuremath\mathtt{dinv}}
\newcommand{\coarea}{\ensuremath\mathtt{coarea}}
\newcommand{\area}{\ensuremath\mathtt{area}}

\newcommand{\ngen}{\ensuremath{\mathtt{ngen}}}

\newcommand{\mgen}{\ensuremath{\mathtt{mgen}}}
\newcommand{\cogen}{\ensuremath{\mathtt{cogen}}}
\newcommand{\cogennn}{\ensuremath{\mathtt{cogen}_{\ge 0}}}

\newcommand{\Gaps}{\mathrm{Gaps}}
\newcommand{\gen}{\mathtt{gen}}
\newcommand{\KT}{K^{\T}}

\newcommand{\KnR}{K^{\R^3}}

\newcommand{\cox}{\ensuremath\mathrm{cox}}

\newcommand{\Inv}{\ensuremath\mathrm{Inv}}
\newcommand{\SSYT}{\ensuremath\mathrm{SSYT}}
\newcommand{\SYT}{\ensuremath\mathrm{SYT}}
\newcommand{\T}{\ensuremath\mathbb{T}}
\newcommand{\Z}{\ensuremath\mathbb{Z}}
\newcommand{\C}{\ensuremath\mathbb{C}}
\newcommand{\D}{\ensuremath\mathrm{D}}
\newcommand{\AD}{\ensuremath\mathcal{D}}
\newcommand{\Aa}{\ensuremath\mathcal{A}}

\newcommand{\PC}{\ensuremath\mathcal{P}}
\newcommand{\Br}{\mathrm{Br}}
\newcommand{\KR}{\text{KR}}
\newcommand{\Hik}{\ensuremath\mathcal{H}}

\newcommand{\N}{\mathbb{N}}

\newcommand{\St}{\mathrm{St}}
\newcommand{\ST}{\widehat{\mathrm{{St}}}}
\newcommand{\I}{\mathrm{I}}
\newcommand{\tr}{{\tilde{r}}}
\newcommand{\ts}{{\lfloor{s}\rfloor}}

\definecolor{teal}{rgb}{.1,.7,.7}
\definecolor{purp}{rgb}{.7,.1,.7}
\definecolor{orn}{rgb}{1,.6,.2}

\newcommand{\SBim}{\mathbb{S}\mathrm{Bim}}
\newcommand{\HHH}{\operatorname{HHH}}
\newcommand{\HH}{\operatorname{HH}}
\newcommand{\one}{\mathbbm{1}}
\newcommand{\JM}{\mathrm{JM}}
\def\R{\mathbb {R}}

\renewcommand{\k}{\mathbb{K}}

\newcommand{\kMod}[1]{\mathbbm{k}\text{-Mod}^{#1}}
\newcommand{\Hom}{\operatorname{Hom}}

\renewcommand{\a}{\alpha}
\renewcommand{\b}{\beta}
\newcommand{\e}{\epsilon}
\renewcommand{\d}{\delta}
\newcommand{\KRC}{C_{\mathrm{KR}}}
\newcommand{\KRH}{H_{\mathrm{KR}}}
\newcommand{\arm}{\mathrm{arm}}
\newcommand{\leg}{\mathrm{leg}}
\newcommand{\G}{\mathcal{G}}
\newcommand{\aA}{\mathsf{A}}
\newcommand{\EB}{\mathtt{EB}}
\newcommand{\AB}{\mathtt{AB}}

\newcommand{\Ch}{\operatorname{Ch}}

\newcommand{\newword}[1]{\emph{\textbf{#1}}}

\newlength\cellsize 

\savebox2{%
\begin{picture}(10,10)
\put(0,0){\line(1,0){10}}
\put(0,0){\line(0,1){10}}
\put(10,0){\line(0,1){10}}
\put(0,10){\line(1,0){10}}
\end{picture}}

\newcommand\cellify[1]{\def\thearg{#1}\def\nothing{}%
\ifx\thearg\nothing\vrule width0pt height\cellsize depth0pt%
  \else\hbox to 0pt{\usebox2\hss}\fi%
  \vbox to 10\unitlength{\vss\hbox to 10\unitlength{\hss$#1$\hss}\vss}}

\newcommand\tableau[1]{\vtop{\let\\=\cr
\setlength\baselineskip{-10000pt}
\setlength\lineskiplimit{10000pt}
\setlength\lineskip{0pt}
\halign{&\cellify{##}\cr#1\crcr}}}

\savebox7{
\begin{picture}(10,10)
  \put(5,5){\circle{9.4}}
\end{picture}}

\newcommand{\cirfy}[1]{\def\thearg{#1}\def\nothing{}%
\ifx\thearg\nothing\vrule width0pt height\cellsize depth0pt%
  \else\hbox to 0pt{\usebox7\hss}\fi%
  \vbox to 10\unitlength{\vss\hbox to 10\unitlength{\hss$#1$\hss}\vss}}

\newcommand\cirtab[1]{\vtop{\let\\=\cr
\setlength\baselineskip{-10000pt}
\setlength\lineskiplimit{10000pt}
\setlength\lineskip{0pt}
\halign{&\cirfy{##}\cr#1\crcr}}}

\usepackage{xcolor}
\usepackage[colorinlistoftodos]{todonotes}

\input xy
\usepackage[all]{xy}
\xyoption{line}
\xyoption{arrow}
\xyoption{color}
\SelectTips{cm}{}

\usepackage{tikz}
\usetikzlibrary{decorations.markings}
\usetikzlibrary{decorations.pathreplacing}

\tikzstyle directed=[postaction={decorate,decoration={markings,
    mark=at position #1 with {\arrow{>}}}}]
\tikzstyle rdirected=[postaction={decorate,decoration={markings,
    mark=at position #1 with {\arrow{<}}}}]

\usetikzlibrary{calc}
\usepackage{relsize}

\tikzset{fontscale/.style = {font=\relsize{#1}}
    }

\tikzset{anchorbase/.style={baseline={([yshift=-0.5ex]current bounding box.center)}},
    tinynodes/.style={font=\tiny,text height=0.75ex,text depth=0.15ex},
    smallnodes/.style={font=\scriptsize,text height=0.75ex,text depth=0.15ex},
    >={Latex[length=1mm, width=1.5mm]}
  }

\newcommand{\Harc}[4] 
{
\def\m{#1} 
\def\n{#2} 
\def\i{#3} 
\def\j{#4} 
\begin{scope}[shift={(\i,\j)}]
\draw[blue, thick]
(0,{(\m-\i+.5)/\m}) .. controls ++({1/\n},0) and ++({-1/\n},0) .. (1,{(\m-\i-.5)/\m});
\draw (0,0) rectangle (1,1);
\end{scope}
}

\newcommand{\Varc}[4] 
{
\def\m{#1} 
\def\n{#2} 
\def\i{#3} 
\def\j{#4} 
\begin{scope}[shift={(\i,\j)}]
\draw[blue, thick]
({(\n-\j-.5)/\n},1) .. controls ++(0,{-1/\m}) and ++(0,{1/\m}) .. ({(\n-\j+.5)/\n},0);
\draw (0,0) rectangle (1,1);
\end{scope}
}

\newcommand{\Rarc}[4] 
{
\def\m{#1} 
\def\n{#2} 
\def\i{#3} 
\def\j{#4} 
\begin{scope}[shift={(\i,\j)}]
\draw[blue, thick]
(0,{(\n-\i-.5)/\m}) .. controls ++({(\j-.5)/(\n+1)},0) and ++(0,{(\i+.5)/(\m+1)}) .. ({(\n-\j+.5)/\n},0);
\draw (0,0) rectangle (1,1);
\end{scope}
}

\newcommand{\Larc}[4] 
{
\def\m{#1} 
\def\n{#2} 
\def\i{#3} 
\def\j{#4} 
\begin{scope}[shift={(\i,\j)}]
\draw[blue, thick]
({(\n-\j-.5)/\n},1) .. controls ++(0,{(-\i-.5)/(\m+1)}) and ++({(-\j-.5)/(\n+1)},0) .. (1,{(\m-\i-.5)/\m});
\draw (0,0) rectangle (1,1);
\end{scope}
}

\newcommand{\Farc}[2] 
{
\def\m{#1} 
\def\n{#2} 
\begin{scope}
\draw[blue, thick]
(.5/\n,0) .. controls ++(0,.1) and ++(.1,0) .. (0,.5/\m);
\draw (0,0) rectangle (1,1);
\end{scope}
}

\newcommand{\Haarc}[4] 
{
\def\m{#1} 
\def\n{#2} 
\def\i{#3} 
\def\j{#4} 
\begin{scope}
\draw[white,line width=2mm]
(0,{(\m-\i+.5)/\m}) .. controls ++({1/\n},0) and ++({-1/\n},0) .. (1,{(\m-\i-.5)/\m});
\draw[blue, thick]
(0,{(\m-\i+.5)/\m}) .. controls ++({1/\n},0) and ++({-1/\n},0) .. (1,{(\m-\i-.5)/\m});
\draw (0,0) rectangle (1,1);
\end{scope}
}

\newcommand{\Vaarc}[4] 
{
\def\m{#1} 
\def\n{#2} 
\def\i{#3} 
\def\j{#4} 
\begin{scope}
\draw[white,line width=2mm]
({(\n-\j-.5)/\n},1) .. controls ++(0,{-1/\m}) and ++(0,{1/\m}) .. ({(\n-\j+.5)/\n},0);
\draw[blue, thick]
({(\n-\j-.5)/\n},1) .. controls ++(0,{-1/\m}) and ++(0,{1/\m}) .. ({(\n-\j+.5)/\n},0);
\draw (0,0) rectangle (1,1);
\end{scope}
}

\newcommand{\Raarc}[4] 
{
\def\m{#1} 
\def\n{#2} 
\def\i{#3} 
\def\j{#4} 
\begin{scope}
\draw[white,line width=2mm]
(0,{(\n-\i-.5)/\m}) .. controls ++({(\j-.5)/(\n+1)},0) and ++(0,{(\i+.5)/(\m+1)}) .. ({(\n-\j+.5)/\n},0);
\draw[blue, thick]
(0,{(\n-\i-.5)/\m}) .. controls ++({(\j-.5)/(\n+1)},0) and ++(0,{(\i+.5)/(\m+1)}) .. ({(\n-\j+.5)/\n},0);
\draw (0,0) rectangle (1,1);
\end{scope}
}

\newcommand{\Laarc}[4] 
{
\def\m{#1} 
\def\n{#2} 
\def\i{#3} 
\def\j{#4} 
\begin{scope}
\draw[white,line width=2mm]
({(\n-\j-.5)/\n},1) .. controls ++(0,{(-\i-.5)/(\m+1)}) and ++({(-\j-.5)/(\n+1)},0) ..  (1,{(\m-\i-.5)/\m});
\draw[blue, thick]
({(\n-\j-.5)/\n},1) .. controls ++(0,{(-\i-.5)/(\m+1)}) and ++({(-\j-.5)/(\n+1)},0) ..  (1,{(\m-\i-.5)/\m});
\draw (0,0) rectangle (1,1);
\end{scope}
}

\newcommand{\Faarc}[2] 
{
\def\m{#1} 
\def\n{#2} 
\begin{scope}
\draw[white,line width=2mm]
(.5/\n,0) .. controls ++(0,.1) and ++(.1,0) .. (0,.5/\m);
\draw[blue, thick]
(.5/\n,0) .. controls ++(0,.1) and ++(.1,0) .. (0,.5/\m);
\draw (0,0) rectangle (1,1);
\end{scope}
}

\newcommand{\dHarc}[4] 
{
\def\m{#1} 
\def\n{#2} 
\def\i{#3} 
\def\j{#4} 
\begin{scope}[shift={(\i,\j)}]
\draw[blue, thick]
(0,{(\m-\i+.5)/\m}) .. controls ++({1/\n},0) and ++({-1/\n},0) .. (1,{(\m-\i-.5)/\m});
\draw (0,0) rectangle (1,1);
\end{scope}
}

\newcommand{\dVarc}[4] 
{
\def\m{#1} 
\def\n{#2} 
\def\i{#3} 
\def\j{#4} 
\begin{scope}[shift={(\i,\j)}]
\draw[blue, thick]
({(\n-\j-1.5)/\n},1) .. controls ++(0,{-1/\m}) and ++(0,{1/\m}) .. ({(\n-\j-.5)/\n},0);
\draw (0,0) rectangle (1,1);
\end{scope}
}

\newcommand{\dRarc}[4] 
{
\def\m{#1} 
\def\n{#2} 
\def\i{#3} 
\def\j{#4} 
\begin{scope}[shift={(\i,\j)}]
\draw[blue, thick]
(0,{(\m-\i-.5)/\m}) .. controls ++({(\n-\j-.5)/(\n+1)},0) and ++(0,{(\m-\i-.5)/(\m+1)}) .. ({(\n-\j-.5)/\n},0);
\draw (0,0) rectangle (1,1);
\end{scope}
}

\newcommand{\dLarc}[4] 
{
\def\m{#1} 
\def\n{#2} 
\def\i{#3} 
\def\j{#4} 
\begin{scope}[shift={(\i,\j)}]
\draw[blue, thick]
({(\n-\j-1.5)/\n},1) .. controls ++(0,{(-\i-1.5)/(\m+1)}) and ++({(-\j-.5)/(\n+1)},0) .. (1,{(\m-\i-1.5)/\m});
\draw (0,0) rectangle (1,1);
\end{scope}
}

\newcommand{\dFarc}[2] 
{
\def\m{#1} 
\def\n{#2} 
\begin{scope}
\draw[blue, thick][shift={({\m-1},{\n-1})}]
({(\n-.5)/\n},1) .. controls ++(0,-.1) and ++(-.1,0) .. (1,{(\m-.5)/\m});
\draw (0,0) rectangle (1,1);
\end{scope}
}

\newcommand{\dHaarc}[4] 
{
\def\m{#1} 
\def\n{#2} 
\def\i{#3} 
\def\j{#4} 
\begin{scope}
\draw[white,line width=2mm]
(0,{(\m-\i+.5)/\m}) .. controls ++({1/\n},0) and ++({-1/\n},0) .. (1,{(\m-\i-.5)/\m});
\draw[blue, thick]
(0,{(\m-\i+.5)/\m}) .. controls ++({1/\n},0) and ++({-1/\n},0) .. (1,{(\m-\i-.5)/\m});
\draw (0,0) rectangle (1,1);
\end{scope}
}

\newcommand{\dVaarc}[4] 
{
\def\m{#1} 
\def\n{#2} 
\def\i{#3} 
\def\j{#4} 
\begin{scope}
\draw[white,line width=2mm]
({(\n-\j-1.5)/\n},1) .. controls ++(0,{-1/\m}) and ++(0,{1/\m}) .. ({(\n-\j-.5)/\n},0);
\draw[blue, thick]
({(\n-\j-1.5)/\n},1) .. controls ++(0,{-1/\m}) and ++(0,{1/\m}) .. ({(\n-\j-.5)/\n},0);
\draw (0,0) rectangle (1,1);
\end{scope}
}

\newcommand{\dRaarc}[4] 
{
\def\m{#1} 
\def\n{#2} 
\def\i{#3} 
\def\j{#4} 
\begin{scope}
\draw[white,line width=2mm]
(0,{(\m-\i-.5)/\m}) .. controls ++({(\n-\j-.5)/(\n+1)},0) and ++(0,{(\m-\i-.5)/(\m+1)}) .. ({(\n-\j-.5)/\n},0);
\draw[blue, thick]
(0,{(\m-\i-.5)/\m}) .. controls ++({(\n-\j-.5)/(\n+1)},0) and ++(0,{(\m-\i-.5)/(\m+1)}) .. ({(\n-\j-.5)/\n},0);
\draw (0,0) rectangle (1,1);
\end{scope}
}

\newcommand{\dLaarc}[4] 
{
\def\m{#1} 
\def\n{#2} 
\def\i{#3} 
\def\j{#4} 
\begin{scope}
\draw[white,line width=2mm]
({(\n-\j-1.5)/\n},1) .. controls ++(0,{(-\i-1.5)/(\m+1)}) and ++({(-\j-.5)/(\n+1)},0) .. (1,{(\m-\i-1.5)/\m});
\draw[blue, thick]
({(\n-\j-1.5)/\n},1) .. controls ++(0,{(-\i-1.5)/(\m+1)}) and ++({(-\j-.5)/(\n+1)},0) .. (1,{(\m-\i-1.5)/\m});
\draw (0,0) rectangle (1,1);
\end{scope}
}

\newcommand{\dFaarc}[2] 
{
\def\m{#1} 
\def\n{#2} 
\begin{scope}
\draw[white,line width=2mm]
({(\n-.5)/\n},1) .. controls ++(0,-.1) and ++(-.1,0) .. (1,{(\m-.5)/\m});
\draw[blue, thick]
({(\n-.5)/\n},1) .. controls ++(0,-.1) and ++(-.1,0) .. (1,{(\m-.5)/\m});
\draw (0,0) rectangle (1,1);
\end{scope}
}

\begin{document}


\title[Khovanov--Rozansky Homology of Coxeter Knots and Schr\"oder Polynomials for paths under any line]{Khovanov--Rozansky Homology of Coxeter Knots and Schr\"oder Polynomials for paths under any line}

\author[C. Caprau]{Carmen Caprau}
\address{Department of Mathematics, California State University, Fresno, CA 93740, USA}
\email{ccaprau@csufresno.edu}

\author[N. Gonz\'alez]{Nicolle Gonz\'alez}
\address{Department of Mathematics,  University of California Berkeley, CA 94720-3840, USA}
\email{nicolle@math.berkeley.edu}

\author[M. Hogancamp]{Matthew Hogancamp}
\address{Department of Mathematics, Northeastern University, Boston, MA 02115, USA}
\email{m.hogancamp@northeastern.edu}

\author[M. Mazin]{Mikhail Mazin}
\address{Department of Mathematics, Kansas State University, Manhattan, KS 66506, USA}
\email{mmazin@math.ksu.edu}





\begin{abstract}
We introduce a family of generalized Schr\"oder polynomials $S_\tau(q,t,a)$, indexed by triangular partitions $\tau$ and prove that $S_\tau(q,t,a)$ agrees with the Poincar\'e series of the triply graded Khovanov--Rozansky homology of the Coxeter knot $K_\tau$ associated to $\tau$.  For all integers $m,n,d\geq 1$ with $m,n$ relatively prime, the $(d,mnd+1)$-cable of the torus knot $T(m,n)$ appears as a special case.  It is known that these knots are algebraic, and as a result we obtain a proof of the $q=1$ specialization of the Oblomkov--Rasmussen--Shende conjecture for these knots.  Finally, we show that our Schr\"oder polynomial computes the hook components in the Schur expansion of the symmetric function appearing in the shuffle theorem under any line, thus proving a triangular version of the $(q,t)$-Schr\"oder theorem.
\end{abstract}

\maketitle
\setcounter{secnumdepth}{4}
\setcounter{tocdepth}{1}
\tableofcontents

\section{Introduction}\label{s:intro}

Triply graded Khovanov--Rozansky homology is a rich invariant of knots and links, which categorifies the HOMFLY-PT polynomial. During the last fifteen years a myriad of conjectures have appeared relating the Khovanov--Rozansky homology of various special families of links to Catalan combinatorics \cite{GorksyCatalan, GL20, GL23}, the Hilbert schemes of points on a plane curve singularity \cite{ORS}, representations of rational DAHA \cite{GORS, cherednik, CD16, CD17}, sheaves on the Hilbert scheme of points in the plane \cite{GorskyNegut,GNR21}, and refined Chern-Simmons theory \cite{aganagic} (see  \S \ref{ss:intro ORS} for brief summary of these conjectures). 
In order to employ a \emph{compute-both-sides} approach to such conjectures, one must first overcome the obvious challenge of computing Khovanov--Rozansky homology of the links or knots in question.

In \cite{EH16}, the third named author and Ben Elias introduced a technique for computing Khovanov--Rozansky homology relying on a recursively defined polynomial $R_{\bf u,v}$ for certain binary sequences $\bf u$ and $\bf v$, with the torus link $T(n,n)$ being the main example.  This technique was then exploited in subsequent papers by Elias, Hogancamp, and Mellit to compute the homologies of various families of links: the paper \cite{Ho17} computes the homologies of $T(m,md)$ and $T(m,md\pm1)$ for $m,d\geq 1$, followed shortly thereafter by \cite{Mellit-Homology} which computes the homology of $T(m,n)$ for $m,n\geq 1$  coprime, followed by \cite{HM}, generalizing all the aforementioned, which computes the homology of $T(m,n)$ for $m,n\geq 1$, not necessarily coprime. Importantly for us, the aforementioned computations include some non-torus-links as well.

The initial motivation for this paper stemmed from the problem of computing and giving a combinatorial interpretation of the Khovanov--Rozansky homology of iterated torus knots of the form $T(m,n)(d,mnd+1)$, where $m,n,d\in \Z_{\geq 1}$ and $m,n$ are coprime. Along the way, we realized such knots are special cases of the knots whose homology was computed in \cite{HM}, which in turn coincide with a certain subfamily of \emph{Coxeter knots} $K_\tau$, where $\tau$ is a \emph{triangular partition}. {Coxeter knots} were previously studied in \cite{OR, GL23} and triangular partitions in \cite{BM}. 

In this paper we investigate the Khovanov--Rozanksy homology of such knots. We introduce the new family of \emph{triangular Schr\"oder polynomials} $S_\tau(q,t,a)$ which we prove compute the Khovanov--Rozanksy homology of Coxeter knots $K_\tau$ for triangular partitions. As a consequence of the knots $T(m,n)(d,mnd+1)$ with $m,n$ coprime being algebraic, we obtain a proof of a conjecture of Oblomkov--Rasmussen--Shende restricted to these knots in the special case when $q=1$. Combinatorially, our 
triangular Schr\"oder polynomials fit nicely into the existing $(q,t)$-Catalan combinatorics \cite{Garsia-Haiman,BG-scifi, Haiman, Haiman-Vanishing, Ha04, HHLRU, HaglundCatalan, GorksyCatalan,  bergeron2016compositional, Mellit-Shuffle, CM18, BHMPS}. They generalize both the classical $(q,t)$-Catalan \cite{Garsia-Haiman} and Schr\"oder polynomials \cite{EHKK03} and correspond to the hook components of the Schur expansion of the shuffle theorem under any line of Blasiak-Haiman-Morse-Pun-Seelinger \cite{BHMPS}, thus providing a triangular generalization of the classical \emph{$(q,t)$-Schr\"oder Theorem} of Haglund \cite{Ha04}.

\subsection{Triangular Schr\"oder Polynomials and Invariant Subsets}
\label{ss:intro schroeder}
Deeply ingrained in the many conjectures surrounding Khovanov--Rozansky homology is the combinatorics of Dyck paths and related $(q,t)$-Catalan objects.  The $(q,t)$-Catalan polynomials were introduced in \cite{Garsia-Haiman} and studied extensively in Haglund's book \cite{HaglundCatalan}. In order to remain consistent with the notation used throughout this paper, we denote the $(q,t)$-Catalan by $C_{n,n+1}(q,t)$, where $n\in \Z_{\geq 1}$.

These polynomials appear in a wide variety of contexts.  First, $C_{n,n+1}(q,t)$ admits a combinatorial definition as a generating function which counts Dyck paths in the $n\times (n+1)$ rectangle (which are certain lattice paths below the line $(n+1)y+nx=n(n+1)$ consisting of north and west steps), weighted by some explicit combinatorially defined statistics $\area$ and $\dinv$. In representation theory, $C_{n,n+1}(q,t)$ appears as the coefficient of the Schur function $s_{(1^n)}$ in the Schur expansion of $\nabla e_n$, (which through work of Haiman \cite{Haiman, Haiman-Vanishing}, coincides with the Frobenius character of the ring of a diagonal coinvariants). Here, $e_n$ is the $n^{\text{th}}$ elementary symmetric function and $\nabla$ is the Bergeron-Garsia  operator on symmetric functions \cite{BG-scifi}, which acts diagonally on modified Macdonald polynomials. In topology, $C_{n,n+1}(q,t)$ also appears as the graded dimension of the lowest $a$-degree component of the Khovanov--Rozansky homology of the $(n,n+1)$-torus knot (this was conjectured by Gorsky in \cite{GorksyCatalan}, and proven in \cite{Ho17}).

We now point out two directions in which $(q,t)$-Catalan polynomials have been generalized. In the first, one replaces the Dyck paths appearing in the definition of $C_{n,n+1}(q,t)$ with the more general \emph{Schr\"oder paths} \cite{S1870}. These are certain lattice paths which are allowed to contain some number of northwest steps in addition to the usual north or west steps of a Dyck path. The resulting generating function $S_{n,n+1}(q,t,a)$ is called the \emph{Schr\"oder polynomial} \cite{EHKK03}, defined so that the coefficient of $a^k$ in $S_{n,n+1}(q,t,a)$ is a generating function counting Schr\"oder paths in the $n\times (n+1)$ rectangle with exactly $k$ diagonal steps. At the level of symmetric functions, this amounts to extracting the coefficient of hook-indexed Schur function $s_{(k,1^{n-k})}$ (for a precise statement see \S \ref{ss: intro schro thm}); at the level of link homology, this amounts to taking the graded dimension of the higher $a$-degree components of Khovanov--Rozansky homology of $T(n,n+1)$.

The second direction in which we can generalize $C_{n,n+1}(q,t)$ is to allow more Dyck paths in rectangles of arbitrary sizes. For instance, for a relatively prime pair of integers $m,n\geq 1$, one has the associated \emph{rational $(q,t)$-Catalan} $C_{m,n}(q,t)$ which counts Dyck paths in an $m\times n$ rectangle, weighted by an analogous pair of statistics $\area$, $\dinv$.  As in the classical case, $C_{m,n}(q,t)$ is indeed the sign component of an associated symmetric function, in this case defined by Hikita \cite{Hi14}.  At the level of link homology, $C_{m,n}(q,t)$ is the graded dimension of the lowest $a$-degree component of the Khovanov--Rozansky homology of $T(m,n)$ (this was conjectured in \cite{cherednik,GorskyNegut}, and proven by Mellit in \cite{Mellit-Homology}).  Of course, one can generalize in both directions, obtaining \emph{rational Schr\"oder polynomial} $S_{m,n}(q,t,a)$ which calculates the hook components of Hikita's symmetric function and also the graded dimension of the Khovanov--Rozansky homology of $T(m,n)$ in all $a$-degrees.

In \cite{BHMPS} Blasiak--Haiman--Morse--Pun--Seelinger provide a further generalization in the aforementioned second direction by considering set of Dyck paths ``under any line''. More precisely given arbitrary real numbers $r,s>0$, a \newword{Dyck path under $L_{r,s}$} is a path in $\Z_{\geq 0}\times \Z_{\geq 0}$ consisting of unit length west and north steps, from $(\lfloor r\rfloor, 0)$ to $(0,\lfloor s\rfloor)$, and staying weakly below $L_{r,s}$.  In \emph{loc.~cit.} it is shown how to construct an analogue of the Hikita symmetric function from the set of Dyck paths under $L_{r,s}$, and a combinatorial interpretation of the associated $(q,t)$-Catalan is given. From this perspective, the present paper provides the missing Schr\"oder polynomial and link homology interpretation for paths ``under any line''. In order to delve further into what this means, we must first introduce the notion of triangular partitions.

Given a partition $\lambda = (\lambda_1, \dots, \lambda_k)$ with $\lambda_1\geq\dots \geq \lambda_k \geq 0$, we identify $\lambda$ with its partition diagram in $\R_{\geq 0} \times \R_{\geq 0}$ and index any cell $\Box \in \lambda$ by its northeast coordinate $(x,y) \in \Z_{>0} \times \Z_{>0}$. Thus a partition can be viewed as a subset $\lambda \subset \Z_{>0} \times \Z_{>0}$. 

Any pair of positive real numbers $r,s>0$ determines a partition $\tau_{r,s}$ by declaring that the cell $\Box^{x,y}$ with northeast corner at $(x,y)\in \Z_{>0}\times\Z_{>0}$ is in $\tau_{r,s}$ if and only if $ry+sx\leq rs$.  In other words, $\tau_{r,s}$ consists of all boxes contained between the line $L_{r,s}:=\{ry+sx=rs\}$ and the coordinate axes. A \newword{triangular partition} is any partition of the form $\tau_{r,s}$ for some real numbers $r,s>0$.  The combinatorics of triangular partitions were studied by Bergeron and the fourth named author in \cite{BM}, and are the underlying indexing object in the generalized shuffle theorem of Blasiak--Haiman--Morse--Pun--Seelinger \cite{BHMPS}. 

Observe that if $\pi$ is a Dyck path under the line $L_{r,s}$ then there is an associated partition $\lambda$ consisting of all cells in the first quadrant and below $\pi$.  By definition, $\lambda$ will be a subpartition of $\tau_{r,s}$.  This correspondence is a canonical bijection between Dyck paths under $L_{r,s}$ and subpartitions of $\tau_{r,s}$.  In light of this observation, a pair $(\lambda,\tau)$ consisting of a triangular partition $\tau$ and a subpartition $\lambda\subset \tau$ will be referred to as a \newword{$\tau$-Dyck path}.   As matter of notation, given real numbers $r,s>0$ we will write $\D(r,s)$ for the set of all $\tau_{r,s}$-Dyck paths.

Our definition of the Schr\"oder polynomial associated to a triangular partition will utilize the recursions introduced by Gorsky-Mazin-Vazirani~\cite{GMV17, GMV20} and thus will be phrased in the language of invariant subsets rather than lattice paths, which we now briefly recall.

For coprime positive integers $m,n$, a set $\Delta \subset \Z_{\geq 0}$ is \newword{0-normalized and $(m,n)$-invariant} if it contains $0$ and is closed under translations by $+m$ and $+n$, i.e. $\Delta+m, \Delta+n \subset \Delta$. Denote by $\Inv_{m,n}^0$ the set of all such subsets. In \cite{GMV17, GMV20}, the fourth named author alongside Gorsky and Vazirani defined a bijection
\[
\mathcal{A}: \D(m,n) \to \Inv_{m,n}^0 \qquad \text{sending} \qquad \pi=(\lambda, \tau_{m,n}) \mapsto \Z_{\geq 0} \setminus \Gaps(\pi),
\]
where $\Gaps(\pi):=\{ mn-my-nx \; | \; (x,y) \in \tau_{m,n} \setminus \lambda \}$. 
More generally, Gorsky--Mazin--Vazirani also considered invariant subsets indexed by certain trinary sequences $\bf x,y$ in letters $\{0,1,\bullet\}$ and defined the polynomial
\[
Q_{{\bf x},{\bf y}}(q,t,a):=\sum_{\Delta\in \Inv_{{\bf x},{\bf y}}} q^{\area'(\Delta)}t^{-\codinv'(\Delta)}\prod_{k\in\cogennn(\Delta)}(1+at^{-\xi_k(\Delta)}),
\]
where $\area, \codinv$, and $\cogen$ are certain combinatorial statistics. 
In particular, they proved that $Q_{\bf x,y}$ satisfies recursion relations that coincide with those of Hogancamp-Mellit \cite{HM}, so that if $\bf u,v$ are the binary sequences obtained from $\bf x,y$ by omitting the bullets $\bullet$, then $Q_{\bf x,y} = R_{\bf u,v}$ \cite{GMV17}. As a consequence, they showed that the $(q,t)$-Schr\"oder (and consequently Catalan) polynomial $S_{m,n}(q,t,a)$ could be obtained directly as a certain $q,t,a$-graded sum over $\Inv_{m,n}^0$.

In order to extend their construction to the setting of $\D(r,s)$ with $r,s$ positive real numbers, we choose an embedding of $\tau_{r,s}$ into a larger triangular partition $\tau_{m,n}$, such that $m,n$ are coprime positive integers with $L_{m,n}$ parallel to $L_{r,s}$. Indeed, such a choice exists since any given triangular partition $\tau$ may be cut out by lines of various slopes \cite{BM}. In particular, if $\tau_{r,s} \subset \tau_{m,n}$, then there exists $0 < z \leq mn$ such that $(r,s) = (z/n, z/m)$, for which if $\ell = \lfloor z \rfloor$ then $\tau_{r,s} = \tau_{\ell/n,\ell/m}$ and $\D(r,s) = \D(\ell/n,\ell/m)$. Since the choice of $\tau_{m,n}$ plays a role in the definitions, we modify the notation slightly in order to make this choice explicit and set, 
\[ \tau_{m,n,\ell}:= \tau_{\ell/n, \ell/m}, \qquad  \qquad \D(m,n,\ell):= \D(\ell/n,\ell/m), \qquad \text{and} \qquad L_{m,n,\ell}:= L_{\ell/n,\ell/m}.
\]
So then, for any such triple $(m,n,\ell)$ we define 
\[ 
\Inv_{m,n,\ell}^0:= \{ \Delta \in \Inv_{m,n}^0 \; |  \; \Delta \cap \Gaps(\tau_{m,n,l},\tau_{m,n}) = \emptyset \},
\]
and prove that 
\begin{equation}\label{eqn:Inv-intro}
\Inv_{{\bf x}(m,n,\ell), {\bf y}(m,n,\ell)}  = (\Inv_{m,n,\ell}^0 -(mn-\ell) +(n+m)) \cap \Z_{\geq 0},
\end{equation}
where ${\bf x}(m,n,\ell), {\bf y}(m,n,\ell)$ are certain trinary sequences constructed directly from the combinatorial data of $L_{m,n,\ell}$.

The \newword{triangular $(q,t)$-Schr\"oder polynomial} associated to the triple $(m,n,\ell)$ is defined by
\begin{equation*}
S_{m,n,\ell}(q,t,a):=t^{|\tau_{m,n,\ell}|}Q_{{\bf x}(m,n,\ell),{\bf y}(m,n,\ell)}(q,t,a).
\end{equation*}

As before, we have a bijection 
$\AD :\D(m,n,\ell) \to \Inv_{{\bf w}(m,n,\ell)}$, allowing us to give a formulation of these polynomials in terms of Schr\"oder paths (see Theorem \ref{thm:Schroder-dyckpaths}).

\begin{introthm}\label{introthm:schroeder}
We have
\begin{equation}\label{eq:intro schroeder 2}
    S_{m,n,\ell}(q,t,a)
    = \sum_{k\geq 0} a^k \left(\sum_{\pi\in\D(m,n,\ell)}
q^{\area(\pi)}t^{\dinv(\pi)} \sum_{\substack{
L\subset\AB(\pi)\\ |L|=k}}
t^{-\sum_{\Box\in L} \xi(\pi,\Box)}\right),
\end{equation}
where $\AB(\pi)$, with $\pi = (\lambda, \tau_{m,n,\ell})$, is the set of all addable boxes of $\lambda$ and $\area$, $\dinv$, and $\xi$ are some explicit statistics.
\end{introthm}
In the language of lattice paths, an addable box of $\lambda$ corresponds to a west step (W) immediately followed by a north step (N) (reading the path right to left and bottom to top). Thus, we can form a Schr\"oder path by replacing each such pair (W,N) in a given subset  $L\subset \AB(\pi)$ of addable boxes with a single northwest pointing diagonal step. In particular, for each fixed $k$, the sum over pairs $(\pi, L)$ can be interpreted as a sum over lattice paths $\pi \in \D(m,n,\ell)$ with exactly $k$ diagonal steps. Consequently, we recover the more familiar rational $(q,t)$-Schr\"oder polynomial $S_{m,n}(q,t,a)$ in the special case $\ell=mn-1$.

The following result, exhibited in Corollary \ref{cor:Catalan}, is now immediate (see \S\ref{sec:Schroder} for details).
\begin{introthm} \label{intro2}
The triangular $(q,t)$-Schr\"oder polynomial $S_{m,n,\ell}(q,t,a)$ specializes to the triangular $(q,t)$-Catalan polynomial of \cite{BM,BHMPS} at $a=0$, 
\begin{equation*}
   S_{m,n,\ell}(q,t,0)=\sum_{\pi\in\D(m,n,\ell)} q^{\area(\pi)}t^{\dinv(\pi)}=:C_{m,n,\ell}(q,t).
\end{equation*}
\end{introthm}

\subsection{The Coxeter Knot Associated to a Partition} 
\label{ss:intro knot families}
The family of (positive) Coxeter knots was introduced in the work of Oblomkov and Rozansky \cite{OR} and was recently proved to coincide with the family of monotone knots independently defined by Galashin--Lam \cite{GL23}. Coxeter knots featured prominently in connections between link homology and the geometry of the Hilbert scheme of points on $\C^2$ (see also \cite{GNR21}). We recall their construction next. 

 Given an integer $M\geq 1$, let $\Br_M$ denote the braid group on $M$ strands with $\sigma_1,\ldots,\sigma_{M-1}$ the usual Artin generators.
Let $\lambda$ be a partition with $\lambda_1<M$, and let $\mu_1\geq \mu_2\geq \cdots \geq \mu_M=0$ be the lengths of the columns of $\lambda$, i.e.~$\mu$ is the transpose partition of $\lambda$, extending by zeros to obtain a sequence of length $M$.  We define the \newword{Coxeter knot associated to $\lambda$}, denoted $K_\lambda$,  to be the closure of the $M$-strand braid
\begin{equation*}
\beta^{cox}_{M,\lambda} := \JM_{M,2}^{\mu_1-\mu_2}\JM_{M,3}^{\mu_2-\mu_3}\dots \JM_{M,M}^{\mu_{M-1}-\mu_M} \sigma_1\sigma_2\cdots\sigma_{M-1},
\end{equation*}
where $\JM_{M,i}\in \Br_{M}$ is the $M$-strand \newword{Jucys-Murphy} braid
\[
\JM_{M,i}=(\sigma_{M-i+1}\sigma_{M-i+2}\cdots \sigma_{M-1})(\sigma_{M-1}\cdots\sigma_{M-i+2}\sigma_{M-i+1})
\]
(these Jucys-Murphy braids are obtained from the more standard ones $(\sigma_{i-1}\cdots \sigma_2\sigma_1)(\sigma_1\sigma_2\cdots\sigma_{i-1})$ by conjugating with the half-twist).
It is straightforward to prove (see Proposition \ref{prop:Ktau}) that the closure of $\beta^{cox}_{M,\lambda}$ depends only on $\lambda$ (and not on $M$) up to isotopy.

We are particularly interested in knots of the form $K_\tau$ where $\tau$ is a triangular partition. A few important examples include the following:
\begin{itemize}
\item[(1)]
If $\tau = \tau_{m,n}$ with $m,n\geq 1$ coprime integers, then $K_\tau=T(n,m)$.  This is most easily seen by translating into the language of monotone knots using a result of Galashin--Lam \cite{GL23}.  

\item[(2)]
If $\tau=\tau_{M,N}$ with $M,N\geq 1$ arbitrary positive integers (not necessarily coprime), then $K_\tau$ is a certain cabled torus knot.  Specifically let $d=\gcd(M,N)$ and write $(M,N)=(md,nd)$, so that $m,n\geq 1$ are coprime.  Then $K_\tau=T(m,n)(d,mnd+1)$, where $T(m,n)(d,mnd+1)$ is the $(d,mnd+1)$-cable of the $(m,n)$-torus knot, $T(m, n)$.  This fact was observed in Galashin--Lam \cite{GL20}, and we give an alternate proof in Appendix \ref{sec:AppendixB}.
\end{itemize}

To compute the Khovanov--Rozansky homology of $K_\tau$, we will recognize that the family of knots $K_\tau$ (in which $\tau$ is a triangular partition) coincides with the family of knots whose homology is already computed by Hogancamp and Mellit in \cite{HM}.  Let us recall this latter family now.

Given $\mathbf{u}\in \{0,1\}^m$ and $\mathbf{v}\in \{0,1\}^n$ binary sequences with exactly one `1', we write $\mathbf{u}=(0^i10^j)$ and $\mathbf{v}=(0^k10^l)$.  Denote by $K_{\mathbf{u},\mathbf{v}}$ the closure of the braid
\[
(\sigma_1\cdots\sigma_{i})(\sigma_1\sigma_2\cdots \sigma_{m-1})^k(\sigma_2\cdots \sigma_{m-1})^l.
\]
For an example diagram of such a link, see Figure \ref{fig:binary knot}.  The main result of \cite{HM} is a computation of the Khovanov--Rozansky homology of $K_{\mathbf{u},\mathbf{v}}$, via a recursion involving general pairs of binary sequences with the same number of occurences of `1'.  Specifically, \cite{HM} defines series $R_{\mathbf{u},\mathbf{v}}(q,t,a)$ indexed by pairs of binary sequences $\mathbf{u}\in \{0,1\}^m$, $\mathbf{v}\in\{0,1\}^n$ with $\sum_i u_i=\sum_jv_j$, and proves that 
\begin{equation}\label{eq:intro HM computation}
(1-q)(at^{-1/2}q^{-1/2})^{-\d(\mathbf{u},\mathbf{v})} P_{K_{\mathbf{u},\mathbf{v}}}^{\KR}(q,t,a) = t^{\d(\mathbf{u},\mathbf{v})} R_{\mathbf{u},\mathbf{v}}(q,t,a),
\end{equation}
where $\d(0^i10^j,0^k10^l)=\frac{1}{2}((i+j)(k+l)-j-l+c-1)$, where $c$ is the number of components of $K_{0^i10^j,0^k10^l}$ (see \S \ref{ss:HM recursion}). Note that $K_{\mathbf{u},\mathbf{v}}$ may be a link of more than one component.  For the case of $K_{\mathbf{u},\mathbf{v}}$ a knot, we prove in \S \ref{sec:trian knot} the following (see Lemma \ref{lem:K-L isotopic} and Theorem \ref{thm: schroder=poincare}):

\begin{introthm}
\label{thm:intro3}
For any triangular partition $\tau$, the knot $K_\tau$ is isotopic to a knot of the form $K_{\mathbf{u},\mathbf{v}}$.  Conversely, if $K_{\mathbf{u},\mathbf{v}}$ is a knot (as opposed to a link with more than one component), then $K_{\mathbf{u},\mathbf{v}}=K_\tau$ for some triangular partition, and moreover $\d(\mathbf{u},\mathbf{v})=|\tau|$.

Furthermore, if $(m,n,\ell)$ is a triple of integers such that $\tau = \tau_{m,n,\ell}$, then the Khovanov--Rozansky homology of $K_\tau$ is free over $\Z$ with Poincar\'e series given by the triangular $(q,t)$-Schr\"oder polynomial up to normalization.  Specifically:
\[
S_{m,n,\ell}(q,t,a)= (1-q)(at^{-1/2}q^{-1/2})^{-|\tau|} P_{K_{\tau}}^{\KR}(q,t,a).\]
\end{introthm}
In the above theorem (and throughout the rest of the paper), we write $P_K^{\KR}(q,t,a)$ for the Poincar\'e series of the triply graded Khovanov--Rozansky homology of $K$, and for brevity refer to this as the \newword{KR series} of $K$.

The mapping from triangular partitions to pairs of binary sequences is many-valued, and requires us to choose an expression of $\tau=\tau_{m,n,\ell}$, similar to the Schr\"oder polynomials themselves (see \S \ref{ss:invariance thm}).  Nonetheless, the knot $K_\tau$ really only depends on $\tau$ up to isotopy, and well-definedness of Khovanov--Rozansky homology establishes the following.

\begin{corollary}
The triangular Schr\"oder polynomial $S_{m,n,\ell}(q,t,a)$ depends only on the triangular partition $\tau=\tau_{m,n,\ell}$.
\end{corollary}

Below we will state some of the most important special cases of Theorem \ref{thm:intro3}.  For combinatorial purposes, it is best to state the result directly in terms of the series $R_{\mathbf{u},\mathbf{v}}$, using \eqref{eq:intro HM computation}.

\begin{itemize}
\item[(1)] The torus knot $T(m,n)$ is $K_\tau$ where $\tau=\tau_{m,n}$ (with $\ell$ taken to be the maximum value available, $\ell=mn-1$).  The associated pair of binary sequences is $(\mathbf{u},\mathbf{v})=(0^{m-1}1,0^{n-1}1)$, and we recover the known identity
\[
S_{m,n}(q,t,a) = t^{\d(m,n)}R_{0^{m-1}1,0^{n-1}1}(q,t,a),
\]
where $S_{m,n}(q,t,a)=S_{m,n,mn-1}(q,t,a)$ is the usual rational Schr\"oder polynomial and $\d(m,n)$ is as in Definition \ref{def:delta MN}.

\item[(2)] The cabled torus knot $T(m,n)(d,mnd+1)$ is $K_\tau$, where $\tau=\tau_{md,nd}$ (again, $\ell$ is taken to be its maximum value $\ell=(md)(nd)-1$).  The associated pair of binary sequences is $(\mathbf{u},\mathbf{v})=(0^{d-1}1 0^{(m-1)d},0^{nd}1)$
Given integers $m,n,d\geq 1$ with $m,n$ coprime, the $(d,mnd+1)$-cable of $T(m,n)$ is the closure of the braid $\b(,0^{nd}1)\in \Br_{md}$, thus 
\[
S_{md,nd}(q,t,a) =  t^{\d(md,nd)}R_{0^{d-1}10^{(m-1)d},0^{nd}1}(q,t,a).
\]
\end{itemize}

\subsection{The Oblomkov-Rasmussen-Shende Conjecture}
\label{ss:intro ORS}
Let $Z\subset \C^2$ be a complex plane curve with a singularity at the origin $0\in \C^2$, and let $K$ denote the \newword{link of the singularity} $0\in Z$ (i.e.~$K$ is the intersection of $Z$ with the 3-sphere of radius $\e$ around 0, for sufficiently small $\e>0$.  A link is \emph{algebraic} if it arises in this way. Algebraic links have been widely studied for past century \cite{B28,K29,EN85}, with their associated plane curve singularities particularly receiving a lot of attention in recent years \cite{OS, MS13, MY14, ORS, GNR21, GMO, OR}.

The \newword{Oblomkov-Rasmussen-Shende conjecture} asserts that the lowest $a$-degree part of  the Khovanov--Rozansky homology of an algebraic link $K$ is isomorphic to the cohomology of the Hilbert scheme of points on the plane curve singularity whose link is $K$ (with the higher $a$-degree components obtained by considering nested Hilbert schemes).  We may refer to this conjecture as the ORS1 conjecture, referring to \cite{ORS} for the precise statement.

When $K$ is a knot, the associated plane curve singularity is called a \newword{unibranch singularity}; in this case ORS propose a second conjecture, which we may refer to as ORS2, which asserts that the lowest $a$-degree part of Khovanov--Rozansky homology of $K$ is isomorphic to the cohomology of the \emph{compactified Jacobian} of the singularity $0\in Z$, with its perverse filtration (the higher $a$-degree components are excluded from this conjecture; they were incorporated in a recent preprint by Oscar Kivinen and Minh-T\^am Quang Trinh \cite{KT24}).  Work of Migliorini and Shende \cite{MS13}, and independently Maulik and Yun \cite{MY14}, establishes the equivalence of ORS1 and ORS2 (for knots, and only the lowest $a$-degree part).

Ignoring the perverse filtration corresponds to setting $q=1$ on the level of Poincar\'e series.  Specialized at $q=1$, the ORS2 conjecture can be stated at the level of Poincar\'e series as
\[
\Big((1-q)\PC^{\KR}_K(q,t,a)\Big)|_{\substack{q=1\\ t=u^{-2}}} = a^\d u^{-\d} \PC_{\overline{JC}}(u) + (\text{terms with $a$-degree $>\d$}),
\]
where $\d$ is the Milnor number of the plane curve singularity whose link is $K$, and $\overline{JC}$ is its compactified Jacobian.

In a recent article \cite{GMO}, the fourth named author jointly with Gorsky and Oblomkov proved that the compactified Jacobian of (the singularity at the origin of) the plane curve $(x(t),y(t))=  (t^{nd}, t^{md}+\lambda t^{md+1}+ \dots)$ has an affine paving with cells indexed by $\tau_{md,nd}$-Dyck paths, with Poincar\'e series computed by a specialization of the triangular $(q,t)$-Catalan polynomial $C_{\tau_{md,nd}}(q,t)$. Combined with the realization that the link of the singularity in this case is $T(m,n)(d,mnd+1)$, we obtain a proof of the ORS2 conjecture at $q=1$ for these knots.

\begin{introthm}\label{intro4}
The second Oblomkov-Rasmussen-Shende conjecture at $q=1$ holds for the $(d,mnd+1)$-cable of the $(m,n)$-torus knot; that is, for $\tau=\tau_{md,nd}$ we have:
\begin{align*}
\Big((1-q)\PC^{\KR}_{K_\tau}(q,t,a)\Big)|_{\substack{q=1\\ t=u^{-2}}} &=  a^\d u^{-\d} u^{2\delta}S_\tau(q,t,0) + (\text{terms with $a$-degree $>\d$})\\
&=a^\d u^{-\d} \PC_{\overline{JC}}(u) + (\text{terms with $a$-degree $>\d$}).
\end{align*}
\end{introthm}
This theorem is restated and proven as Theorem \ref{thm: ORS for gen curves} in \S \ref{subsec:ORS}.

\subsection{The $(q,t)$-Schr\"oder Theorem}\label{ss: intro schro thm}
The classical \newword{shuffle theorem}, conjectured by Haglund-Haiman-Loehr-Remmel-Ulyanov \cite{HHLRU} and proven by Carlsson-Mellit \cite{CM18}, gives a combinatorial formula in terms of $\tau_{n,n+1}$-Dyck paths for $\nabla e_n$.  The \newword{$(q,t)$-Schr\"oder Theorem} was conjectured by Egge, Haglund, Killpatrick and Kremer \cite{EHKK03} and proven by Haglund \cite{Ha04} over ten years before Carlsson and Mellit's proof of the shuffle theorem. It states that the hook components of the Schur expansion of $\nabla e_n$ coincide with the classical Schr\"oder polynomials $S_{n,n+1}(q,t,a)$, i.e. the following equality holds: 
\[
S_{n,n+1}(q,t,a)= \sum_{k\geq 0} \langle \nabla e_n, h_ke_{n-k} \rangle \; a^k. 
\]

Motivated by their study of certain superpolynomials in the double affine Hecke algebra and their connection to Khovanov--Rozansky homology of torus knots, Gorsky and Negut \cite{GorskyNegut} conjectured a rational generalization of the shuffle theorem (which Mellit \cite{Mellit-Shuffle} proved), wherein $\nabla e_n$ is replaced by the action of certain operator $P_{m,n}$ in the \emph{elliptic Hall algebra} on the element $1$ in the ring of symmetric functions and the combinatorial side now ranges over $\tau_{m,n}$-Dyck paths with $m,n$ relatively prime integers. The action of the elliptic Hall algebra $\mathcal{E}$ on symmetric functions and its connection to Macdonald and $(q,t)$-Catalan combinatorics is well documented with both algebraic and geometric formulations studied by Feigin-Tsymbaliuk \cite{FT} and Burban, Schiffman, and Vasserot \cite{BS, schiffman, SV11, SV12, SV13}. Inspired by this construction, Blasiak-Haiman-Morse-Pun-Seelinger \cite{BHMPS} proved that a similar equality could be given by considering certain operators $\mathsf{D}_{\tau_{r,s}} \in \mathcal{E}$ indexed by triangular partitions $\tau_{r,s}$, with $r,s$ arbitrary positive real numbers:
\begin{equation*}
\mathsf{D}_{\tau_{r,s}}(1) = \sum_{\pi \in \D(r,s)} q^{\area(\pi)} t^{\dinv(\pi)} \omega(\G_{\nu(\pi)}(X;t^{-1})),
\end{equation*}
where $\omega$ is the standard Weyl involution on symmetric functions and $\G_{\nu(\pi)}(X;t^{-1})$ is the LLT-polynomial for some particular tuple of skew shapes $\nu(\pi)$. Although this so-called \newword{shuffle theorem under any line} fully generalized the combinatorial side of the previous shuffle theorems, its connections to Khovanov--Rozansky homology and torus links remained completely open. In this article, we make considerable progress towards bridging this gap by showing that the hook components of the Schur expansion of $\mathsf{D}_{\tau}(1)$ are exactly given by our triangular $(q,t)$-Schr\"oder polynomial $S_\tau(q,t,a)$ and thus the KR series of the Coxeter knot $K_\tau$. Using a superization procedure for LLT-polynomials, we prove:
\begin{introthm}\label{intro5}
The triangular $(q,t)$-Schr\"oder theorem holds:
\[
S_{\tau_{r,s}}(q,t,a)= \sum_{k\geq 0} \langle \Hik_{\tau_{r,s}}(X;q,t), h_ke_{\lfloor s \rfloor-k} \rangle \;a^k. 
\]
\end{introthm}
This result is stated and proven as Corollary \ref{cor:Schroder-Shuffle}. Note that $\Hik_\tau(X;q,t)$ is the symmetric function obtained by the action of particular elliptic Hall algebra elements indexed by a triangular partition $\tau$ on $1$. 

\subsection{Further Directions} There are many open questions that remain to be studied. Although we were able to compute the KR series for all Coxeter \emph{knots}, ideally one would like such a result for all Coxeter \emph{links} and general monotone curves. It is our hope that an extension of our work might yield light in this direction. 

Combinatorially there are also several directions worth pursuing. In a recent article, the second named author alongside Simental and Vazirani \cite{GSV} gave a `finite' analog of the rational shuffle theorem in which the elliptic Hall algebra is replaced the double affine Hecke algebra (DAHA). Conjecturally, their construction coincides with the bigraded $GL_r$-character of certain quantized Gieseker variety representations. On the other hand, Galashin and Lam's superpolynomial for monotone links is built from the DAHA. It it unknown for exactly which monotone links the superpolynomial agrees with the KR series. Thus, one can ask whether a generalization to the triangular setting of the finite shuffle theorem of Gonz\'alez-Simental-Vazirani could yield light on a geometric interpretation of our triangular Schr\"oder polynomials, which combined with the results in this paper might provide a partial answer to when the KR series and the superpolynomial coincide.

\subsection{Structure of the paper:} In \S\ref{sec:KR and HM}, we recall the necessary topological notions of knots and cables, the construction of Khovanov--Rozansky homology and binary recursion of Hogancamp and Mellit. \S\ref{sec:GMV-Recursions} is similarly laying the combinatorial background in the form of invariant subsets, the bijection onto Dyck paths, and the Gorsky-Mazin-Vazirani recursions that form the building blocks of the upcoming definitions. \S\ref{sec:Schroder} is the combinatorial heart of the article and where Theorems \ref{introthm:schroeder} and \ref{intro2} are proved. Here, triangular partitions are further explored and the $(q,t)$-Schr\"oder polynomials are finally defined. It is in \S\ref{sec:trian knot} that most of the main results come together. We recall the work of Galashin-Lam relating monotone and Coxeter knots, and follow up by connecting them further to the binary knots of Hogancamp-Mellit, ending with the proofs of Theorems \ref{thm:intro3} and \ref{intro4}. Finally, in \S\ref{sec:qtSchroderTheorem} we relate the triangular $(q,t)$-Schr\"oder polynomial with the existing Catalan combinatorics. We review the shuffle theorems, LLT polynomials, and superization, before proving Theorem \ref{intro5}. Lastly, in Appendix \ref{sec:AppendixA} we provide an alternative self contained derivation of the standardization procedure of  \S\ref{sec:qtSchroderTheorem} and in Appendix \ref{sec:AppendixB} we give a way of extending the binary sequences in \S\ref{sec:trian knot} to certain links. 

\section*{Acknowledgments}
We are grateful to the American Institute of Mathematics for hosting the workshop ``Algebra, Geometry, and Combinatorics of Link Homology" where this work was initiated. We also thank Eugene Gorsky for fruitful conversations. C. Caprau was partially supported by NSF RUI grant, DMS-2204386.

\section{Khovanov--Rozansky Homology and the Hogancamp-Mellit Recursion} \label{sec:KR and HM} 

We begin by recalling some basic topological notions.

\subsection{Framed Knots and Cables}\label{ss:framed knots}

Let $D^2 =\{z\in \C\:|\: |z|\leq 1\}$ denote the unit disk in the plane $\C = \R^2$, and let $S^1=\partial D^2$ be the unit circle.  Let $\T=S^1\times S^1$ denote the 2-dimensional torus.  We identify $S^1=\R/\Z$ via the mapping $\R\twoheadrightarrow S^1$ sending $t\mapsto e^{2\pi i t}$, which induces an identification of $\T$ with $\R^2/\Z^2$ (restricting to the unit square $[0,1]^2$ gives the familiar realization of $\T$ as the square with opposite edges identified).  The \emph{longitude} of $\T$ is the copy of $S^1$ identified as the image of the $x$-axis $\R\times\{0\}$ with \emph{leftward orientation}.  The \emph{meridian} of $\T$ is the copy of $S^1$ identified as the image of the $y$-axis $\{0\}\times \R$ with its \emph{upward} orientation.

We may embed $\T$ in $\R^3$ in a standard fashion with image the set of points of distance $1/2$ from $\{x^2+y^2=1\}$, so that the meridian of $\T$ maps to the circle with parametrization $\theta\mapsto (0,1,0)+\frac{1}{2}(0,-\cos(\theta),\sin(\theta))$, and the longitude of $\T$ maps to the circle parametrized by $\theta\mapsto \frac{1}{2}(\cos(\theta),\sin(\theta),0)$.

Given relatively prime integers $m,n$, we consider the line segment in $\R^2$ from $(m,0)$ to $(0,n)$.  The image of this segment in $\T$ is a closed curve, denoted $T(m,n)^\T$.  The image of $T(m,n)^\T$ under the standard embedding $\T\hookrightarrow \R^3$ is the \newword{$(m,n)$-torus knot}, denoted $T(m,n)$.  Given $d\geq 1$, let $T(md,nd)^\T$ denote $d$ parallel copies of $T(m,n)^\T$ in $\T$, whose image in $\R^3$ is the torus link $T(md,nd)$.

We will be interested in knots which are obtained as cables of torus knots.  In order to discuss these, we must introduce the notion of framed knots.  
A \newword{framed knot} $\mathcal{K}$ is an embedding of the {solid torus} $S^1\times D^2\hookrightarrow \R^3$, regarded up to isotopy.  Given any point $p\in D^2$ we obtain an ordinary knot $K$ by restricting $\mathcal{K}$ to an embedding $S^1\times\{p\}\hookrightarrow \R^3$.  The knot $K$ does not depend on $p$, up to isotopy, and we say that $\mathcal{K}$ is a framing of $K$.

If $\mathcal{K}$ is a framed knot and $p_1\neq p_2\in D^2$ are distinct points, then we may consider knots $K_i$ obtained as the images of $S^1\times\{p_i\}$ under $\mathcal{K}$.  Let $r$ denote the \newword{linking number} of the 2-component link $K_1\cup K_2$  (upon choosing a diagram for $K_1\cup K_2$, the linking number is one half the number of crossings between $K_1$ and $K_2$, counted with signs).  This linking number does not depend on the choice of $p_1$ and $p_2$, and is called the \newword{framing coefficient of $\mathcal{K}$}.  Equivalently (but perhaps less precise), the framing coefficient of $\mathcal{K}$ is the number of twists in the embedded annulus obtained by taking the image of $S^1\times a$, where $a\subset D^2$ is an embedded arc connecting the points $p_1$ and $p_2$.

Up to isotopy, a framed knot is completely determined by its framing coefficient $r$, and we will abuse notation and write $\mathcal{K}=(K,r)$.  The \newword{Seifert framing} of $K$ is the unique framing with coefficient $r=0$.

Given a framed knot $\mathcal{K}=(K,r)$, restricting from the solid torus to its boundary yields embedding $\phi_{K,r}\colon \T\hookrightarrow \R^3$.  When $K=U$ is the unknot with Seifert framing, we obtain the standard embedding $\phi_{U,0}\colon \T\hookrightarrow \R^3$.  For more general $r\in \Z$, the embedding $\phi_{U,r}$ sends the closed curve $T(p,q)^\T$ to $T(p,q+rp)$.

\begin{definition}\label{def:cable}
Let $K$ be a knot.  The \newword{$(p,q)$-cable of $K$}, denoted $K(p,q)$, is defined to be the image of $T(p,q)^\T$ under the embedding $\phi_{K,0}\colon \T\hookrightarrow \R^3$ introduced above.
\end{definition}

\begin{remark}\label{rmk:cable other framings}
We can also realize $K(p,q)$ as the image of $T(p,q-rp)$ under $\phi_{K,r}$ for any $r\in \Z$, since the effect of changing the framing from 0 to $r$ is cancelled by the insertion of $r$ negative full twists (in the meridian direction) into $T(p,q)^\T$, obtaining $T(p,q-rp)^\T$.
\end{remark}

\subsection{Braid Presentations of Some Special Knots}
\label{ss:special cable}

In this paper we are particularly interested in the $(d,mnd+1)$-cable of the torus knot $T(m,n)$, denoted $T(m,n)(d,mnd+1)$, where $m,n,d$ are positive integers with $m,n$ relatively prime. 
These knots are the simplest instances of algebraic knots (see \S \ref{subsec:ORS}).  Note that this family includes positive torus knots as the special case $d=1$.   In this section we will focus on finding a braid presentation of these knots.

Let $\Br_n$ denote the \newword{braid group} on $n$ strands.  Let $\one_n\in \Br_n$ denote the identity braid.  For $1\leq i\leq n-1$ let $\sigma_i$ denote the standard Artin generator, and let $\cox_n:=\sigma_1\cdots \sigma_{n-1}\in \Br_n$ denote the \newword{standard Coxeter braid} on $n$ strands.  These are pictured in Figure \ref{eq:artin and coxeter}.

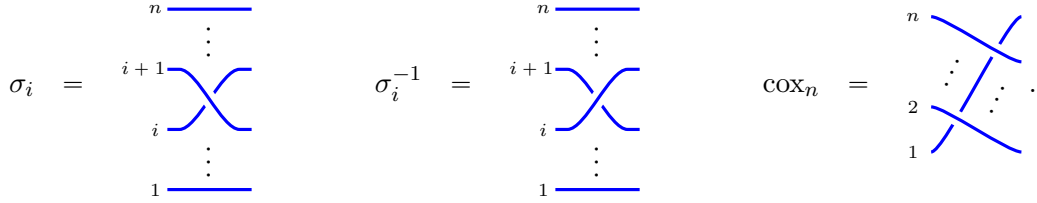
\begin{figure}[ht]
\begin{equation*}
\sigma_i \ \ = \ \ 
\begin{tikzpicture}[scale=.8,baseline=.5cm]
\draw[blue, very thick]
(-.2,0)--(0,0) ..controls++(.3,0)and++(-.3,0)..(1,1)--(1.2,1);
\draw[white, line width=2mm]
(-.2,1)--(0,1) ..controls++(.3,0)and++(-.3,0)..(1,0)--(1.2,0);
\draw[blue, very thick]
(-.2,1)--(0,1) ..controls++(.3,0)and++(-.3,0)..(1,0)--(1.2,0);
\draw[blue, very thick]
(-.2,-1)--(1.2,-1)
(-.2,2)--(1.2,2);
\node[rotate=90] at (.5,-.5) {$\cdots$};
\node[rotate=90] at (.5,1.5) {$\cdots$};
\node at (-.4,-1) {\tiny $1$};
\node at (-.4,0) {\tiny $i$};
\node at (-.6,1) {\tiny $i+1$};
\node at (-.4,2) {\tiny $n$};
\end{tikzpicture}
\qquad \qquad
\sigma_i\iinv \ \ = \ \ 
\begin{tikzpicture}[scale=.8,baseline=.5cm]
\draw[blue, very thick]
(-.2,1)--(0,1) ..controls++(.3,0)and++(-.3,0)..(1,0)--(1.2,0);
\draw[white, line width=2mm]
(-.2,0)--(0,0) ..controls++(.3,0)and++(-.3,0)..(1,1)--(1.2,1);
\draw[blue, very thick]
(-.2,0)--(0,0) ..controls++(.3,0)and++(-.3,0)..(1,1)--(1.2,1);
\draw[blue, very thick]
(-.2,-1)--(1.2,-1)
(-.2,2)--(1.2,2);
\node[rotate=90] at (.5,-.5) {$\cdots$};
\node[rotate=90] at (.5,1.5) {$\cdots$};
\node at (-.4,-1) {\tiny $1$};
\node at (-.4,0) {\tiny $i$};
\node at (-.6,1) {\tiny $i+1$};
\node at (-.4,2) {\tiny $n$};
\end{tikzpicture}
\qquad \qquad
\cox_n \ \ = \ \ 
\begin{tikzpicture}[baseline=.8cm,scale=.6]
\draw[blue, very thick]
(0,0)..controls++(.3,0)and++(-.3,0)..(2,3);
\draw[white, line width=2mm]
(0,3)..controls++(.3,0)and++(-.3,0)..(2,2)
(0,1)..controls++(.3,0)and++(-.3,0)..(2,0);
\draw[blue, very thick]
(0,3)..controls++(.3,0)and++(-.3,0)..(2,2)
(0,1)..controls++(.3,0)and++(-.3,0)..(2,0);
\node[rotate=65] at (1.5,1.2) {$\cdots$};
\node[rotate=65] at (.5,1.8) {$\cdots$};
\node at (-.4,0) {\tiny $1$};
\node at (-.4,1) {\tiny $2$};
\node at (-.4,3) {\tiny $n$};
\end{tikzpicture}.
\end{equation*}
\caption{The standard Coxeter generators of the braid group (left) and the standard Coxeter braid (right).}\label{eq:artin and coxeter}
\end{figure}

Given braids $\b_1\in \Br_{n_1}$ and $\b_2\in \Br_{n_2}$, let $\b_1\sqcup \b_2\in \Br_{n_1+n_2}$ denote the union of the braids $\b_1$ and $\b_2$.  In terms of the diagrams above, $\b_1\sqcup \b_2$ consists of $\b_2$ above $\b_1$. Given a braid $\b$, let $\hat{\b}$ denote the oriented link obtained by closing $\b$.

\begin{example}\label{ex:Tmn braid}
For integers $M,N$ with $M\geq 1$, the torus link $T(M,N)$ is the closure of the braid $\cox_M^N$.
\end{example}

\begin{lemma}\label{lemma:Tndmd as cable}
For $m,n$ coprime and $d\in \Z_{\geq 1}$, we have $T(md,nd)=T(m,n)(d,mnd)$.
\end{lemma}
\begin{proof}
By definition, $T(md,nd)$ consists of $d$ parallel copies of $T(m,n)$, where the notion of ``parallel'' comes from the embedding in the standard torus $\T$.  On the other hand, $T(m,n)(d,mnd)$ is the image of $T(d,0)^\T$ under the embedding $\phi_{T(m,n),mn}$; this consists of $d$ parallel copies of $T(m,n)$, where the notion of ``parallel'' comes from the framing coefficient $mn$, so that two distinct copies of $T(m,n)$ in $\phi_{T(m,n),mn}(T(d,0)^\T)$ have linking number $mn$.  To prove the lemma, we just need to see that any two distinct copies of $T(m,n)$ in $T(md,nd)$ have linking number $mn$.  To see this, assume without loss of generality that $d=2$.  Then we have a presentation of $T(2m,2n)$ as the closure of the braid $(\cox_{2m})^{2n}$.  Of the $2n(2m-1)$ crossings, $2mn$ of these involve the two distinct components, so the linking number between these components is $mn$, which proves the lemma.
\end{proof}

\begin{figure}[ht]
\[\includegraphics[scale=.17,angle=0,origin=c]{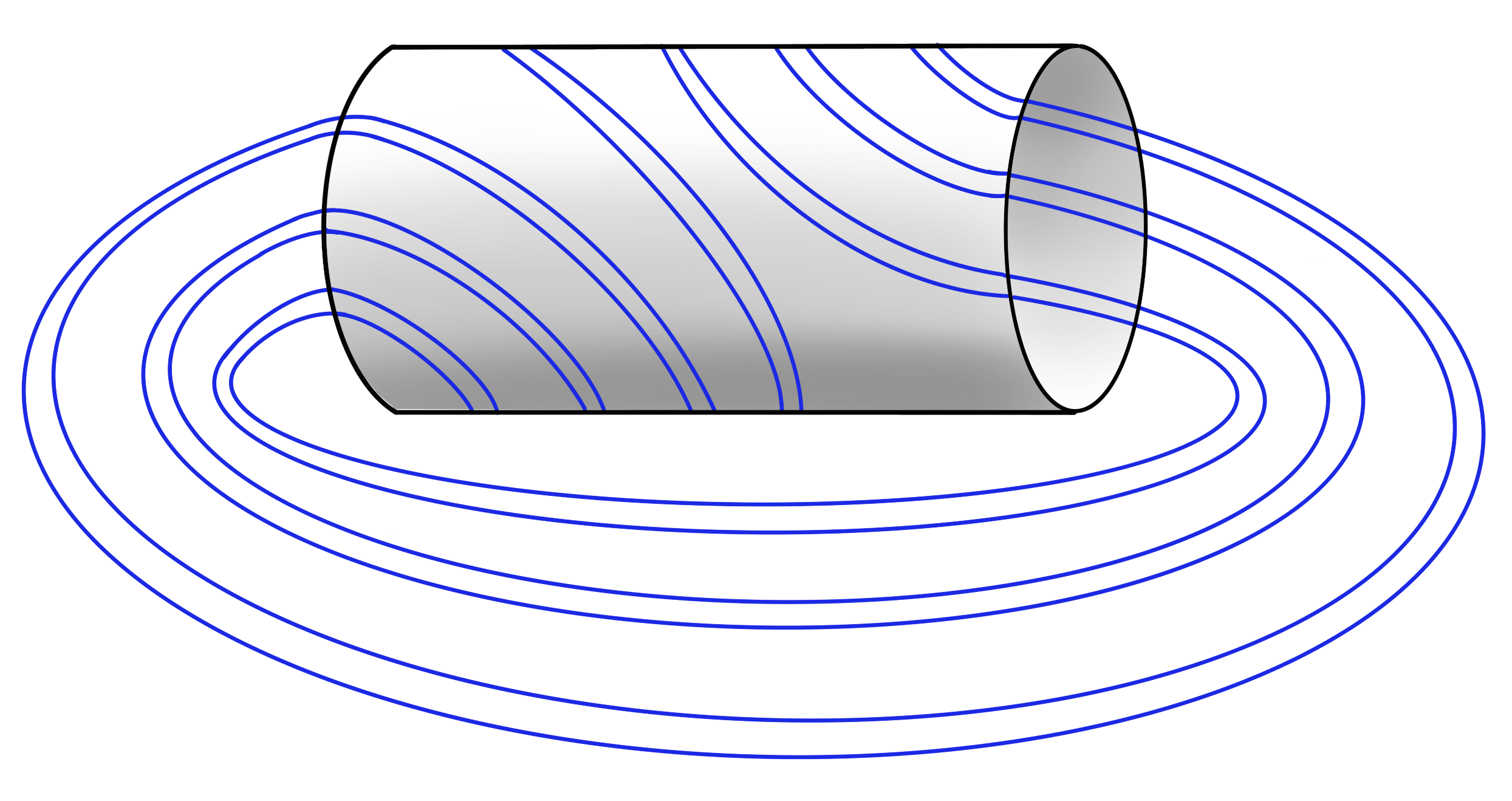}\]
\caption{Realizing $T(6,8)$ as the $(2,24)$ cable of $T(3,4)$.}
\end{figure}

From Lemma \ref{lemma:Tndmd as cable} we deduce the following.

\begin{lemma}\label{lemma:cable presentation}
$T(m,n)(d,mnd+1)$ can be represented as the closure of $(\cox_d\sqcup \one_{(m-1)d})(\cox_{md})^{nd}$. 
\end{lemma}
\begin{proof}
Lemma \ref{lemma:Tndmd as cable} tells us that $T(m,n)(d,mnd)=T(md,nd)$.  It is known that $T(md,nd)$ can be represented as the closure of the braid $(\cox_{md})^{nd}$.  To obtain $T(m,n)(d,mnd+1)$ from $T(m,n)(d,mnd)$, insert one more copy of the braid $\cox_d$, and the statement follows.
\end{proof}

\subsection{Soergel Bimodules and Rouquier Complexes}
\label{ss:HH}
We briefly recall the main ideas needed for the construction of Khovanov--Rozansky homology and refer the reader to \cite{KR1, Kh-Sbim, KR2, EH16, HM, SoergelBook, Mellit-Homology} for detailed exposure on this topic. 

Fix an integer $n\geq 0$ and let $\k$ be a commutative ring.  Let $R=\k[x_1,\ldots,x_n]$ regarded as a graded ring with $\deg(x_i)=2$.  If $B$ is a graded $(R,R)$-bimodule and $k\in \frac{1}{2}\Z$, then we denote by $q^k B$ the bimodule $B$ with grading ``shifted up by $2k$ units''; that is, $(q^kB)_i = B_{i-2k}$.  Let $\star=\otimes_R$ denote the monoidal operation on bimodules, and let $\boxtimes=\otimes_\mathbb{K}$ denote the external tensor product.

If $X$ is a complex of graded $(R,R)$-bimodules, then we let $\Omega X$ denote the cohomological shift; that is, the $i$-th chain group of $\Omega X$ is $X^{i-1}$, with negated differential $d_{\Omega X} = - d_X$.

Given $k\in \frac{1}{2}\Z$, we let $t^k X:= \Omega^{2k}q^{-k}X$.  Observe that $q,t,\Omega$ are not independent, but satisfy the relation
$q^{\frac{1}{2}}t^{\frac{1}{2}} = \Omega.$

\begin{figure}[ht]
\[
\includegraphics[scale=.2,angle=0,origin=c]{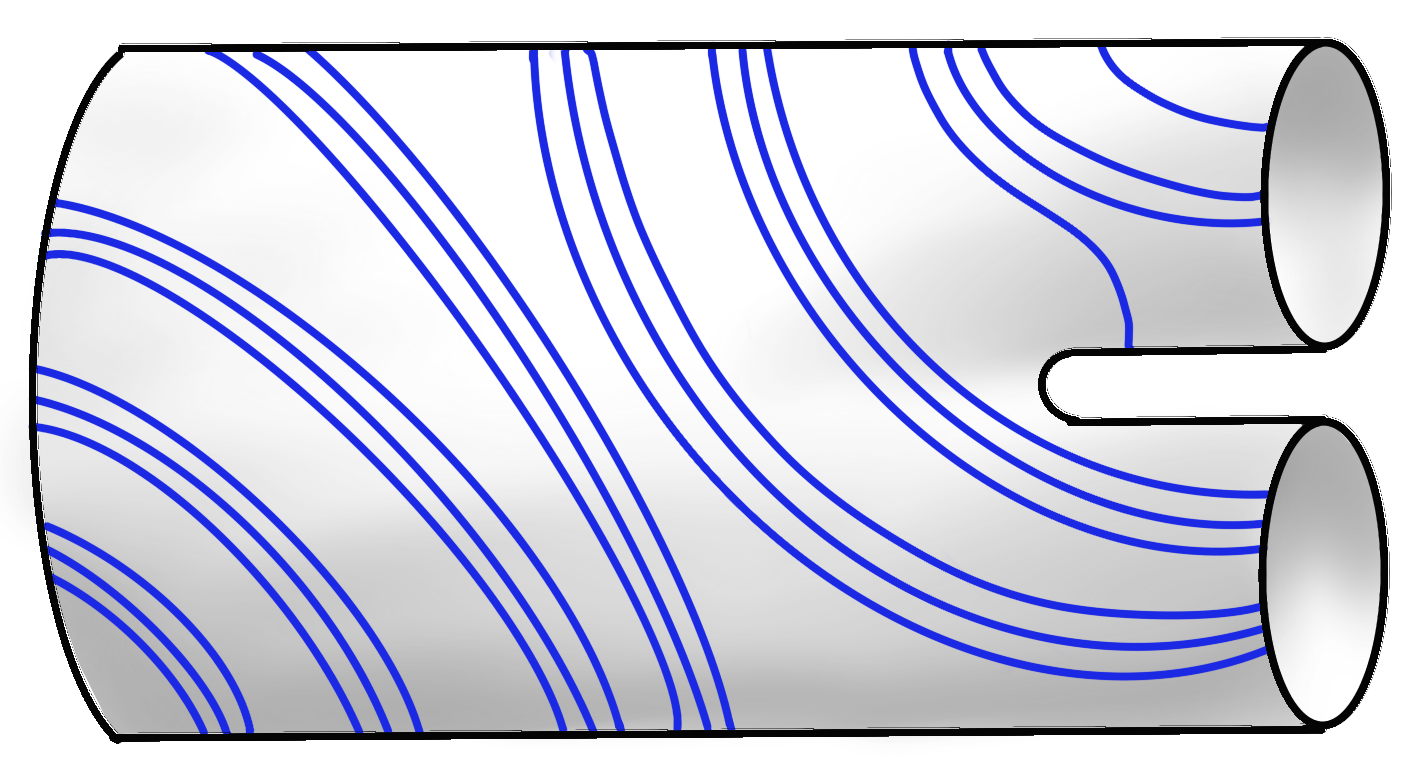}
\]
\caption{A braid whose closure is $T(3,4)(3,37)$, shown here as embedded in a pair-of-pants surface.
}\label{fig: 3-37 cable}
\end{figure}

We let $\SBim_n$ denote the category of \newword{Soergel bimodules} for the symmetric group $S_n$ \cite{SoergelBook}.  This is a monoidal subcategory of the category of graded $(R,R)$-bimodules. We denote by $ \Ch^b(\SBim_n)$ the category of bounded chain complexes in $\SBim_n$. 

The grading shifts interact with $\Hom$ spaces as follows:
\[
\Hom_{\SBim_n}(B,q^k B') = q^k \Hom_{\SBim_n}(B,B') \ ,\qquad \Hom_{\SBim_n}(q^l B,B') = q^{-l}\Hom_{\SBim_n}(B,B').
\]

Now, for $1\leq i\leq n-1$, we let 
\[B_i:=q^{-\frac{1}{2}}R\otimes_{R^{s_i}}R,\]
where $s_i$ is the simple transposition $(i,i+1)$ in the symmetric group $S_n$, and $R^{s_i}\subset R$ is the graded subalgebra of $s_i$-invariant polynomials.  We have canonical degree zero maps
\[
B_i\rightarrow q^{-\frac{1}{2}}R \ \text {given by} \ f\otimes g\mapsto fg
\]
and
\[
q^{\frac{1}{2}} R\rightarrow B_i \ \text{given by} \ 1\mapsto x_i\otimes 1 - 1\otimes x_{i+1} .
\]

\begin{definition}
Let $T_i^{\pm1}= T(\sigma_i^{\pm1})$ be following two-term complexes:
\[
T_i :=\left(t^{-1/2}B_i\rightarrow R\right) \quad \text{and} \qquad T_i^{-1} := \left(R\rightarrow t^{1/2} B_i\right).
\]
The \newword{Rouquier complex} associated to a braid $\beta = \sigma^\pm_{i_1}\dots\sigma^\pm_{i_k} \in \Br_n$ is the chain complex $T(\beta) \in \Ch^b(\SBim_n)$ obtained by tensoring 
\[
T(\sigma^\pm_{i_1}),\dots, T(\sigma^\pm_{i_k}).
\]
\end{definition}

\begin{remark}
The above conventions for $q$ and $t$ will enable a direct comparison between our computations of $\HHH$ and the $(q,t)$-Catalan numbers.  It is more traditional in Soergel bimodule literature to denote the grading and homological shifts of $X$ by $X(1)$ and $X[1]$, respectively; these are related to our $q,t,\Omega$ by
\[
X(1) = q^{-\frac{1}{2}}X \ ,\qquad X[1] = \Omega\iinv X.
\]
\end{remark}

\subsection{Khovanov--Rozansky Homology} Let $\kMod{\Z\times \Z}$ denote the category of bigraded $\k$-modules.  Let $\HH^\cdot\colon \SBim_n\to \kMod{\Z\times \Z}$ be the $\k$-linear functor sending a graded $(R,R)$-bimodule to its Hochschild cohomology.  Extending to complexes gives us a functor 
\[\HH^\cdot\colon \Ch^b(\SBim_n)\to \Ch^b(\kMod{\Z\times \Z}).\]
The complex $\HH^\cdot(X)$ is triply graded via
\[
\HH^{i,j,k}(X) := \text{Soergel degree-$i$ component of } \HH^j(X^k).
\]
The homology of $\HH^\cdot(X)$ is denoted by
\[
\HHH^\cdot(X):=H^\cdot(\HH^\cdot(X)).
\]
On triply graded $\k$-modules, we will use the grading shifts $q,a,t$ defined by
\[
(qM)^{i,j,k} = M^{i-2,j,k} \ , \qquad (aM)^{i,j,k} = M^{i+2,j-1,k} \ ,\qquad (tM)^{i,j,k}=M^{i+2,j,k-2},
\]
and the cohomological shift
\[
(\Omega M)^{i,j,k} = M^{i,j,k-1}.
\]
Note that $qt = \Omega^2$.  Note also that $a$ acts by shifing up by 1 in the Hochschild degree, together with a shift down by 2 in the Soergel bimodule degree.
\begin{definition}\label{def:KR homology}
Let $\b\in \Br_n$ and $e=e(\b)$ be the number of positive crossings minus the number of negative crossings, and let $c=c(\b)$ denote the number of components of the link represented by $\b$. Define the \newword{normalizing exponent} of the braid $\b$ by the formula
\begin{equation}\label{eq:normalizing exponent}
\d(\b):=\frac{1}{2}(e(\b)+c(\b)-n(\b)).
\end{equation} 
The \newword{Khovanov--Rozansky complex} of $\b$ is
\begin{align*}
\KRC(\b)
&:= (at^{1/2}q^{-1/2})^{\d(\b)} \HH^\cdot(T(\b)).
\end{align*}
Its homology is the \newword{Khovanov--Rozansky} homology of the link $\widehat{\b}$ represented by $\b$, and is denoted by $\KRH(\b):=(at^{1/2}q^{-1/2})^{\d(\b)}\HHH(T(\b))$.  
\end{definition}

\begin{remark}\label{rmk:normalizing exponent and genus}
If $K$ is a positive knot and $\b$ is a positive braid representing $K$, then the normalizing exponent $\d(\b)$ agrees with the \emph{genus} of $K$ (in general, $\d(\b)$ is the genus of the Seifert surface produced by Seifert's algorithm). In particular, for positive braids $\d(\b)$ depends only on the braid closure, hence is actually a well-defined knot invariant. Note that all the braids of interest in this paper are positive, so this remark applies.
\end{remark}

Recall that for any tri-graded $\k$-module $M = \bigoplus_{i,j,k} M^{i,j,k}$, the \newword{Poincar\'e series} of $M$ computes the sum of the graded ranks over $\k$; namely, 
\[\PC(M) = \sum_{i,j,k} q^it^ja^k \mathrm{rank}_\k(M^{i,j,k}).\]

\begin{definition}\label{def:KR series}
The \newword{Khovanov--Rozansky (KR) series} of a link $L=\widehat{\b}$ is defined to be $\PC_{L}^{\KR}(q,t,a):= \PC(\KRH(\b))$.
\end{definition}

\begin{theorem}[\cite{KR1, Kh-Sbim}] \label{thm:KhR isotopic}
If $\b$ and $\b'$ represent isotopic links, then there is a homotopy equivalence $\KRC(\b)\simeq \KRC(\b')$, and consequently $\PC_{\widehat{\b}}^{\KR}(q,t,a) = \PC_{\widehat{\b'}}^{\KR}(q,t,a)$.
\end{theorem}

\begin{example}\label{ex:delta MN}
For non-negative integers $M,N\geq 0$, the torus link $T(M,N)$ is the closure of $\cox_M^N$.  For this braid we have $e=(M-1)N$ and $c=\gcd(M,N)$, hence
\[
\d = \frac{MN-M-N+\gcd(M,N)}{2}.
\]
\end{example}

\begin{example}\label{ex:delta mnd}
For integers $m,n,d\geq 0$ with $m,n$ coprime, we consider the $(d,mnd+1)$-cable of $T(m,n)$ as the closure of $(\cox_d\sqcup \one_{(m-1)d})\cox_{md}^{nd}$, as in Lemma \ref{lemma:cable presentation}.  For this braid we have $e=(d-1)+(md-1)(nd)$, with $c=1$ and number of strands $md$, hence
\[
\d = \frac{(md)(nd)-md-nd+d}{2}.
\]
\end{example}

The integers $\d$ appearing in the above examples will occur frequently in this paper, and we will abbreviate them.

\begin{definition}\label{def:delta MN}
For integers $M,N\geq 0$, we let $\d(M,N):=\frac{MN-M-N+\gcd(M,N)}{2}$.
\end{definition}

\subsection{Links from Binary Sequences}
\label{ss:binary links}

\begin{definition}\label{def:Cvw}
If $\mathbf{u}=(u_1,\ldots,u_m)\in \{0,1\}^m$, we let $|\mathbf{u}|=\sum_i u_i$.  
Given sequences $\mathbf{u}\in \{0,1\}^m$, $\mathbf{v}\in \{0,1\}^n$ with $|\mathbf{u}|=|\mathbf{v}|$, we define a complex $C(\mathbf{u},\mathbf{v})\in \Ch^b(\SBim_n)$ as follows.  Abbreviate by writing $l=|\mathbf{u}|=|\mathbf{v}|$.  Let $\pi_{\mathbf{u}}\in S_m$ denote the shuffle permutation corresponding to $\mathbf{u}$, i.e.~ the permutation given in one-line notation by $(i_1,\ldots,i_k, j_1,\ldots,j_l)$ with $k+l=m$, where $i_1<\cdots <i_k$ are the indices for which $u_i=0$, and $j_1<\cdots < j_l$ are the indices for which $u_j=1$.  Let $\b_{\mathbf{u}}$ denote the positive braid lift of $\pi_{\mathbf{u}}\iinv$.  Next, write $\mathbf{v}$ as $\mathbf{v}=(0^{r_0}10^{r_1}1 \cdots 1 0^{r_l})$, and set
\[
\gamma_{m,\mathbf{v}}:= (\cox_m)^{r_0} (\one_1\boxtimes \cox_{m-1})^{r_1}\cdots (\one_{l}\boxtimes \cox_{m-l})^{r_{l}}.
\]
Finally, let $\b(\mathbf{u},\mathbf{v})=\b_{\mathbf{u}}\gamma_{m,\mathbf{v}}$, and
\[
C(\mathbf{u},\mathbf{v}):=T(\b(\mathbf{u},\mathbf{v}))\star (K_{l}\boxtimes \one_{n-l}),
\]
where $K_l$ is the complex constructed in \cite{EH16, HM}.  When $l=1$ we let $K_{\mathbf{u},\mathbf{v}}$ denote the closure of the braid $\b(\mathbf{u},\mathbf{v})$.
\end{definition}

\begin{example}
If $\mathbf{u}=(10001)$ and $\mathbf{v}=(011000)$, then $C(\mathbf{u},\mathbf{v})$ may be pictured as in Figure \ref{fig:Cuv}.
\end{example}

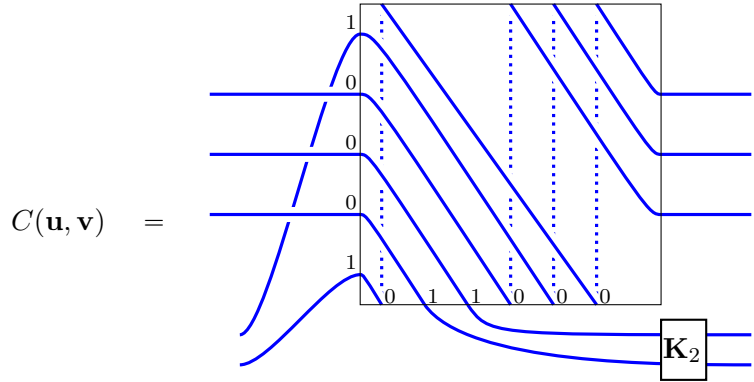
\begin{figure}[ht]
\[
C(\mathbf{u},\mathbf{v})\ \ \ = \ \ \ 
\begin{tikzpicture}[scale=4,baseline=1cm]
\def\m{5}
\def\n{7}
\def\r{2} 
%
\draw[blue, very thick]
(-.4,-.1) .. controls ++(.1,0) and ++(-.1,0).. (0,{(5-.5)/\m})
(-.4,-.2) .. controls ++(.1,0) and ++(-.1,0).. (0,{(1-.5)/\m});
\draw[blue, very thick]
({(2-.5)/\n},0) ..controls ++(.14,-.2)and++(-.3,0).. (1.3,-.2)
({(3-.5)/\n},0) ..controls ++(.07,-.1).. (1.3,-.1);
\foreach \i in {2,3,4}{  
\draw[white, line width=2mm] (-.5,{(\i-.5)/\m})--(0,{(\i-.5)/\m});
\draw[blue, very thick] (-.5,{(\i-.5)/\m})--(0,{(\i-.5)/\m});
}
\foreach \j in {1,4,5,6}{ 
\draw[blue, very thick,dotted] ({(\j-.5)/\n},1) -- ({(\j-.5)/\n},0);
}
\foreach \i in {1,2,3,4,5}{ 
\draw[white, line width=2mm] (0,{(\i-.5)/\m}) .. controls ++({\i/(20*\m)},0)..  ({(\i-.5)/\n},0);
\draw[blue, very thick] (0,{(\i-.5)/\m}) .. controls ++({\i/(20*\m)},0)..  ({(\i-.5)/\n},0);
}
\foreach \j in {1}{ 
\draw[white, line width=2mm] ({(\j-.5)/\n},1) -- ({(\j+\m-.5)/\n},0);
\draw[blue, very thick] ({(\j-.5)/\n},1) -- ({(\j+\m-.5)/\n},0);
}
\foreach \j in {2,3,4}{ 
\draw[white, line width=2mm]  (1,{(\j-.5)/\m}) .. controls ++({(\j-\m-1)/(20*\m)},0).. ({(\j+\r-.5)/\n},1);
\draw[blue, very thick]  (1,{(\j-.5)/\m}) .. controls ++({(\j-\m-1)/(20*\m)},0).. ({(\j+\r-.5)/\n},1);
\draw[blue, very thick] (1.3,{(\j-.5)/\m})--(1,{(\j-.5)/\m});
}
\draw (0,0) rectangle (1,1);
\draw[thick,fill=white] (1,-.25) rectangle (1.15,-.05);
\node at (1.07,-.15) {$\mathbf{K}_2$};
\node at (-.03,{(1-.3)/\m}) {\scriptsize{$1$}};
\node at (-.03,{(2-.3)/\m}) {\scriptsize{0}};
\node at (-.03,{(3-.3)/\m}) {\scriptsize{0}};
\node at (-.03,{(4-.3)/\m}) {\scriptsize{0}};
\node at (-.03,{(5-.3)/\m}) {\scriptsize{1}};
\node at ({(1-.3)/\n},.03) {\scriptsize{0}};
\node at ({(2-.3)/\n},.03) {\scriptsize{1}};
\node at ({(3-.3)/\n},.03) {\scriptsize{1}};
\node at ({(4-.3)/\n},.03) {\scriptsize{0}};
\node at ({(5-.3)/\n},.03) {\scriptsize{0}};
\node at ({(6-.3)/\n},.03) {\scriptsize{0}};
\end{tikzpicture}
\]
\caption{The complex $C({\bf u,v})$ for $\mathbf{u}=(10001)$ and $\mathbf{v}=(011000)$.}\label{fig:Cuv}
\end{figure}

The following is proven in \cite{HM}.
\begin{lemma}\label{lemma:Cvw link}
If $|\mathbf{u}|=|\mathbf{v}|=1$, then $\HHH(\b(\mathbf{u},\mathbf{v})) \cong \HHH(C(\mathbf{u},\mathbf{v})) \otimes \k[x_1]$.
\end{lemma}

\begin{lemma}\label{lemma:binary braids}
For $\mathbf{u}=0^i10^j$ and $\mathbf{v}=0^k10^l$ we have that $K_{\mathbf{u},\mathbf{v}}$ is the closure of the braid
\[
\b(0^i10^j,0^k10^l)=(\sigma_1\cdots\sigma_{i})(\sigma_1\sigma_2\cdots \sigma_{m-1})^k(\sigma_2\cdots \sigma_{m-1})^l
\]
with normalizing exponent \eqref{eq:normalizing exponent} given by
\begin{equation}\label{eq:binary normalization}
\d(\b(0^i10^j,0^k10^l)) = \frac{(i+j)(k+l)-j-l+c-1}{2}.
\end{equation}
\end{lemma}
\begin{proof}
Immediate from the definitions.
\end{proof}

\begin{lemma}\label{lemma:cable via binary}
Given integers $m,n,d\geq 0$ with $m,n$ coprime, the $(d,mnd+1)$-cable of $T(m,n)$ coincides with $K_{0^{d-1}1 0^{(m-1)d},0^{nd}1}$.  Moreover, the normalizing exponent \eqref{eq:binary normalization} agrees with $\d(md,nd)$ from Definition \ref{def:delta MN}.
\end{lemma}
\begin{proof}
From the definitions we have $\b(0^{d-1}10^{(m-1)d},0^{nd}1)=(\cox_d\sqcup \one_{d(m-1)})\cox_{md}^{nd}$, whose closure is the $(d,mnd+1)$-cable of the torus knot $T(m,n)$, by Lemma \ref{lemma:cable presentation}.  Comparison of the normalizing exponent with $\d(md,nd)$ is elementary.
\end{proof}

\begin{construction}\label{constr:Kuv}
Let $\mathbf{u}\in \{0,1\}^M$, $\mathbf{v}\in \{0,1\}^N$ be binary sequences with $|\mathbf{u}|=|\mathbf{v}|=1$.  We provide a description of the link $K_{\mathbf{u},\mathbf{v}}$ from Definition \ref{def:Cvw}.  First, choose real numbers $0<y_1<\cdots<y_M<1$ and $0<x_1<\cdots<x_N<1$.  Consider the unit square $[0,1]^2$ with the following set of points on its boundary:
\[
S=\bigcup_{1\leq i\leq M} (0,y_i) \cup \bigcup_{1\leq j\leq N} (x_j,0) \cup \bigcup_{\substack{1\leq i\leq M \\ i\neq a}} (1,y_i)\cup \bigcup_{\substack{1\leq j\leq N \\ j\neq b}} (x_j,1),
\]
where $a$ is the index with $u_a=1$ and $b$ is the index with $v_b=1$.
Let $D$ be union of arcs of the following types:
\begin{enumerate}
    \item Consider the ordered tuples of boundary points
    \begin{align*}
        (p_1,\ldots,p_{M+N-1})&=\Big((0,y_1),\ldots,(0,y_M),(x_1,1),\ldots, \widehat{(x_b,1)},\ldots,(x_N,1)\Big),\\
        (q_1,\ldots,p_{M+N-1})&=\Big((x_1,0),\ldots,(x_N,0),(1,y_1),\ldots\widehat{(1,y_a)},\ldots,(1,y_M)\Big).
    \end{align*}
    For $1\leq k\leq M+N-1$, let $A_k$ denote the straight line segment from $p_k$ to $q_k$.
    \item For each $1\leq i\leq M$ with $i\neq a$, let $H_i$ denote an arc connecting $(0,y_i)$ to $(1,y_i)$, disjoint from the interior of $[0,1]^2$.  Assume that the $H_i$, for $1\leq i\leq M$,  are pairwise disjoint.
    \item For each $1\leq j\leq N$ with $j\neq b$, let $V_j$ denote an arc connecting $(x_j,0)$ to $(x_j,1)$, beneath the arcs $A_k$, and disjoint from each $H_i$.  Assume that the $V_j$, for $1\leq j\leq N$, are pairwise disjoint.
    \item The only boundary points that remain unpaired are $(0,y_a)$ and $(x_b,0)$.  Choose $\e>0$ and let $B_-$ be an arc connecting $(0,y_a)$ to $(-\e,-\e)$ below the all arcs $H_i$ (and disjoint from all other arcs); let $B_+$ be an arc connecting $(-\e,-\e)$ to $(x_b,0)$ above all arcs $V_j$ (and disjoint from all other arcs).  Let $B=B_-\cup B_+$.
\end{enumerate}
The union of all arcs of the form $A_k,H_i,V_j$ and $B$ is a closed, piecewise linear diagram which represents $K_{\mathbf{u},\mathbf{v}}$.
\end{construction}

\begin{example}\label{ex:binary knot}
Construction \ref{constr:Kuv} is illustrated for $\mathbf{u}=0010$ and $\mathbf{v}=01000$ in Figure~\ref{fig:binary knot}. The arc $B$ has been drawn in orange in order to facilitate the comparison with the monotone link.
\end{example}

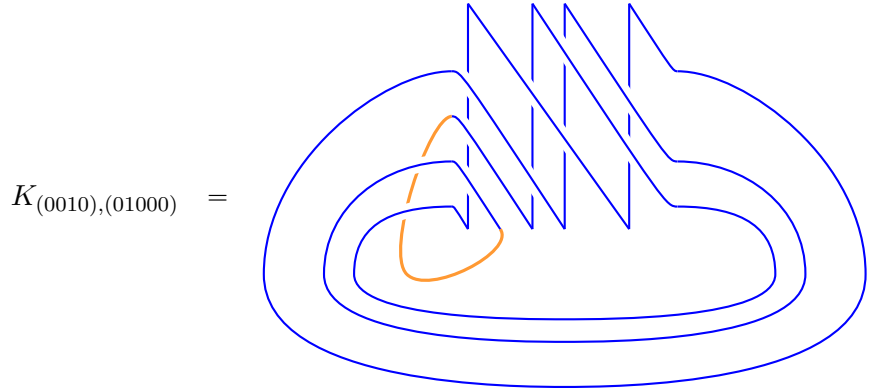
\begin{figure}[ht]
\[
K_{(0010),(01000)} \ \ = \ \ 
\begin{tikzpicture}[scale=3,anchorbase]
\def\m{5}
\def\n{7}
\def\r{2} 
\def\I{3} 
\def\J{2} 
%
\draw[orn, very thick]
(0,.5)..controls++(-.1,0)and++(-.1,.1)..(-.2,-.2)..controls++(.1,-.1)and++(.07,-.1)..({(2-.5)/\n},0);
\foreach \i in {1,2,4}{  
\draw[white, line width=2mm]
(0,{(\i-.5)/\m})..controls++(-.3,0)and++(0,{.1*(\i+1)})..
({-.3-\i/(1.5*\m)},-.2)..controls++(0,{-.1*(\i+1)})and++(-.2,0)..
(.5,{-\i/(2*\m)-.3})..controls++(.2,0)and++(0,{-.1*(\i+1)})..
({1.3+\i/(1.5*\m)},-.2)..controls++(0,{.1*(\i+1)})and++(.3,0)..
(1,{(\i-.5)/\m});
\draw[blue,  thick]
(0,{(\i-.5)/\m})..controls++(-.3,0)and++(0,{.1*(\i+1)})..
({-.3-\i/(1.5*\m)},-.2)..controls++(0,{-.1*(\i+1)})and++(-.2,0)..
(.5,{-\i/(2*\m)-.3})..controls++(.2,0)and++(0,{-.1*(\i+1)})..
({1.3+\i/(1.5*\m)},-.2)..controls++(0,{.1*(\i+1)})and++(.3,0)..
(1,{(\i-.5)/\m});
}
\foreach \j in {1,3,4,6}{ 
\draw[blue, thick] ({(\j-.5)/\n},1) -- ({(\j-.5)/\n},0);
}
\foreach \i in {2,3,4}{ 
\draw[white, line width=2mm] (0,{(\i-.5)/\m}) .. controls ++({\i/(20*\m)},0)..  ({(\i-.5-.1*\m)/\n},.1);
\draw[blue,  thick] (0,{(\i-.5)/\m}) .. controls ++({\i/(20*\m)},0)..  ({(\i-.5)/\n},0);
}
{ \def\i{1}
\draw[blue,  thick] (0,{(\i-.5)/\m}) .. controls ++({\i/(20*\m)},0)..  ({(\i-.5)/\n},0);
}
\foreach \j in {1}{ 
\draw[white, line width=2mm] ({(\j-.5+.1*\m)/\n},1-.1) -- ({(\j+\m-.5-.1*\m)/\n},.1);
\draw[blue,  thick] ({(\j-.5)/\n},1) -- ({(\j+\m-.5)/\n},0);
}
\foreach \j in {1,2,4}{ 
\draw[white, line width=2mm]  (1,{(\j-.5)/\m}) .. controls ++({(\j-\m-1)/(20*\m)},0).. ({(\j+\r-.5+.1*\m)/\n},1-.1);
\draw[blue,  thick]  (1,{(\j-.5)/\m}) .. controls ++({(\j-\m-1)/(20*\m)},0).. ({(\j+\r-.5)/\n},1);
}
\end{tikzpicture}
\]
\caption{Construction \ref{constr:Kuv} illustrated for $\mathbf{u}=0010$ and $\mathbf{v}=01000$.}\label{fig:binary knot}
\end{figure}

\subsection{The Hogancamp-Mellit Recursion}
\label{ss:HM recursion}
In \cite{HM}, the third named author and Mellit presented a recursive computation of $\HHH(C(\mathbf{u},\mathbf{v}))$ (for 
$(C(\mathbf{u},\mathbf{v})$ as defined in Definition \ref{def:Cvw}), which we now recall. 

\begin{definition}\label{def:Rrecursion}
Let ${\bf u } \in \{0,1\}^{m+\ell}$ and ${\bf v } \in \{0,1\}^{n+\ell}$ be two binary sequences with $|{\bf u}|=|{\bf v}|=\ell$. Define $R_{{\bf u},{\bf v}} \in \N[q,t^\pm, a, (1-q)^{-1}]$ as the power series uniquely determined by the following recursions:
\begin{align*}
R_{0{\bf u},0{\bf v}} & =t^{-|{\bf u}|}R_{{\bf u}1,{\bf v}1}+qt^{-|{\bf u}|}R_{{\bf u}0,{\bf v}0}, &
R_{1{\bf u},0{\bf v}}&=R_{{\bf u}1,{\bf v}}, & R_{\emptyset,0^n}&=\left(\frac{1+a}{1-q}\right)^n, 
\\
R_{1{\bf u},1{\bf v}}&=(t^{|{\bf u}|}+a)R_{{\bf u},{\bf v}},&
R_{0{\bf u},1{\bf v}}&=R_{{\bf u},{\bf v}1},
&
R_{0^m,\emptyset}&=\left(\frac{1+a}{1-q}\right)^m,
\end{align*}
where $R_{\emptyset, \emptyset}:=1$. 
\end{definition}

\begin{theorem}[\cite{HM}] \label{thm:RHHH}For any two binary sequences ${\bf u} $ and ${\bf v } $ with $|{\bf u}|=|{\bf v}|$,
we have
\[
\PC(\HHH^\cdot(C(\mathbf{u},\mathbf{v}))) = R_{\mathbf{u},\mathbf{v}}.
\]
In particular, if $|\mathbf{u}|=|\mathbf{v}|=1$, then the KR series of the link $K_{\mathbf{u},\mathbf{v}}:=\hat{\b}(\mathbf{u},\mathbf{v})$ is given by
\begin{equation}\label{eq:Luv Poinc}
\PC_{K_{\mathbf{u},\mathbf{v}}}^{\KR}(q,t,a)=\frac{(at^{1/2}q^{-1/2})^{\d(\b(\mathbf{u},\mathbf{v}))}}{1-q} R_{\bf u,v}(q,t,a),
\end{equation}
where $\d(\b)$ is the normalizing exponent \eqref{eq:normalizing exponent}.
\end{theorem}
\begin{remark}
The denominator $\frac{1}{1-q}$ in \eqref{eq:Luv Poinc} comes from the tensor factor $\k[x_1]$ in Lemma \ref{lemma:Cvw link}.  
\end{remark}
\begin{remark}
The special case ${\bf u}=0^{M-1}1$, ${\bf v}=0^{N-1}1$ yields the KR series of $T(M,N)$.  One of the central aims of this paper is to give a combinatorial interpretation of $K_{\bf u, v}$ and its KR series, when $K_{\bf u, v}$ is a knot.
\end{remark}

This together with Lemma \ref{lemma:cable via binary} gives the following.
\begin{corollary}\label{cor: poincare of cable}
For integers $m,n,d\geq 0$ with $m,n$ coprime, the KR series of the cable knot $T(m,n)(d,mnd+1)$ is given by
\[
\PC_{T(m,n)(d,mnd+1)} = \frac{(at^{1/2}q^{-1/2})^{\d(md,nd)}} {1-q} R_{(0^{d-1}10^{(m-1)d} , 0^{nd}1)}.
\]\qed
\end{corollary}
Here we are using $\d(md,nd)$ from Definition \ref{def:delta MN}, and the computation from Example \ref{ex:delta mnd}.

\section{Dyck paths, Invariant Subsets, and the Gorsky-Mazin-Vazirani Recursion}~
\label{sec:GMV-Recursions}

In this article a \newword{box} or \newword{cell} will be a unit square of the form $[x-1,x]\times [y-1,y]$ in $\R^2$, with $x,y\in \Z_{\geq 1}$.  The point $(x,y)$ is the \emph{northeast (NE) corner} of this box, while $(x-1,y-1)$ is the \emph{southwest (SW) corner}.  We will generally identify a box $\Box$ with its NE corner, thereby labeling boxes by elements of $\Z_{\geq 1}\times \Z_{\geq 1}$.

Given two real numbers $r,s\in\mathbb{R}_{>0},$ consider the line $L_{r,s}:=\{x/r+y/s=1\}$ connecting points $(r,0)$ and $(0,s)$ on the coordinate axes.

\begin{definition}\label{def:box on line}
    Given a line $L_{r,s}$, we say that a box $\Box=(x,y)$ \newword{lies on $L$}, if we have inequalities $(x-1)/r+(y-1)/s\leq 1$ (which expresses that the SW corner of $\Box$ lies weakly below $L_{r,s}$) and $x/r+y/s>1$  (which expresses that the NE corner of $\Box$ lies strictly above $L_{r,s}$).  We say that $\Box$ is below $L$ if $x/r+y/s\leq 1$, and above $L_{r,s}$ if $(x-1)/r+(y-1)/s>1$.
\end{definition}

Given an integer $m$, let $[m]:=\{1,\ldots,m\}$. We will identify a partition $\lambda=(\lambda_1,\ldots,\lambda_\ell)$ with its associated Young diagram (written in the French style) so that $\lambda_1\geq \lambda_2\geq \dots \geq \lambda_\ell$.

\begin{definition}[\cite{BM}]
A partition $\tau$ is \newword{triangular} if there exists a line $L_{r,s}=\{x/r+y/s=1\}$ such that $\tau$ consists of all the boxes contained between $L_{r,s}$ and the coordinate axes. If so, we write $\tau=\tau_{r,s}$ and say that $L_{r,s}$ \newword{cuts out} $\tau_{r,s}.$ Note that the same triangular partition can be cut out by different lines (see Figure \ref{fig: triangular partition}).
\end{definition}

\begin{figure}
\begin{tikzpicture}[scale=1]

\filldraw [fill=yellow!50!white] (0,0)--(0,3)--(1,3)--(1,2)--(2,2)--(2,1)--(4,1)--(4,0)--(0,0);

\draw [thin, color=gray!50!white] (0,0) grid +(7,5);
\draw [->] (0,0)--(0,5.3);
\draw [->] (0,0)--(7.3,0);

\draw [thick, color=green!50!black] (0,0)--(0,3)--(1,3)--(1,2)--(2,2)--(2,1)--(4,1)--(4,0)--(0,0);
\draw [purple](0,3.9)--(5.9,0);
\draw [purple](0,4.6)--(5.2,0);
\draw [purple](0,3.6)--(6.6,0);

\draw [thick, color=blue!50!black] (1,5)--(1,4)--(2,4)--(2,3)--(3,3)--(3,2)--(5,2)--(5,1)--(7,1);

\end{tikzpicture}
\caption{The partition $(4,2,1)$ is triangular and is cut out by different lines; for instance $(4,2,1)=\tau_{3.9,5.9}=\tau_{4.6,5.2}=\tau_{3.6,6.6}.$}\label{fig: triangular partition}
\end{figure}
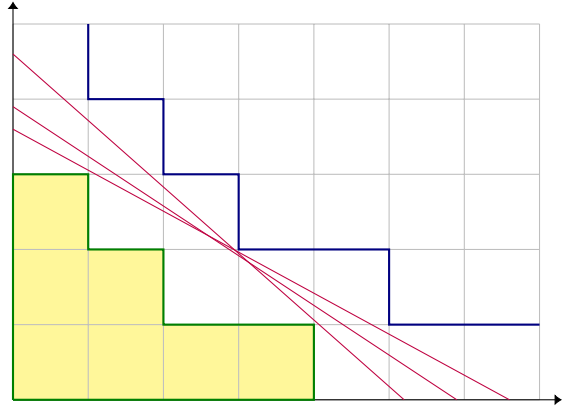

\begin{definition}\label{def:Dyck}
A \newword{Dyck path} is a pair $\pi=(\lambda,\tau)$, where $\tau$ is a triangular partition and $\lambda\subset \tau$ is a subpartition.  We refer to $(\lambda,\tau)$ as a \emph{$\tau$-Dyck path}, and write $\D(\tau)$ for the set of all such pairs.  When $\tau=\tau_{r,s}$ we write $\D(r,s):=\D(\tau_{r,s})$ and refer to elements of this set as \emph{$r,s$-Dyck paths}.

The \newword{area} of a Dyck path $\pi=(\lambda,\tau)$ is defined as
\[
\area(\pi):=|\tau|-|\lambda|.
\]
\end{definition}

\subsection{Rational Dyck Paths and Anderson Fillings} For the remainder of this section we mostly specialize to the case when $r,s$ are relatively prime integers and will resume the general triangular setting in \S\ref{sec:Schroder}. For the most part, the results in this subsection are well known and straightforward but we include the proofs for completeness. 

So then, let $m,n\in \mathbb{Z}_{>0}$ be relatively prime integers and consider the rectangular diagram $[m]\times[n]$.  We may decompose $[m]\times[n]$ as a disjoint union of three subsets:
\[
[m]\times [n] = \tau_{m,n}\sqcup \mathrm{diag}_{m,n}\sqcup \tau_{m,n}^\vee,
\]
where $\mathrm{diag}_{m,n}$ is the set of boxes that lie on the diagonal line $L_{m,n},$ $\tau_{m,n}$ is the corresponding triangular partition, and $\tau_{m,n}^\vee$ consists of those boxes in $[m]\times [n]$ which lie above the diagonal.

\begin{remark}
Since $m,n$ are relatively prime, there are no integer solutions to $nx+my=nm$ with $(x,y)\in [m]\times [n]$. This implies that the strict inequalities $nx+my<nm$ and $nx'+my'>nm$ are equivalent to the corresponding weak inequalities for boxes in $[m]\times [n]$.
\end{remark}

Note that the cardinality of $[m]\times[n]$ is $mn$, the cardinality of $\mathrm{diag}_{m,n}$ is $n+m-1$, and $|\tau_{m,n}|=|\tau_{m,n}^\vee|$ by symmetry, which forces
\[
|\tau_{m,n}| = \frac{(m-1)(n-1)}{2}.
\]

The following lemma is elementary but quite useful.

\begin{lemma}\label{lemma:mn lemma}
Let $m,n\geq 1$ be coprime integers.  Suppose we are given integers $x,x',y,y'\in \Z$ such that $nx+my=nx'+my'$.  Then $y,y'\in [n]$ implies $(x,y)=(x',y')$ (by symmetry $x,x'\in [m]$ implies $(x,y)=(x',y')$ as well).
\end{lemma}
\begin{proof}
Suppose $nx+my=nx'+my'$ for some boxes $(x,y)$ and $(x',y')$ in $Z\times[n]$.  This implies that $|y-y'|$ is divisible by $n$, and $1\leq y,y'$ forces $|y-y'|<n$, which then forces $y-y'=0$.  This then implies $nx=nx'$, hence $x=x'$.
\end{proof}

\begin{definition}[\cite{Anderson}]\label{def:anderson label}
For coprime integers $m,n\geq 1$,  the \newword{$(m,n)$-Anderson filling} is the function $\gamma_{m,n}:\Z^2 \to \Z$ defined by 
\[\gamma_{m,n}(x,y) = mn-nx-my.\]
Given a box $\Box$ with NE corner $(x,y)$, we refer to $\gamma_{m,n}(\Box):=\gamma_{m,n}(x,y)$ as the \newword{$(m,n)$-Anderson label of $\Box$}.  When $m,n$ are understood, we will write $\gamma=\gamma_{m,n}$ and use the terms Anderson filling and label unambigiously.
\end{definition}

\begin{definition}
    Let $\Gamma_{m,n}$ denote the additive submonoid of $\Z$ generated by $m$ and $n$.  That is, $\Gamma_{m,n}=\{k\in \Z\:|\: k = nx+my \text{ for some $x,y\in \Z_{\geq 0}$}\}$. $\Gamma_{m,n}$ is also called the \newword{semigroup generated by $m$ and $n.$}
\end{definition}

The main relevant properties of the Anderson filling are stated in the lemma below.  Before we state it, we make some very simple observations.  First, the fact that $\gamma(x-m,y+n)=\gamma(x,y)$ means that $\gamma$ descends to a well-defined function on the quotient $\gamma'\colon \Z\times \Z / \langle (-m,n)\rangle\to \Z$.  Secondly, modulo the equivalence relation $(x,y)\sim (x-m,y+n)$, every cell in $\Z\times \Z$ is equivalent to a unique cell in $\Z\times [n]$.  That is to say, the obvious composition
\begin{equation}\label{eq:strip and quotient}
\Z\times[n] \hookrightarrow \Z\times \Z\twoheadrightarrow (\Z\times \Z)/\langle(-m,n)\rangle
\end{equation}
is a bijection.

\begin{lemma}\label{lem:Anderson}
For coprime integers $m,n\geq 1$, the Anderson filling $\gamma=\gamma_{m,n}$ satisfies:
\begin{enumerate}
    \item $\gamma(\Box)>0$ if and only if the box $\Box$ is strictly below the line $nx+my=nm$.
    \item $\gamma$ induces a bijection  $\gamma'\colon \Z\times \Z / \langle (-m,n)\rangle\to \Z$.
    \item $\gamma$ restricts to a bijection $\Z\times [n]\to \Z$.
    \item we have $(x,y)\in \Z_{\leq 0}\times [n]$ if and only if $\gamma(x,y)\in\Gamma_{m,n}.$
\end{enumerate}
\end{lemma}
\begin{proof}
Statement (1) is clear from the definitions.

For statement (2), observe that $\gamma'$ is surjective since $\gamma$ is.  To prove that $\gamma'$ is injective, note that modulo $(-m,n)$, every $(x,y)\in \Z^2$ is equivalent to a cell in $\Z\times[n]$, and Lemma \ref{lemma:mn lemma} tells us that $\gamma$ restricted to $\Z\times [n]$ is injective.  This proves statement (2), while precomposing with the bijection \eqref{eq:strip and quotient} proves statement (3).

For statement (4), observe that
\begin{itemize}
    \item $(\gamma')\iinv(0)=(-1,n)$,
    \item $(\gamma')\iinv(k+n)=(\gamma')\iinv(k)+(-1,0)$,
    \item $(\gamma')\iinv(k+m)=(\gamma')\iinv(k)+(0,-1)$ (modulo $(k,0)\sim (k-m,n)$).
\end{itemize}
These observations make it clear that $(\gamma')\iinv(\Gamma_{m,n})$ equals $\Z_{\geq 0}\times [n]$, which completes the proof of statement (4).
\end{proof}

\begin{definition}\label{def:associated lattice path}
When $(r,s)=(m,n)\in\mathbb{Z}^2_{>0}$ are integers, it is natural to identify an $m,n$-Dyck path $\pi=(\lambda,\tau_{m,n})$ with the lattice path $\tilde{\pi}$ from $(m,0)$ to $(0,n)$ following the boundary of $\lambda$ (hence the name `Dyck path'). To avoid confusion, we will call $\tilde{\pi}$ the \newword{associated lattice path} of $\pi$. 
\end{definition}

\begin{remark}
    Note that the associated lattice path depends on the choice of the integers $(m,n).$ For instance, one has $\tau_{n,n+1}=\tau_{n+1,n}$ and $\D(n,n+1)=\D(n+1,n),$ but the associated lattice paths are slightly different.
\end{remark}

\subsection{The $\dinv$ statistic}\label{def:dinv}
The \newword{dinv} statistic was introduced by Mark Haiman in the early $2000$'s \cite{Haiman,Haiman-Vanishing}, but didn't appear in print until $2009$, at least not in full generality.

\begin{definition}[\cite{LW}]
    Given any partition $\lambda$, define $\dinv_{p}(\lambda)$ to be the cardinality of the following set:
    \begin{equation*}
        \dinv_{p}(\lambda):= \Big\vert \left\{\Box\in \lambda:\frac{\leg(\Box)}{\arm(\Box)+1}<p\le\frac{\leg(\Box)+1}{\arm(\Box)}\right\} \Big\vert,
    \end{equation*}
    where  $\leg(\Box)$ and $\arm(\Box)$ denote the numbers of boxes in $\lambda$ that lie strictly above and in the same column and strictly right and in the same row as $\Box$, respectively (see Figure \ref{fig:dinv}).    
\end{definition}

\begin{figure}
\begin{tikzpicture}[scale=.7,every node/.style={scale=.7}]
\draw[fill = blue!60!white] (0,0) rectangle (1,5);
\draw[fill = red!60!white] (0,0) rectangle (7,1);
\draw[fill = yellow] (0,0) rectangle (1,1);
 \draw[very thin, gray] (0,0) grid (8,6);
\draw [thick](0,0) grid (1,5);
\draw [thick](0,0) grid (7,1);
\draw[thick, purple] (7,1) -- (1,6);
\draw[thick, orange] (8,1) -- (1,5);
\node[rotate=90] at (-.5,3){$\overbrace{\hspace{40mm}}$};
\node at (-1.2,3)[scale=1.3]{leg};
\node at (4,-.5){$\underbrace{\hspace{60mm}}$};
\node at (4,-1)[scale=1.3]{arm};
\node at (.5,.5)[scale=1.3]{$c$};
\end{tikzpicture}
\caption{The arm and leg of a box $c$ in $\Z^2$ denoted in red and blue, respectively. The orange and purple diagonal lines denote the lines with slopes $
\frac{\leg(c)}{\arm(c)+1}$ and $\frac{\leg(c)+1}{\arm(c)}$ used to determine the contribution of $c$ towards $\dinv$.  }\label{fig:dinv}
\end{figure}
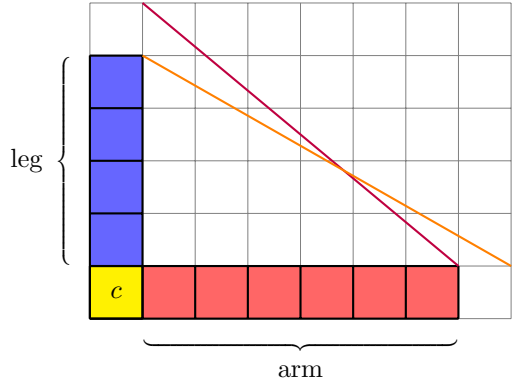

\begin{definition}
    Let $(r,s)\in\mathbb{R}_{>0}^2$ be a pair of positive real numbers. Then for an $r,s$-Dyck path $\pi=(\lambda,\tau_{r,s})$, set
    \begin{equation*}
        \dinv_{s/r}(\pi):=\dinv_{s/r}(\lambda).
    \end{equation*}
\end{definition}

Note that the statistic $\dinv_{s/r}$ depends not only on the partition $\tau_{r,s}$ and the set $\D(r,s),$ but also on the slope $s/r.$ For instance, one has $\tau_{n,n+1}=\tau_{n+1,n}$ and $\D(n,n+1)=\D(n+1,n),$ but $\dinv_{n/(n+1)}$ and $\dinv_{(n+1)/n}$ are two different statistics. However, when the slope $s/r$ (or $n/m$) is understood from the context, we will skip the subscript and use $\dinv$ instead of $\dinv_{s/r}$ (or $\dinv_{n/m}$).

\subsection{Invariant Subsets}

Let $(m,n)\in\mathbb{Z}^2_{>0}$ be a pair of relatively prime integers. In this case, it will be more convenient for us to use an alternative definition of the $\dinv=\dinv_{n/m}$ statistic originally inspired by the geometry of compactified Jacobians and developed in \cite{GM13}. See also \cite{Ma} and \cite{GMV17} for a simpler proof of the equivalence of the two definitions of $\dinv.$ To do so, we need some additional combinatorial tools. We follow the notations from \cite{GMV20}.

\begin{remark}
    Note that in \cite{GMV20}, the more general case where $m$ and $n$  are not necessarily relatively prime is considered. However, we only need the relatively prime case in this paper, and therefore we simplify the notations accordingly. 
\end{remark}

\begin{definition}\label{def:invariant subsets}
For each subset $\Delta \in \Z_{\geq 0}$, let $\Delta+n :=\{ k \in \Z \:|\: k-n \in \Delta\}$ and denote by $\overline{\Delta} = \Z_{\ge 0} \setminus \Delta$, the set complement of $\Delta$.  We say $\Delta$ is an \newword{$(m,n)$-invariant set} if $\Delta+n \subset \Delta$, $\Delta+m \subset \Delta.$ We will denote the set of all $(m,n)$-invariant subsets by $\Inv_{m,n}$ and call the elements of $\overline{\Delta}$ the \newword{gaps} in $\Delta$. 
In the case that $0 \in \Delta$, we say $\Delta$ is \newword{0-normalized} and write $\Inv_{m,n}^{0}$ for the collection of all such sets. 
\end{definition}

The sets of \newword{$n$-generators, $m$-generators} and \newword{co-generators} of any $\Delta \in \Inv_{m,n}$ are defined as follows:
\begin{align*}
\ngen(\Delta)&:= \Delta \setminus (\Delta+n) = \{ k \in \Delta \:|\: k-n \not\in \Delta\},  \\
\mgen(\Delta)&:= \Delta \setminus (\Delta+m) = \{ k \in \Delta \:|\: k-m \not\in \Delta\},\\
\cogen(\Delta) &:= \{ k \in \mathbb{Z}\setminus\Delta \:|\: k+n \in \Delta, k+m \in \Delta \}\\
\cogennn(\Delta) &:=\cogen(\Delta)\cap\mathbb{Z}_{\ge 0}.
\end{align*}
In particular, there is precisely one generator per equivalence class (mod $n$ or $m$), so that $|\ngen| = n$ and $|\mgen| = m$. 

\begin{definition} \label{def:area-dinv}
Let $[k,k+L]=\{k, k+1,\dots, k+L\} \in \Z$. For any $\Delta \in \Inv_{m,n}$, set:
\begin{align*}
\area(\Delta)&:= | \overline{\Delta}|,\\
\codinv(\Delta) &:= \sum_{k \in \ngen(\Delta)} |[k,k+m-1] \cap \overline{\Delta}| = \sum_{k \in \ngen(\Delta)} |\{ k \leq x < k+m \:|\: x \not\in \Delta\}|,
\end{align*}
so that area counts the number of gaps of $\Delta$ and codinv the number of gaps in length $m$ intervals whose starting value is an $n$-generator. We also define:
\begin{align*}
\delta(m,n)&:= \frac{(m-1)(n-1)}{2}, \\
\coarea(\Delta)&:= \delta(m,n) - \area(\Delta),\\
 \dinv(\Delta)&:= \delta(m,n) - \codinv(\Delta).
\end{align*}
\end{definition}

Note that the definition of $\d(m,n)$ above agrees with Definition \ref{def:delta MN}.  An $\area$ and $\dinv$ preserving bijection between $\D(m,n)$ and $\Inv_{m,n}^0$ was introduced in \cite{GM13}:

\begin{definition}[\cite{GM13}]\label{def:dyck-inv-bijection}
Let $\pi=(\lambda,\tau_{m,n})$ be a $m,n$-Dyck path, then the set of \newword{gaps} of $\pi$ is given by
\begin{equation*}
    \Gaps(\pi):=\{\gamma_{m,n}(\Box)|\Box\in\tau_{m,n}\setminus\lambda\}.
\end{equation*}
    The bijection $\Aa_{m,n}:\D(m,n)\to \Inv_{m,n}^0$ is then defined by
    \begin{equation*}
        \Aa_{m,n}(\pi):=\mathbb{Z}_{\ge 0}\setminus \Gaps(\pi).
    \end{equation*}
 \end{definition}   
 Notice that  $\Aa_{m,n}$ depends on the choice of the relatively prime integers $(m,n),$ not only on the partition $\tau_{m,n}$ and the set $\D(m,n).$ Nonetheless, whenever $(m,n)$ are clear from the context, we will skip the subscript and write $\Aa=\Aa_{m,n}.$
 
\begin{definition}\label{def:addable}
    For a partition $\lambda$, a box $\Box \not\in \lambda$ is called \newword{addable} if $\lambda \cup \Box$ is a partition diagram. Equivalently, $\Box$ is addable if both its south and west boundaries intersect the boundary of $\lambda.$ Thus, a box is addable for a Dyck path $\pi=(\lambda,\tau)$ if it is addable for $\lambda$. 
\end{definition}

Let $\tilde{\pi}$ be the associated lattice path of $\pi=(\lambda,\tau_{m,n})\in\D(m,n).$ Then under this bijection, we have the following correspondences:
\begin{align*}
&\{ \gamma (x,y) \;|\; \text{ the top boundary of the box } (x,y) \text{ is a horizontal step of } \tilde{\pi}\}= \mgen( \Aa(\pi)),\\
&\{ \gamma (x,y)\;|\; \text{ the right boundary of the box } (x,y) \text{ is the a vertical step of } \tilde{\pi}\}=  \ngen( \Aa(\pi)),\\
&\{ \gamma(x,y)\
\;|\; (x,y) \text{ is an addable box of } \lambda\} =  \cogen( \Aa(\pi)).
\end{align*}

See Figure \ref{fig:boxes to generators} for an example.

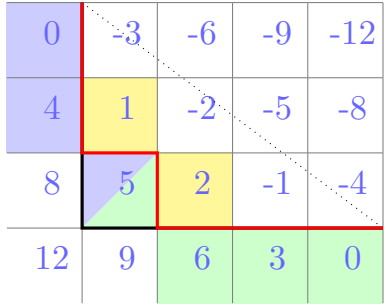
\begin{figure}
    \centering
    \begin{tikzpicture}[scale=1]
\draw (0,3.2) node {};
\fill [fill=blue!20!white] (-1,2) rectangle (0,3);
\fill [fill=blue!20!white] (-1,1) rectangle (0,2);
\fill [fill=blue!20!white] (0,0)--(0,1)--(1,1)--(0,0);
\fill [fill=green!20!white] (1,-1) rectangle (2,0);
\fill [fill=green!20!white] (2,-1) rectangle (3,0);
\fill [fill=green!20!white] (3,-1) rectangle (4,0);
\fill [fill=green!20!white] (0,0)--(1,1)--(1,0)--(0,0);
\fill [fill=yellow!50!white] (0,1) rectangle (1,2);
\fill [fill=yellow!50!white] (1,0) rectangle (2,1);
 \pgfmathtruncatemacro{\m}{3} 
  \pgfmathtruncatemacro{\n}{4} 
 \pgfmathtruncatemacro{\mi}{\m-1} 
    \pgfmathtruncatemacro{\ni}{\n-1} 
  \foreach \x in {0,...,\n} 
    \foreach \y in {0,...,\m}     
       {\pgfmathtruncatemacro{\label}{ - \m * \x - \n *  \y +\m*\n} 
       \node  at (\x-.4,\y-.4) [blue!60!white,scale=1.2]{ \label};} 
\draw[dotted] (0,\m)--(\n,0); 
  \foreach \x in {0,...,\n} 
    \foreach \y in {-1,...,\mi}   
      \draw[gray] (\x,\y)--(\x,\y+1);
        \foreach \y in {0,...,\m}     
       \foreach \x in {-1,...,\ni}   
         \draw[gray] (\x,\y)--(\x+1,\y);
\draw [very thick] (4,0)--(0,0)--(0,3); 
\draw[ very thick,red](0,3)--(0,1)--(1,1)--(1,0)--(4,0); 

\end{tikzpicture}

    \caption{The Dyck path $\pi=((1),(2,1))\in\D(4,3)$ with the path $\tilde{\pi}$ denoted in red. The set $\Gaps(\pi)=\{1,2\}$ obtained from $(2,1) \setminus (1)$ is highlighted in yellow and coincide in this case with the addable boxes, or co-generators, of $\pi$. The boxes corresponding to the $4$-generators of $\Aa(\pi)=\mathbb{Z}_{\ge 0}\setminus\Gaps(\pi)$ are shaded green, the boxes corresponding to the $3$-generators are shaded blue. Note that $5$ and $0$ are both $3$ and $4$-generators.}
    \label{fig:boxes to generators}
\end{figure}

\begin{theorem}[\cite{GM13}] \label{thm:dinvA=dinv}
    The bijection $\Aa:\D(m,n)\to \Inv^0_{m,n}$ is $\area$ and $\dinv$ preserving:
    \begin{equation*}
      \area(\Aa(\pi) ) = \area(\pi) \qquad \text{and} \qquad \dinv(\Aa(\pi))=\dinv_{n/m}(\pi).
    \end{equation*}
\end{theorem}

We follow \cite{GM13,GMV20} and define the $(q,t)$-Catalan and Schr\"oder polynomials as follows.

\begin{definition} For any coprime $m,n \in \Z_{>0}$, the \newword{rational $(q,t)$-Catalan polynomial} $ C_{m,n}(q,t)$ is given by:
 \begin{equation}\label{formula: Catalan definition}
        C_{m,n}(q,t):=\sum_{\Delta\in\Inv_{m,n}^0}q^{\area(\Delta)}t^{\dinv(\Delta)}=(1-q)\sum_{\Delta\in\Inv_{m,n}}q^{\area(\Delta)}t^{\dinv(\Delta)}.
    \end{equation}
\end{definition}

For each $k \in  \Z$, let  
\begin{equation}\label{eq:xi}
\xi_k(\Delta): = |\{ x \in \ngen(\Delta) | n+k+1\leq x \leq k+n+m\}|.
\end{equation}

\begin{definition}
For each coprime pair $(m,n)$, the \newword{Schr\"oder polynomial} $S_{m,n}(q,t,a)$ is:
\begin{equation}\label{formula: Schroder definition}
S_{m,n}(q,t,a):=\sum_{\Delta\in \Inv_{m,n}^0} q^{\area(\Delta)}t^{\dinv(\Delta)}\prod_{k\in\cogen(\Delta)}\left(1+at^{-\xi_k(\Delta)}\right).
\end{equation}
\end{definition}

It is evident that at $a=0$, the Schr\"oder polynomial recovers the rational $(q,t)$-Catalan polynomial. 

\begin{remark}
    In the case when $m$ and $n$ are not relatively prime, the set $\Inv_{m,n}^0$ is infinite. Still, the same formulas \eqref{formula: Catalan definition} and \eqref{formula: Schroder definition} yield well-defined power series. In \cite{GMV20} the definition of the rational $(q,t)$-Catalan and -Schr\"oder series include the non-relatively prime case. However, for the purposes of this paper, it is enough to consider the relatively prime case only.
\end{remark}

\begin{example}
In Figure \ref{figure: 4,3 Schroder} we can see the bijection between the $4,3$-Dyck paths and the $0$-normalized $(4,3)$-invariant subsets, with the 
 the $\area$ and $\dinv$ statistics, and additional information to compute the Schr\"oder polynomial explicitly computed in each case. Summing all the corresponding factors, we get the Schr\"oder polynomial $S_{4,3}(q,t,a)$. Namely,    
    \begin{align*}
    S_{4,3}(q,t,a)=&t^3(1+a)(1+at^{-1})(1+at^{-2})+qt^2(1+a)(1+at^{-1})+qt(1+a)(1+at^{-1})\\
    &\ \ \ \ \ \ \ \ +q^2t(1+a)(1+at^{-1})+q^3(1+a)\\
    =&q^3 + q^2t + qt^2 + t^3 + qt + a(q^3 + q^2t + qt^2 + t^3 + q^2 + 2qt + t^2 + q + t)\\
    &\ \ \ \ \ \ \ \ +a^2(q^2 + qt + t^2 + q + t + 1) + a^3.
    \end{align*}
\end{example}

\begin{figure}[ht]
\center
\begin{tabular}{|c|c|c|c|}
\hline
\begin{tikzpicture}[scale=0.5]
\draw (0,3.2) node {};
\fill [fill=gray!20!white] (0,0) rectangle (1,2);
\fill [fill=gray!20!white] (0,0) rectangle (2,1);%
 \pgfmathtruncatemacro{\m}{3} 
  \pgfmathtruncatemacro{\n}{4} 
 \pgfmathtruncatemacro{\mi}{\m-1} 
    \pgfmathtruncatemacro{\ni}{\n-1} 
  \foreach \x in {0,...,\n} 
    \foreach \y in {1,...,\m}     
       {\pgfmathtruncatemacro{\label}{ - \m * \x - \n *  \y +\m*\n} 
       \node  at (\x-.4,\y-.4) [blue!60!white, scale=.8]{\footnotesize \label};} generates the grid and Anderson labels
\draw[dotted] (0,\m)--(\n,0); 
  \foreach \x in {0,...,\n} 
    \foreach \y  [count=\yi] in {0,...,\mi}   
      \draw (\x,\y)--(\x,\yi);
        \foreach \y in {0,...,\m}     
       \foreach \x  [count=\xi] in {-1,...,\ni}   
         \draw (\x,\y)--(\xi-1,\y);
\draw [very thick] (4,0)--(0,0)--(0,3); 
\draw[ very thick,red](0,3)--(0,2)--(1,2)--(1,1)--(2,1)--(2,0)--(4,0); 

\end{tikzpicture}
& 
\begin{tikzpicture}[scale=0.5]
\draw (0,5) node {$\Gaps=\emptyset$};
\draw (0,4) node {$\Delta=\mathbb{Z}_{\ge 0}$};
\draw (0,3) node {$\area(\Delta)=0$}; 
\end{tikzpicture}
& 
\begin{tikzpicture}[scale=0.5]
\draw (0,5) node {$3$-$\gen=\{0,1,2\}$};
\draw (0,4) node {$\codinv(\Delta)=0$};
\draw (0,3) node {$\dinv(\Delta)=3$};
\end{tikzpicture}
& 
\begin{tikzpicture}[scale=0.5]
\draw (0,5) node {$\cogen=\{-3,-2,-1\}$};     
\draw (0,4) node {\footnotesize $(1+a)(1+at^{-1})(1+at^{-2})$}; 
\draw (0,3) node {};
\end{tikzpicture}
\\
\hline
\begin{tikzpicture}[scale=0.5]
\draw (0,3.2) node {};
\fill [fill=yellow!50!white] (0,1) rectangle (1,2);
\fill [fill=gray!20!white] (0,0) rectangle (2,1);%
 pick an m and n HERE
 \pgfmathtruncatemacro{\m}{3} 
  \pgfmathtruncatemacro{\n}{4} 
 \pgfmathtruncatemacro{\mi}{\m-1} 
    \pgfmathtruncatemacro{\ni}{\n-1} 
  \foreach \x in {0,...,\n} 
    \foreach \y in {1,...,\m}     
       {\pgfmathtruncatemacro{\label}{ - \m * \x - \n *  \y +\m*\n} 
       \node  at (\x-.4,\y-.4) [blue!60!white, scale=.8]{\footnotesize \label};} 
generates the grid and Anderson labels
\draw[dotted] (0,\m)--(\n,0); 
  \foreach \x in {0,...,\n} 
    \foreach \y  [count=\yi] in {0,...,\mi}   
      \draw (\x,\y)--(\x,\yi);
        \foreach \y in {0,...,\m}     
       \foreach \x  [count=\xi] in {-1,...,\ni}   
         \draw (\x,\y)--(\xi-1,\y);
\draw [very thick] (4,0)--(0,0)--(0,3); 
\draw[ very thick,red](0,3)--(0,1)--(2,1)--(2,0)--(4,0); 

\end{tikzpicture}
& 
\begin{tikzpicture}[scale=0.5]
\draw (0,5) node {$\Gaps=\{1\}$};
\draw (0,4) node {$\Delta=\mathbb{Z}_{\ge 0}\setminus\{1\}$};
\draw (0,3) node {$\area(\Delta)=1$}; 
\end{tikzpicture}
& 
\begin{tikzpicture}[scale=0.5]
\draw (0,5) node {$3$-$\gen=\{0,2,4\}$};
\draw (0,4) node {$\codinv(\Delta)=1$};
\draw (0,3) node {$\dinv(\Delta)=2$};
\end{tikzpicture}
& 
\begin{tikzpicture}[scale=0.5]
\draw (0,5) node {$\cogen=\{-1,1\}$};     
\draw (0,4) node {$(1+a)(1+at^{-1})$}; 
\draw (0,3) node {};
\end{tikzpicture}
\\
\hline
\begin{tikzpicture}[scale=0.5]
\draw (0,3.2) node {};
\fill [fill=gray!20!white] (0,0) rectangle (1,2);
\fill [fill=yellow!40!white] (1,0) rectangle (2,1);%
 \pgfmathtruncatemacro{\m}{3} 
  \pgfmathtruncatemacro{\n}{4} 
 \pgfmathtruncatemacro{\mi}{\m-1} 
    \pgfmathtruncatemacro{\ni}{\n-1} 
  \foreach \x in {0,...,\n} 
    \foreach \y in {1,...,\m}     
       {\pgfmathtruncatemacro{\label}{ - \m * \x - \n *  \y +\m*\n} 
       \node  at (\x-.4,\y-.4) [blue!60!white, scale=.8]{\footnotesize \label};} 
\draw[dotted] (0,\m)--(\n,0); 
  \foreach \x in {0,...,\n} 
    \foreach \y  [count=\yi] in {0,...,\mi}   
      \draw (\x,\y)--(\x,\yi);
        \foreach \y in {0,...,\m}     
       \foreach \x  [count=\xi] in {-1,...,\ni}   
         \draw (\x,\y)--(\xi-1,\y);
\draw [very thick] (4,0)--(0,0)--(0,3); 
\draw[ very thick,red](0,3)--(0,2)--(1,2)--(1,0)--(4,0); 

\end{tikzpicture}
& 
\begin{tikzpicture}[scale=0.5]
\draw (0,5) node {$\Gaps=\{2\}$};
\draw (0,4) node {$\Delta=\mathbb{Z}_{\ge 0}\setminus\{2\}$};
\draw (0,3) node {$\area(\Delta)=1$}; 
\end{tikzpicture}
& 
\begin{tikzpicture}[scale=0.5]
\draw (0,5) node {$3$-$\gen=\{0,1,5\}$};
\draw (0,4) node {$\codinv(\Delta)=2$};
\draw (0,3) node {$\dinv(\Delta)=1$};
\end{tikzpicture}
& 
\begin{tikzpicture}[scale=0.5]
\draw (0,5) node {$\cogen=\{-3,2\}$};     
\draw (0,4) node {$(1+a)(1+at^{-1})$}; 
\draw (0,3) node {};
\end{tikzpicture}
\\
\hline
\begin{tikzpicture}[scale=0.5]
\draw (0,3.2) node {};
\fill [fill=gray!20!white] (0,0) rectangle (1,1);
\fill [fill=yellow!40!white] (0,1) rectangle (1,2);
\fill [fill=yellow!40!white] (1,0) rectangle (2,1);
 \pgfmathtruncatemacro{\m}{3} 
  \pgfmathtruncatemacro{\n}{4} 
 \pgfmathtruncatemacro{\mi}{\m-1} 
    \pgfmathtruncatemacro{\ni}{\n-1} 
  \foreach \x in {0,...,\n} 
    \foreach \y in {1,...,\m}     
       {\pgfmathtruncatemacro{\label}{ - \m * \x - \n *  \y +\m*\n} 
       \node  at (\x-.4,\y-.4) [blue!60!white, scale=.8]{\footnotesize \label};} 
\draw[dotted] (0,\m)--(\n,0); 
  \foreach \x in {0,...,\n} 
    \foreach \y  [count=\yi] in {0,...,\mi}   
      \draw (\x,\y)--(\x,\yi);
        \foreach \y in {0,...,\m}     
       \foreach \x  [count=\xi] in {-1,...,\ni}   
         \draw (\x,\y)--(\xi-1,\y);
\draw [very thick] (4,0)--(0,0)--(0,3); 
\draw[ very thick,red](0,3)--(0,1)--(1,1)--(1,0)--(4,0); 

\end{tikzpicture}
& 
\begin{tikzpicture}[scale=0.5]
\draw (0,5) node {$\Gaps=\{1,2\}$};
\draw (0,4) node {$\Delta=\mathbb{Z}_{\ge 0}\setminus\{1,2\}$};
\draw (0,3) node {$\area(\Delta)=2$}; 
\end{tikzpicture}
& 
\begin{tikzpicture}[scale=0.5]
\draw (0,5) node {$3$-$\gen=\{0,4,5\}$};
\draw (0,4) node {$\codinv(\Delta)=2$};
\draw (0,3) node {$\dinv(\Delta)=1$};
\end{tikzpicture}
& 
\begin{tikzpicture}[scale=0.5]
\draw (0,5) node {$\cogen=\{1,2\}$};     
\draw (0,4) node {$(1+a)(1+at^{-1})$}; 
\draw (0,3) node {};
\end{tikzpicture}
\\
\hline
\begin{tikzpicture}[scale=0.5]
\draw (0,3.2) node {};
\fill [fill=yellow!40!white] (0,0) rectangle (1,2);
\fill [fill=yellow!40!white] (0,0) rectangle (2,1);
 \pgfmathtruncatemacro{\m}{3} 
  \pgfmathtruncatemacro{\n}{4} 
 \pgfmathtruncatemacro{\mi}{\m-1} 
    \pgfmathtruncatemacro{\ni}{\n-1} 
  \foreach \x in {0,...,\n} 
    \foreach \y in {1,...,\m}     
       {\pgfmathtruncatemacro{\label}{ - \m * \x - \n *  \y +\m*\n} 
       \node  at (\x-.4,\y-.4) [blue!60!white, scale=.8]{\footnotesize \label};} 
\draw[dotted] (0,\m)--(\n,0); 
  \foreach \x in {0,...,\n} 
    \foreach \y  [count=\yi] in {0,...,\mi}   
      \draw (\x,\y)--(\x,\yi);
        \foreach \y in {0,...,\m}     
       \foreach \x  [count=\xi] in {-1,...,\ni}   
         \draw (\x,\y)--(\xi-1,\y);
\draw [very thick] (4,0)--(0,0)--(0,3); 
\draw [very thick, red] (4,0)--(0,0)--(0,3); 

\end{tikzpicture}
& 
\begin{tikzpicture}[scale=0.5]
\draw (0,5) node {$\Gaps=\{1,2,5\}$};
\draw (0,4) node {$\Delta=\mathbb{Z}_{\ge 0}\setminus\{1,2,5\}$};
\draw (0,3) node {$\area(\Delta)=3$}; 
\end{tikzpicture}
& 
\begin{tikzpicture}[scale=0.5]
\draw (0,5) node {$3$-$\gen=\{0,4,8\}$};
\draw (0,4) node {$\codinv(\Delta)=3$};
\draw (0,3) node {$\dinv(\Delta)=0$};
\end{tikzpicture}
& 
\begin{tikzpicture}[scale=0.5]
\draw (0,5) node {$\cogen=\{5\}$};     
\draw (0,4) node {$(1+a)$}; 
\draw (0,3) node {};
\end{tikzpicture}
\\
\hline
\end{tabular}
\caption{All $4,3$-Dyck paths are on the left, with the area boxes in yellow. The corresponding $(4,3)$-invariant subsets are in the second column, together with the $\area$, $3$-generators, $\codinv,$ and $\dinv$ in the third column; the co-generators and contributing factors of $S_{4,3}(q,t,a) $ are given in the fourth column.}\label{figure: 4,3 Schroder}
\end{figure}

\subsection{Recursions for Invariant Subsets}
In \cite{GMV20} Gorsky-Mazin-Vazirani introduced certain power series indexed by invariant subsets and proved that they satisfy recursions similar to the Hogancamp-Mellit recursions \cite{HM}, thus providing an alternative model for computing the Khovanov--Rozansky homology of torus links. We now recall this construction. 

For any sequence ${\bf w}=(w_0,\dots,w_{m+n-1}) \in \{0,1\}^{n+m}$, set 
\begin{equation}\label{eq:Inv-w}
 \Inv_{\bf w}:= \{ \Delta \in \Inv_{m,n} \;|\; i \in \Delta \cap [0,n+m-1] \text{ if and only if }w_i =1\}.
 \end{equation}
A sequence $\bf w$ is \newword{admissible} if $\Inv_{\bf w} \neq \emptyset$. 

\begin{definition}\label{def:area' codinv'}
For any $\Delta \in \Inv_{m,n}$, define the modified statistics:
\begin{align*}
\area'(\Delta)&:= | \overline{\Delta} \cap \Z_{\geq n+m}|,\\
\codinv'(\Delta)&:= \sum_{x \in \ngen(\Delta)}  | \overline{\Delta} \cap \Z_{\geq n+m} \cap [x,x+m-1]| - {\binom{\xi_{-1}(\Delta)}{2}}.
\end{align*}
\end{definition}

In particular, for $\Delta \in \Inv_{0^{m+n}}$ we have the identities:
\begin{align*}
\area'(\Delta) = -(n+m) + \area(\Delta), && \codinv'(\Delta) = \codinv(\Delta).
\end{align*}

Given a binary sequence ${\bf w}=(w_0,\dots,w_{m+n-1}) \in \{0,1\}^{m+n}$, we now construct two trinary sequences ${\bf x} \in \{0,1,\bullet\}^m$ and ${\bf y} \in \{0,1,\bullet\}^n$ such that,
\begin{itemize}[leftmargin=*]
\item[-]  ${\bf x}$ records gaps (by $0$), $n$-generators (by $1$), and any other element in $\Delta \cap [n,n+m-1]$ (by $\bullet$), and
\item[-] ${\bf y}$ records gaps (by $0$), $m$-generators (by $1$), and any other element in $\Delta \cap [m,n+m-1]$ (by $\bullet$). 
\end{itemize} 

\begin{definition}\label{def:trinaryseq} 
Given an admissible sequence ${\bf w} \in \{0,1\}^{m+n}$, define ${\bf x} = (x_0, \dots x_{m-1})$ and ${\bf y} = (y_0, \dots y_{n-1})$ by setting:
\begin{align*}
x_i := 
\begin{cases}
0 & ; w_{n+i} =0,\\
\bullet &; w_{n+i} = w_i=1,\\
1 &; w_{n+i} = 1 \text{ and } w_i = 0,
\end{cases}
&\text{ and }
y_i := 
\begin{cases}
0 & ; w_{m+i} =0,\\
\bullet &; w_{m+i} = w_i=1,\\
1 &; w_{m+i} = 1 \text{ and } w_i = 0.
\end{cases}
\end{align*}
We say a pair of sequences ${\bf x, y}$ are admissible if and only if ${\bf w}$ is admissible. 
\end{definition}

\begin{example} \label{ex:7-4sequences}
Let $(m,n)=(7,4)$ and $\Delta=\mathbb{Z}_{\ge 0}\setminus\{0,1,2,3,4,6,7,8,10\}.$
Then, the associated binary sequence $w=00000100010$ yields the ternary sequences ${\bf x}=01000\bullet 0$ and ${\bf y} =0010$,
since the only $4$-generator in $[4,10]$ is $5$ and the only $7$-generator in $[7,10]$ is $9$. In particular, $9$ is not a $4$-generator since $9-4=5 \in \Delta$. Thus, $ \Delta\in I_{00000100010}= I_{01000\bullet 0, 0010}$.
\end{example}

\begin{definition}\label{def:rec poly}
Given admissible trinary sequences $\bf x$ and $\bf y$ as above, set $\Inv_{\bf x,y} := \Inv _{\bf w}$ and let
\begin{equation*}
Q_{{\bf x},{\bf y}}(q,t,a):=\sum_{\Delta\in \Inv_{{\bf x},{\bf y}}} q^{\area'(\Delta)}t^{-\codinv'(\Delta)}\prod_{k\in\cogennn(\Delta)}(1+at^{-\xi_k(\Delta)}).
\end{equation*}
\end{definition} 

Note that one only considers non-negative co-generators in Definition \ref{def:rec poly}. However, if $\min(\Delta)\ge m+n-1\ge\min(m,n)$, then there are no negative co-generators and $\cogen(\Delta) = \cogen_{\geq 0}(\Delta)$.  In particular, one immediately obtains the following result.

\begin{lemma}\cite[Rem. 5.6]{GMV20}
    The Schr\"oder polynomial arises as a special case of $Q_{\bf x,y}$,
\begin{equation*}
    S_{m,n}(q,t,a)= {(1-q)t^{\delta(m,n)}}Q_{0^{n},0^{n}}(q,t,a)=t^{\delta(m,n)} Q_{0^{m-1}1,0^{n-1}1}(q,t,a).
\end{equation*}
\end{lemma}

For any sequence ${\bf x}$ of length $\ell$ and $\star \in \{0,1,\bullet\}$, let ${\bf x} \star = (x_0,\dots, x_\ell, \star)$ and $ \star {\bf x}= (\star, x_0,\dots, x_\ell )$.

\begin{theorem}\cite[Thm. 4.10 and 4.13]{GMV20} \label{thm:Qrecursion}
The following recursions hold:
\begin{align*}
Q_{0{\bf x},0{\bf y}}&=t^{-|{\bf x}|}Q_{{\bf x}1,{\bf y}1}+qt^{-|{\bf x}|}Q_{{\bf x}0,{\bf y}0}, &
Q_{1{\bf x},0{\bf y}}&=Q_{{\bf x}1,{\bf y}\bullet},&
Q_{\bullet{\bf x},\bullet{\bf y}}&=Q_{{\bf x}\bullet,{\bf y}\bullet},\\
Q_{1{\bf x},1{\bf y}}&=(t^{|{\bf x}|}+a)Q_{{\bf x}\bullet,{\bf y}\bullet},&
Q_{0{\bf x},1{\bf y}}&=Q_{{\bf x}\bullet,{\bf y}1},&
Q_{\emptyset,\emptyset}&=1.
\end{align*}
Thus, letting $(\bf{u,v})$ be the sequences obtained from $(\bf{x,y})$ by omitting all $\bullet$'s we have, 
\begin{equation*}
R_{{\bf u},{\bf v}}(q,t,a)=Q_{{\bf x},{\bf y}}(q,t,a).
\end{equation*}
\end{theorem}

\begin{corollary}[\cite{GMV20,Mellit-Homology}]
\label{cor:R vs catalan}
For any $m,n$ coprime, the KR series of the $(m,n)$-torus knot is given by
\begin{align*}
\PC_{T(m,n)} = \frac{(at^{1/2}q^{-1/2})^{\delta(m,n)}} {1-q} Q_{0^{m-1}1 , 0^{n-1}1}(q,t,a)=\frac{(at^{-1/2}q^{-1/2})^{\delta(m,n)}} {1-q} S_{m,n}(q,t,a).
\end{align*}
\end{corollary}

\section{Schr\"oder Polynomials for Triangular Partitions}\label{sec:Schroder} 

Generalized $(q,t)$-Catalan polynomials for triangular partitions were defined \cite{BHMPS}, with more general combinatorial properties of such partitions studied by Bergeron-Mazin \cite{BM}.  In this section, we extend many of the combinatorial results from  \S\ref{sec:GMV-Recursions} and define the Schr\"oder polynomial of an arbitrary triangular partition.  

\subsection{Triangular Partitions Revisited}\label{subsec:m,n-triangular}
In order to define the Schr\"oder polynomial for a triangular partition $\tau$, we need to relate $\tau$-Dyck paths to some invariant subsets in $\mathbb{Z}_{>0}.$ To achieve this, we first embed $\tau$ inside an appropriate partition $\tau_{m,n}$ for relatively prime integers $m$ and $n.$

\begin{definition}
Let $\tau$ be a triangular partition and let $m,n\in\mathbb{Z}_{>0}$ be positive relatively prime integers. We say that $\tau$ is \newword{$(m,n)$-triangular} if there exists a pair of positive real numbers $(r,s)$ such that $\tau=\tau_{r,s},$ $s/r=n/m,$ and $s<n.$
\end{definition}

\begin{lemma}\label{lemma:triangular}
    For any triangular partition $\tau$, there exists a pair of positive relatively prime integers $(m,n)$ such that $\tau$ is $(m,n)$-triangular.
\end{lemma}

\begin{proof}
    It is easy to observe (see page 2 in \cite{BM}) that for any triangular partition $\tau$, there exists a line $L_{\tilde{r},\tilde{s}}=\{\tilde{r}y+\tilde{s}x= \tilde{s}\tilde{r}\}$ such that $\tau=\tau_{\tilde{r},\tilde{s}}$ and $L_{\tilde{r},\tilde{s}}$ does not pass through any integer points. Then for any $(r,s)$ sufficiently close to $(\tilde{r},\tilde{s})$, one also gets $\tau=\tau_{\tilde{r},\tilde{s}}=\tau_{r,s}.$ In particular, one can additionally require that $s/r$ is rational, and if $s/r=n/m$  is the reduced form, i.e. $n$ and $m$ are positive relatively prime integers, then $s<n.$ Indeed, the set of such pairs $(r,s)$ is everywhere dense. We conclude that $\tau$ is $(m,n)$-triangular.
\end{proof}

\begin{remark} \label{rem:notation}
Equivalently, one could flip the perspective and begin with a pair of fixed positive relatively prime integers $(m,n)$. Then all the $(m,n)$-triangular partitions can be obtained by the following procedure. Consider the $(m,n)$-diagonal line $L_{m,n}=\{nx+my=nm\},$ shift it down to $\{nx+my=z\}$ for some $0<z<mn,$ and take $(r,s)=(z/n,z/m),$ so that $\{nx+my=z\}=L_{r,s}$ and $\tau_{r,s}$ is an $(m,n)$-triangular partition. By varying $0<z<mn$, one obtains all $(m,n)$-triangular partitions.
\end{remark}

\begin{definition}\label{def:shifted diagonal}
In the context of the proof of Lemma \ref{lemma:triangular} above, we call the diagonal line $L_{r,s}= \{sx+ry = rs\}$ the \newword{shifted diagonal} and the diagonal line $L_{m,n}=\{mx+ny = mn\}$ the \newword{main diagonal}.
\end{definition}

\begin{lemma}
Let $0<z<mn$ and let $\ell=\lfloor z\rfloor.$ Then $\tau_{z/n,z/m}=\tau_{\ell/n,\ell/m}.$    
\end{lemma}

\begin{proof}
    It suffices to show that a box $\Box=(i,j)$ lies below $L_{z/n,z/m}$ if and only if it lies below $L_{\ell/n,\ell/m}.$ Indeed, since $i,j,m,n\in\mathbb{Z},$
\begin{equation*}
    in+jm\le z \Leftrightarrow in+jm\le \lfloor z\rfloor=\ell.
\end{equation*}
\end{proof}

\begin{corollary}\label{cor:(m,n,l)}
    For every triangular partition $\tau$, there exists a triple of positive integers $(m,n,\ell),$ such that $m$ and $n$ are relatively prime, $\ell<mn,$ and $\tau=\tau_{\ell/n,\ell/m}$ is $(m,n)$-triangular.
\end{corollary}

\noindent \underline{\textbf{Notation:}} In what follows, the choice of main diagonal with respect to which a triangular partition $\tau_{r,s}$ is considered will be very important. Thus, to make this choice explicit, all the combinatorial constructions will henceforth depend on the triple of integers $(m,n,\ell)$ rather than the pair of real numbers $(r,s)$, using the triple to index the various objects. We will start by adopting the following notation whenever the triple $(m,n,\ell)$ satisfies the conditions of Corollary \ref{cor:(m,n,l)}:
\begin{align*}
    \tau_{m,n,\ell}&:=\tau_{\ell/n,\ell/m}=\tau_{r,s},\\
    \D(m,n,\ell)&:=\D(\ell/n,\ell/m)=\D(r,s),\\
    L_{m,n,\ell}&:=L_{\ell/n,\ell/m} = \{nx+my=\ell\}.
\end{align*}

\begin{remark}
Note that the line $L_{m,n,\ell}$ passes through the NE corner of the box $(a,b)\in\mathbb{Z}\times [n],$ whose Anderson filling equals $\gamma(a,b)=mn-na-mb=mn-\ell.$ If $b>0$, then $(a,b)\in\tau_{m,n,\ell}$ and it is the box of $\tau_{m,n,\ell}$ with the maximal Anderson filling. If $b\le 0$, then all the labels in $\tau_{m,n,\ell}$ are less then $mn-\ell.$
\end{remark}

\subsection{Triangular Schr\"oder Polynomials}\label{ss:rs invariant subsets}
Let $\tau$ be a triangular partition and let $(m,n,\ell)$ be as in Corollary \ref{cor:(m,n,l)}, so that $\tau=\tau_{m,n,\ell}.$ In order to define triangular Schr\"oder polynomials, we will first need to define a pair of trinary sequences ${\bf x}(m,n,\ell)\in\{0,1,\bullet\}^m$ and ${\bf y}(m,n,\ell)\in\{0,1,\bullet\}^n.$ We will later see that the associated Schr\"oder polynomial actually depends only on the triangular partition $\tau$ itself and not choice the triple $(m,n,\ell).$

Recall the Anderson filling $\gamma(x,y)=\gamma_{m,n}(x,y)=mn-nx-my$ from Definition \ref{def:anderson label}. Since for any such $(m,n,\ell)$ we have $\tau_{m,n,\ell}\subseteq \tau_{m,n}$, then the pair $(\tau_{m,n,\ell},\tau_{m,n})$ is an $m,n$-Dyck path as in Definition \ref{def:Dyck}.

 In particular, we have
\begin{equation*}
    \Gaps(\tau_{m,n,\ell},\tau_{m,n})=\{\gamma(\Box)\; |\; \Box\in\tau_{m,n}\setminus\tau_{m,n,\ell}\}.
\end{equation*}

To each such Dyck path, we associate a particular subcollection of $0$-normalized $(m,n)$-invariant subsets $\Inv_{m,n}^0$ (Definition \ref{def:invariant subsets}).

\begin{definition}\label{def: Inv-shifted}
For any triple $(m,n,\ell)$ with $m$ and $n$ relatively prime and $\ell<mn$, set 
\[ \Inv_{m,n,\ell}^0:= \{ \Delta \in \Inv_{m,n}^0 \; |  \; \Delta \cap \Gaps(\tau_{m,n,l},\tau_{m,n}) = \emptyset \}. \]
\end{definition}

Recall the bijection $\Aa_{m,n}:\D(m,n)\to \Inv_{m,n}^0$ mapping $(\lambda,\tau_{m,n}) \mapsto \mathbb{Z}_{\ge 0}\setminus\Gaps(\lambda,\tau_{m,n})$. 

\begin{definition}
    Let $\Aa_{m,n,\ell}:\D(m,n,\ell)\to \Inv_{m,n,\ell}^0$ be the map obtained via the natural restriction,  
    \begin{equation*}
    \Aa_{m,n,\ell}(\lambda,\tau_{m,n,\ell}):=\Aa_{m,n}(\lambda,\tau_{m,n}).
    \end{equation*}
\end{definition}

We will soon see in Proposition \ref{prop:Dmap} that the map $\Aa_{m,n,\ell}:\D(m,n,\ell)\to \Inv_{m,n,\ell}^0$ is also a bijection.

\begin{definition}\label{def:W-U}
For any triple $(m,n,\ell)$ with $m$ and $n$ relatively prime and $\ell<mn$, define
\begin{align*}
W_{m,n,\ell} &:= \{ \gamma(x,y) \; | \; \text{ the box $(x,y)\in \Z\times[n]$ lies on the shifted diagonal $nx+my = l$} \},\\
U_{m,n,\ell} &:= \{ \gamma(x,y) \; | \; \text{ the box $(x,y)\in \Z_{\le 0}\times[n]$ lies on the shifted diagonal $nx+my = l$} \},
\end{align*}
and set
\begin{equation} \label{eq:InvUW}
\Inv_{m,n,\ell}^0 := \{ \Delta \in \Inv_{m,n}^0 \; | \; \Delta \cap W_{m,n,\ell} = U_{m,n,\ell}\}.
\end{equation}
\end{definition}

\begin{remark}\label{rmk:Urs} Recall from Lemma \ref{lem:Anderson} that $\Gamma_{m,n}$ can be identified with the set of Anderson labels $\gamma(x,y)$ with $(x,y)\in \Z_{\leq 0}\times [n]$. Hence, we can write
    \[
        U_{m,n,\ell}=\Gamma_{m,n}\cap W_{m,n,\ell}
    \]
    and the condition defining $\Inv_{m,n,\ell}^0$ is $\Delta\cap W_{m,n,\ell}=\Gamma_{m,n}\cap W_{m,n,\ell}$.
\end{remark}

\begin{example}\label{example:l=mn-1 first part}
Consider the special case when $\ell=mn-1.$   In this case, $W_{m,n,mn-1}=\{-m-n+1,-m-n+1,\ldots,0\}$ and $U_{m,n,mn-1}=\{0\},$ so that $\Inv_{m,n,mn-1}^0=\Inv_{m,n}^0.$ Moreover $\tau_{m,n,mn-1}=\tau_{m,n},$ and one naturally recovers the case studied in \cite{GMV20}.
\end{example}

We now explicitly characterize $W_{m,n,\ell}$.

\begin{lemma}\label{lem:W-interval}
The set $W_{m,n,\ell}$ is given by the discrete interval $W_{m,n,\ell}=\big[\; mn-\ell -(m+n), mn-\ell-1  \;\big]\cap \Z$.
\end{lemma}

\begin{proof} 
Consider the set $S$ consisting of boxes in $\Z \times [n]$ lying on the shifted diagonal $L_{m,n,\ell}=\{nx+my=\ell\}$ so that $W_{m,n,\ell} = \gamma(S)$. According to Definition \ref{def:box on line}, a cell $(a,b)$ is in  $S$, if and only if
\[n(a-1)+m(b-1)\leq \ell< na+mb.\]
Hence, $S$ consists precisely of cells $(a,b)$ satisfying the condition that \[\ell< na+mb\leq \ell+n+m.\]
Consequently, since $mn-\gamma(x,y) = nx+my$, then we see that $(a,b) \in S$ if and only if
\[
\ell< mn-\gamma(a,b)\leq \ell+n+m, \]
which holds if and only if
\[mn-\ell -(n+m) \leq \gamma(a,b) < mn-\ell.
\]
Since by Lemma \ref{lem:Anderson} the Anderson function is bijective onto $\Z \times [n]$, every value in this interval is achieved by $\gamma$.  This completes the proof.
 \end{proof}

\begin{lemma}\label{lemma: initial interval for Inv_r,s^0}  
    All invariant sets in $\Inv_{m,n,\ell}^0$ have the same intersection with the interval $[0,mn-\ell-1]$. In fact, for any $\Delta\in\Inv_{m,n,\ell}^0$, one has 
    \begin{equation*}
    \Delta\cap [0,mn-\ell-1]=\Gamma_{m,n}\cap [0,mn-\ell-1].    
    \end{equation*}     
   \end{lemma}

\begin{proof}
    Let $\Delta\in\Inv_{m,n,\ell}^0$ be arbitrary. Clearly, $\Gamma_{m,n}\subset\Delta$ because $\Delta$ contains $0$ and is $m,n$-invariant.  Suppose $k\in \Z$ satisfies $0< k<mn-\ell$ and $k\notin\Gamma_{m,n}$. We must show that $k\notin \Delta$.
    
    Let $(a,b)\in \mathbb{Z}\times [n]$ be the unique box such that $k=\gamma(a,b)$, i.e.~$k=mn-(na+mb)$. The Anderson label being positive means $na+mb<mn$.  Moreover, $a>0$ since $k\notin\Gamma_{m,n}$, and $b>0$ by assumption, so in fact $(a,b)\in \tau_{m,n}$.
    
    The inequality $k<mn-\ell$ is equivalent to $na+mb>\ell$, which implies $(a,b)\notin \tau_{m,n,\ell}$.  Since $\Delta$ is disjoint from $\gamma(\tau_{m,n}\setminus\tau_{m,n,\ell})$ by definition of $\Inv_{m,n,\ell}^0$, it follows that $(a,b)\notin \Delta$, as claimed. 
\end{proof}

The next step is to characterize $\Inv_{m,n,\ell}^0$ in terms of certain trinary sequences. Before doing so we introduce some notation. Recall that for any $A \subset \Z$ and $i \in \Z$, we have $A+i = \{ k \in \Z | k-i \in A\}$. Thus, given a subset $\mathcal{S} \subset \mathcal{P}(\Z)$ of the power set of $\Z$, we define 
\[ \mathcal{S}+ i:= \{ A + i | A \in \mathcal{S}\}, \]
\[ \mathcal{S}\cap \Z_{\geq 0} := \{ A \cap \Z_{\geq 0} | A \in \mathcal{S}\}. \]

\begin{definition}\label{def:from slope to binary}
For any triple $(m,n,\ell)$ with $m$ and $n$ relatively prime and $\ell<mn$, let 
\[{\bf w}(m,n,\ell) := (w_0, w_1, \dots, w_{m+n-1})\in \{0,1\}^{m+n}\]
be the binary sequence defined such that $w_i =1$ if and only if $i + mn-\ell -(n+m) \in \Gamma_{m,n}$ (equivalently, if and only if $i + mn-\ell -(n+m) \in U_{m,n,\ell}$).

Furthermore, similar to Definition \ref{def:trinaryseq}, let ${\bf x}(m,n,\ell) = (x_0, \dots x_{m-1})$ and ${\bf y}(m,n,\ell) = (y_0, \dots y_{n-1})$ be given by:
\begin{equation}\label{eq:trinarysequences}
x_i = 
\begin{cases}
0 & ; w_{n+i} =0,\\
\bullet &; w_{n+i} = w_i=1,\\
1 &; w_{n+i} = 1 \text{ and } w_i = 0,
\end{cases}
\text{ and    }\ \ 
y_i = 
\begin{cases}
0 & ; w_{m+i} =0,\\
\bullet &; w_{m+i} = w_i=1,\\
1 &; w_{m+i} = 1 \text{ and } w_i = 0.
\end{cases}
\end{equation}

Finally, let ${\bf u}(m,n,\ell)$ and ${\bf v}(m,n,\ell)$ be the binary sequences obtained from ${\bf x}(m,n,\ell)$ and ${\bf y}(m,n,\ell)$ respectively by omitting all the $\bullet$'s.
\end{definition}

\begin{proposition}\label{prop:shiftedInv}
Retain the notation from Definition \ref{def:from slope to binary}.  The map $\mathcal{I}: \Inv_{m,n,\ell}^0 \to \Inv_{{\bf w}(m,n,\ell)}$ sending\begin{equation*}
\Delta \mapsto (\Delta - (mn-\ell) +(n+m)) \cap \Z_{\geq 0}    
\end{equation*}
is a bijection. In particular, 
\[\Inv_{{\bf w}(m,n,\ell)}  = (\Inv_{m,n,\ell}^0 -(mn-\ell) +(n+m)) \cap \Z_{\geq 0}.
\]
\end{proposition}

\begin{proof}
By \eqref{eq:InvUW} we have that: 
\begin{align*}
\Inv^0_{m,n,\ell} &= \{ \Delta \in \Inv_{m,n}^0 | \Delta \cap W_{m,n,\ell} = U_{m,n,\ell} \}\\
& = \{ \Delta \in \Inv_{m,n}^0 | j \in \Delta \cap W_{m,n,\ell} \Leftrightarrow j \in U_{m,n,\ell} \}. \end{align*}
Since by  Lemma \ref{lem:W-interval} we know that $W_{m,n,\ell}- (mn-\ell)+(n+m) = [0,m+n-1]\cap \Z$, it follows that
 $j \in \Delta \cap W_{m,n,\ell}$ if and only if $ j - (mn-\ell)+(n+m) \in (\Delta - (mn-\ell)+(n+m))\cap [0,m+n-1] $. 
 
 On the other hand, by \eqref{eq:Inv-w}:
\begin{align*}
 \Inv_{{\bf w}(m,n,\ell)}&= \{ \Delta' \in \Inv_{m,n} \;|\; i \in \Delta' \cap [0,m+n-1] \Leftrightarrow w_i =1\}\\
& = \{ \Delta' \in \Inv_{m,n} | i \in \Delta' \cap [0 , m+n -1  ] \Leftrightarrow i+ mn-\ell-(n+m) \in U_{m,n,\ell} \}.
 \end{align*}
Thus for any $\Delta\in\Inv_{m,n,\ell}^0$, one gets $\mathcal{I}(\Delta)\in\Inv_{{\bf w}(m,n,\ell)}.$ Furthermore, by Lemma \ref{lemma: initial interval for Inv_r,s^0}, the inverse to $\mathcal{I}$ is given by 
\begin{equation*}
    \Delta'\mapsto (\Delta'+mn-\ell-(n+m))\cup\Gamma_{m,n}.
\end{equation*}
Therefore $\mathcal{I}:\Inv_{m,n,\ell}^0\to \Inv_{{\bf w}(m,n,\ell)}$ is a bijection.
\end{proof}

The following Corollary follows immediately from definitions.

\begin{corollary}\label{cor:ngen and cogen for Phi}
    Let $\Delta\in\Inv_{m,n,\ell}^0,$ then one has
    \begin{align*}
        &\cogennn(\mathcal{I}(\Delta))=\left(\cogen(\Delta)-(mn-\ell)+(m+n)\right)\cap\mathbb{Z}_{\ge 0},\\
        &\ngen(\mathcal{I}(\Delta))\cap\mathbb{Z}_{\ge n}=\left(\ngen(\Delta)-(mn-\ell)+(m+n)\right)\cap\mathbb{Z}_{\ge n}.
    \end{align*}
\end{corollary}

Upon first inspection, it may not be entirely clear why the intersection with $\Z_{\geq 0}$ is needed. Indeed, whenever $mn-\ell -(n+m)\leq 0 $ such a restriction is unnecessary and in fact,
\begin{equation}
\Inv_{{\bf w}(m,n,\ell)}  =\Inv_{m,n,\ell}^0 -(mn-\ell) +(n+m).
\end{equation}
However, when $mn-\ell -(n+m)> 0 $ such a shift may generate negative values, which is not allowed since invariant subsets consist, by definition, of only nonnegative numbers. 

We are finally ready to define the $(q,t)$-Schr\"oder polynomial for any triple $(m,n,\ell)$ of positive integers, such that $m$ and $n$ are relatively prime and $\ell<mn.$ 

\begin{definition}\label{def:schroeder}
The \newword{triangular $(q,t)$-Schr\"oder polynomials} are defined as the family
\begin{equation*}
S_{m,n,\ell}(q,t,a):=t^{|\tau_{m,n,\ell}|}Q_{{\bf x}(m,n,\ell),{\bf y}(m,n,\ell)}(q,t,a).
\end{equation*}
\end{definition}

In Theorem \ref{thm: schroder=poincare} we will prove that these polynomials depend only on the triangular partition $\tau=\tau_{m,n,\ell}$ (and not on the choice of $m,n,l$), hence the term \emph{triangular} is indeed well-suited.

Recall the map $\Aa_{m,n,\ell}: \D(m,n,\ell) \to \Inv_{m,n,\ell}^0$ mapping $(\lambda,\tau_{m,n,\ell}) \mapsto \Z_{\geq 0} \setminus \Gaps(\lambda,\tau_{m,n})$ and the bijection $\mathcal{I}: \Inv_{m,n,\ell}^0 \to \Inv_{{\bf w}(m,n,\ell)}$  sending $\Delta \mapsto (\Delta - (mn-\ell) +(n+m)) \cap \Z_{\geq 0}$ from Proposition \ref{prop:shiftedInv}.

\begin{proposition} \label{prop:Dmap}
The map $\Aa_{m,n,\ell}: \D(m,n,\ell) \to \Inv_{m,n,\ell}^0$ is a bijection. Thus, the composition $\AD:= \mathcal{I}\circ \Aa_{m,n,\ell} :\D(m,n,\ell) \to \Inv_{{\bf w}(m,n,\ell)}$ sending
\begin{align*}
 (\lambda,\tau_{m,n,\ell}) \mapsto (\Z_{\geq 0}\setminus \Gaps(\lambda,\tau_{m,n}) - (mn-\ell) +(n+m)) \cap \Z_{\geq 0}
 \end{align*}
is also a bijection. Furthermore, for any $\pi \in \D(m,n,\ell)$ we have
\begin{equation*}
        \area(\pi)=\area'(\AD(\pi)) \qquad \text{and} \qquad
        \dinv(\pi)+\codinv'(\AD(\pi))=|\tau_{m,n,\ell}|.
    \end{equation*}
\end{proposition}

\begin{proof}
Suppose $\pi =(\lambda,\tau_{m,n,\ell}) \in \D(m,n,\ell)$ and let $\hat{\pi}:=(\lambda,\tau_{m,n})\in\D(m,n)$.  Then $\Gaps(\tau_{m,n,\ell},\tau_{m,n}) \subseteq \Gaps(\hat{\pi})$.  Hence, 
\begin{equation*}
\Aa_{m,n,\ell}(\pi) \cap \Gaps(\tau_{m,n,\ell},\tau_{m,n})= (\Z_{\geq 0}) \setminus \Gaps(\hat{\pi})\cap \Gaps(\tau_{m,n,\ell},\tau_{m,n}) = \emptyset.
\end{equation*}

Thus, $\Aa_{m,n,\ell}(\pi) \in \Inv_{m,n,\ell}^0$. Since all the statements are reversible, then evidently $\Aa_{m,n,\ell}:\D(m,n,\ell) \to \Inv_{m,n,\ell}^0$ is indeed a bijection. Hence, combined with Proposition \ref{prop:shiftedInv}, the second claim follows. 

 Now, by Definition \ref{def:area' codinv'}, one gets
    \begin{align*}
        \area'(\AD(\pi))
        =&|\{k\in\mathbb{Z}_{\ge n+m} \;|\; k\notin \AD(\pi)\}|\\ 
        =&|\{k\in\mathbb{Z}_{\ge mn-\ell} \;|\; k\notin \Aa_{m,n,\ell}(\pi)\}|\\
        =&\left|\{k\in\mathbb{Z}_{\ge mn-\ell} \;|\; k\notin \Gamma_{m,n}\}\right|-\left|\{k\in\Aa_{m,n,\ell}(\pi) \;|\; k\notin \Gamma_{m,n}\}\right|\\
        =&|\tau_{m,n,\ell}|-|\lambda|=\area(\pi).
    \end{align*}
Similarly, by Definition \ref{def:area' codinv'} one gets
    \begin{equation*}
      \codinv'(\AD(\pi))= \sum_{x \; \in \; \ngen(\AD(\pi))}  \left| \overline{\AD(\pi)} \cap \Z_{\geq m+n} \cap [x,x+m-1]\right| - \frac{\xi_{-1}(\AD(\pi)) (\xi_{-1}(\AD(\pi))-1)}{2},  
    \end{equation*}
    where 
    \begin{align*}
    \xi_{-1}(\AD(\pi))=&\left|\{ x \in \ngen(\AD(\pi)) | n\leq x < n+m\}\right|\\
    =&\left|\{ x \in \ngen(\Aa_{m,n,\ell}(\pi)) | mn-\ell-m\leq x < mn-\ell\}\right|=1,
    \end{align*}
   since by Lemma \ref{lemma: initial interval for Inv_r,s^0}, $$\Aa_{m,n,\ell}(\pi)\cap[0,mn-\ell-1]=\Gamma_{m,n}\cap [0,mn-\ell-1],$$ and $\ngen(\Gamma_{m,n})=\{0,m,2m,\ldots,(n-1)m\}.$ Hence,
    \begin{align*}
      \codinv'(\AD(\pi))&= \sum_{x \in \ngen(\AD(\pi))}  \left| \overline{\AD(\pi)} \cap \Z_{\geq m+n} \cap [x,x+m-1]\right|\\
      &=\sum_{x \in \ngen(\Aa_{m,n,\ell}(\pi))}  \left| \overline{\Aa_{m,n,\ell}(\pi)} \cap \Z_{\geq mn-\ell} \cap [x,x+m-1]\right|.
    \end{align*}
    Lemma \ref{lemma: initial interval for Inv_r,s^0} also implies that for every number $k\in[0,mn-\ell-1]$, there exists a unique $x\in\ngen(\Aa_{m,n,\ell}(\pi))$ such that $k\in [x,x+m-1].$ Hence, we obtain
    \begin{align*}
      \codinv'(\AD(\pi))&= \sum_{x \in \ngen(\Aa_{m,n,\ell}(\pi))}  \left| \overline{\Aa_{m,n,\ell}(\pi)} \cap \Z_{\geq 0} \cap [x,x+m-1]\right|-\left|\overline{\Aa_{m,n,\ell}(\pi)}\cap [0,mn-\ell-1]\right|\\
      &= \codinv(\Aa_{m,n,\ell}(\pi))-\left|\overline{\Gamma_{m,n}}\cap [0,mn-\ell-1]\right|\\
      &=\codinv(\Aa_{m,n,\ell}(\pi))-\left(\left|\overline{\Gamma_{m,n}}\right|-\left|\overline{\Gamma_{m,n}}\cap \mathbb{Z}_{\ge mn-\ell}\right|\right)\\
      &=\codinv(\Aa_{m,n,\ell}(\pi))-|\tau_{m,n}|+|\tau_{m,n,\ell}|.
    \end{align*}
    Concluding the proof, since from Definition \ref{def:area-dinv} we have the relation
 \begin{equation*}
        \dinv(\Aa_{m,n,\ell}(\pi))+\codinv(\Aa_{m,n,\ell}(\pi))=\delta(m,n) = |\tau_{m,n}|,
            \end{equation*}
then combined with Theorem \ref{thm:dinvA=dinv} we obtain the desired identity:
    \begin{equation*}
        \dinv(\pi)+\codinv'(\AD(\pi))=\dinv(\Aa_{m,n,\ell}(\pi))+\codinv(\Aa_{m,n,\ell}(\pi))-|\tau_{m,n}|+|\tau_{m,n,\ell}|=|\tau_{m,n,\ell}|.
    \end{equation*}
\end{proof}

One can use Proposition \ref{prop:Dmap} to rewrite the triangular $(q,t)$-Schr\"oder polynomial $S_{m,n,\ell}(q,t,a)$ in term of the partition $\tau_{m,n,\ell}$, as follows:

\begin{align}\label{eq:Schroder2}
    S_{m,n,\ell}(q,t,a)&=t^{|\tau|}Q_{{\bf x}(m,n,\ell),{\bf y}(m,n,\ell)}(q,t,a) \nonumber \\
    &=t^{|\tau|}\sum_{\Delta\in\Inv_{{\bf x}(m,n,\ell),{\bf y}(m,n,\ell)}} q^{\area'(\Delta)}t^{-\codinv(\Delta)}\prod_{k\in\cogennn(\Delta)}(1+at^{-\xi_k(\Delta)}) \nonumber \\
    &=\sum_{\pi\in\D(m,n,\ell)} q^{\area(\pi)}t^{\dinv(\pi)}\prod_{k\in\cogennn(\AD(\pi)))}(1+at^{-\xi_k(\AD(\pi))}).
\end{align}

\begin{corollary}\label{cor:Catalan}
The triangular $(q,t)$-Schr\"oder polynomial $S_{m,n,\ell}(q,t,a)$ specializes to the triangular $(q,t)$-Catalan polynomial of \cite{BM,BHMPS} at $a=0$, 
\begin{equation}\label{eq:catalan triang}
    S_{m,n,\ell}(q,t,0)=\sum_{\pi\in\D(m,n,\ell)} q^{\area(\pi)}t^{\dinv(\pi)}=:C_{m,n,\ell}(q,t).
\end{equation}
\end{corollary}
\begin{remark}\label{rem:Catalan}
The triangular $(q,t)$-Catalan polynomials of \cite{BHMPS} arise from the shuffle theorem under any line and thus depend only on the triangular partition $\tau$ and not the choice of triple. Indeed, we will show both directly in Theorem \ref{thm: schroder=poincare} and as a consequence of Corollary \ref{cor:Schroder-Shuffle} that $S_{m,n,\ell}(q,t,a)$ depends only on the partition and not the choice of triple, so that Corollary \ref{cor:Catalan} is indeed consistent with \cite{BM,BHMPS}. 
\end{remark}

\subsection{A Schr\"oder Path Reformulation}
There is also a way to reformulate the elements of $\cogennn(\AD(\pi))$ and the statistics $\xi_k$ in terms of the Dyck paths $\pi\in\D(m,n,\ell)$, which allows us to define $S_{m,n,\ell}$ in terms of certain lattice paths with diagonal steps classically known as \emph{Schr\"oder paths}. 

\begin{definition}
For an $m,n,\ell$-Dyck path $\pi=(\lambda,\tau_{m,n,\ell})$, the \newword{associated lattice path} $\tilde{\pi}$ is the same as the associated lattice path of the $m,n$-Dyck path $(\lambda,\tau_{m,n})$ (Definition \ref{def:Dyck}). In other words, $\tilde{\pi}$ is the south-east lattice path from $(0,n)$ to $(m,0)$ that follows the boundary of $\lambda.$    
\end{definition}

\begin{definition}
    Let $\pi=(\lambda,\tau_{m,n,\ell})\in\D(m,n,\ell)$ be a Dyck path and $\tilde{\pi}$ be the associated lattice path. The \newword{east boundary} of $\pi$ is the set $\EB(\pi)$ of boxes $\Box$ in $\mathbb{Z}\times (0,n]$, such that the east boundary of $\Box$ overlaps a vertical step of the path $\tilde{\pi}.$ Let also $\AB(\pi)$ denote the set of addable boxes of $\lambda$ (Definition \ref{def:addable}). Finally, define the \newword{shifted Anderson filling} $\hat{\gamma}:(a,b)\mapsto \gamma(a,b)-(mn-\ell)+(n+m).$
\end{definition}

\begin{lemma} \label{lem:cogen}
    Let $\pi=(\lambda,\tau_{m,n,\ell})\in\D(m,n,\ell)$ be an $m,n,\ell$-Dyck path. The shifted Anderson filling $\hat{\gamma}$ provides a bijection between $\AB(\pi)$ and $\cogennn(\AD(\pi)).$ Furthermore, if $k\in\cogennn(\AD(\pi))$ then 
    \begin{equation*}
        \xi_k(\AD(\pi))=\left|\{(a,b)\in\EB(\pi)|k+n<\hat{\gamma}(a,b)\le k+n+m\}\right|.
    \end{equation*}
\end{lemma}

\begin{proof}
    Let $k\in\cogennn(\AD(\pi))$ be a non-negative co-generator and $(a,b)\in\mathbb{Z}\times [n]$ be the unique box such that $\hat{\gamma}(a,b)=k$ (similar to $\gamma,$ $\hat{\gamma}$ also defines a bijection between $\mathbb{Z}\times [n]$ and $\mathbb{Z}$). 
    
    Since $k\ge 0$ and $k\notin\AD(\pi)$ it follows that 
    $$
    k-(n+m)+(mn-\ell)=\gamma(a,b)\notin\Aa_{m,n,\ell}(\pi).
    $$ 
    Similarly, since $k+n=\hat{\gamma}(a-1,b)\in\AD(\pi)$ and $k+m=\hat{\gamma}(a,b-1)\in\AD(\pi),$ it follows that 
    $$
    \gamma(a,b)+n=\gamma(a-1,b)\in\Aa_{m,n,\ell}(\pi)\text{ and }\gamma(a,b)+m=\gamma(a,b-1)\in\Aa_{m,n,\ell}(\pi).
    $$ 
    Therefore, $\gamma(a,b)\in\cogen(\Aa_{m,n,\ell}(\pi)).$ Recall that if $\pi=(\lambda,\tau_{m,n,\ell})$ then $\Aa_{m,n,\ell}(\pi)=\Aa_{m,n}(\lambda,\tau_{m,n})$ and $\gamma$ provides a bijection between $\cogen(\Aa_{m,n}(\lambda,\tau_{m,n}))$ and $\AB(\pi).$ 

    Reversing the argument, let $(a,b)\in\AB(\pi).$ Then $\gamma(a,b)\in\cogen(\Aa_{m,n,\ell}(\pi)).$ Furthermore, since the southwest corner of the box $(a,b)$ lies below the shifted diagonal $L_{m,n,\ell}=\{nx+my=\ell\},$ it follows that
    \begin{align*}
        &n(a-1)+m(b-1)\le\ell\\
        \Leftrightarrow\ \  & mn-na-mb+(m+n)-(mn-\ell)\ge 0\\
        \Leftrightarrow\ \  & \hat{\gamma}(a,b)\ge 0.
    \end{align*}
    Hence, using Corollary \ref{cor:ngen and cogen for Phi}, we get 
    $$    \hat{\gamma}(a,b)\in\cogennn(\AD(\pi))=\left(\cogen(\Aa_{m,n,\ell}(\pi))-(mn-l)+(m+n)\right)\cap\mathbb{Z}_{\ge 0}.
    $$
    By definition, for a non-negative co-generator $k\in\cogennn(\AD(\pi))$, one obtains
    \begin{align*}
        \xi_k(\AD(\pi))=|\{ x \in \ngen(\AD(\pi)) | n+k+1\leq x \leq k+n+m\}|.
    \end{align*}
    Since we only need to count $x\ge n+k+1\ge n,$ applying Corollary \ref{cor:ngen and cogen for Phi} one gets
    \begin{align*}
        \xi_k(\AD(\pi))&=|\{ x \in \ngen(\AD(\pi))\cap\mathbb{Z}_{\ge n} | n+k+1\leq x \leq k+n+m\}|\\
        &=|\{ x \in \left(\ngen(\Aa_{m,n,\ell}(\pi))-(mn-\ell)+(m+n)\right)\cap\mathbb{Z}_{\ge n} | n+k+1\leq x \leq k+n+m\}|\\
        &=|\{ x+(mn-\ell)-(m+n) \in \ngen(\Aa_{m,n,\ell}(\pi)) | n+k+1\leq x \leq k+n+m\}|.
    \end{align*}
    Recall that $\gamma$ provides a bijection between $\EB(\pi)$ and $\ngen(\Aa_{m,n,\ell}(\pi))$, and that $x+(mn-\ell)-(m+n)=\gamma(a,b)$ if and only if $x=\hat{\gamma}(a,b).$ Hence
    \begin{align*}
        \xi_k(\AD(\pi))&=|\{ x+(mn-\ell)-(m+n) \in \ngen(\Aa_{m,n,\ell}(\pi)) | n+k+1\leq x \leq k+n+m\}|\\
        &=|\{(a,b)\in\EB(\pi)|k+n<\hat{\gamma}(a,b)\le k+n+m\}|.
    \end{align*}
\end{proof}
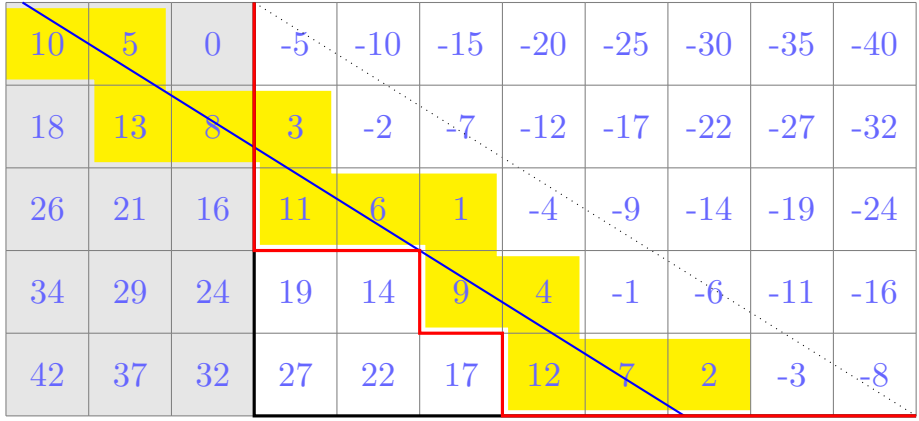
\begin{figure}
\center
\begin{tikzpicture}[scale=1.1]
\fill [fill=gray!20!white] (-3,0) rectangle (0,5);
\draw[line width = 27pt, yellow](-3,4.5)--(-1.5,4.5)--(-1.5,3.5)--(.5,3.5)--(.5,2.5)--(2.5,2.5)--(2.5,1.5)--(3.5,1.5)--(3.5,0.5)--(6,0.5);\
%
 \pgfmathtruncatemacro{\m}{5} 
  \pgfmathtruncatemacro{\n}{8} 
 \pgfmathtruncatemacro{\mi}{\m-1} 
    \pgfmathtruncatemacro{\ni}{\n-1} 
  \foreach \x in {-2,...,\n} 
    \foreach \y in {1,...,\m}     
       {\pgfmathtruncatemacro{\label}{ - \m * \x - \n *  \y +\m*\n} 
       \node  at (\x-.5,\y-.5) [blue!60!white, scale=1.2]{\label};} 
\draw[dotted] (0,\m)--(\n,0); 
  \foreach \x in {-3,...,\n} 
    \foreach \y  [count=\yi] in {0,...,\mi}   
      \draw[gray] (\x,\y)--(\x,\yi);
        \foreach \y in {0,...,\m}     
       \foreach \x  [count=\xi] in {-3,...,\ni}   
         \draw[gray] (\x,\y)--(\x+1,\y);
\draw[ thick, blue] (-2.8,5)--(5.2,0); 
\draw [very thick] (8,0)--(0,0)--(0,5); 
\draw[ very thick,red](0,5)--(0,2)--(2,2)--(2,1)--(3,1)--(3,0)--(8,0); 
\end{tikzpicture}
\caption{Displayed are the lattice path associated with the triangular partition $\lambda=(3,2)=\tau_{8,5,26}=\tau_{5.2,3.25}$ in red, the Anderson labels with respect to $(m,n)=(8,5)$ denoted in blue, the (dotted) main diagonal, and the (solid) shifted diagonal.}\label{fig:5-8}
\end{figure}

The above lemma motivates the following definition. 

\begin{definition}\label{def:xi-AB}
For any $\pi\in\D(m,n,\ell)$ and addable box $\Box\in\AB(\pi)$, define:
    \begin{align*}
        \xi(\pi,\Box):=&\left|\{\Box'\in\EB(\pi)|\hat{\gamma}(\Box)+n<\hat{\gamma}(\Box')\le \hat{\gamma}(\Box)+n+m\}\right|\\
        =&\left|\{\Box'\in\EB(\pi)|\gamma(\Box)+n<\gamma(\Box')\le \gamma(\Box)+n+m\}\right|.
    \end{align*}
\end{definition}

We can now given an expression for $S_{m,n,\ell}(q,t,a)$ in the language of $\tau_{m,n,\ell}$-Dyck paths.

\begin{theorem}\label{thm:Schroder-dyckpaths}
We have
\begin{equation}\label{eq:Schroder-dyckpaths}
S_{m,n,\ell}(q,t,a)= \sum_{k\geq 0} a^k \left(\sum_{\pi\in\D(m,n,\ell)}
q^{\area(\pi)}t^{\dinv(\pi)} \sum_{\substack{
L\subset\AB(\pi)\\ |L|=k}}
t^{-\sum_{\Box\in L} \xi(\pi,\Box)}\right).
\end{equation}
\end{theorem}
\begin{proof}
Combining \eqref{eq:Schroder2} and Lemma \ref{lem:cogen} we obtain, 
\begin{align*}
    S_{m,n,\ell}(q,t,a)&=\sum_{\pi\in\D(m,n,\ell)} q^{\area(\pi)}t^{\dinv(\pi)}\prod_{k\in\cogennn(\AD(\pi)))}\left(1+at^{-\xi_k(\AD(\pi))}\right)\\
    &=\sum_{\pi\in\D(m,n,\ell)} q^{\area(\pi)}t^{\dinv(\pi)}\prod_{\Box\in\AB(\pi)}(1+at^{-\xi(\pi,\Box)}).
\end{align*}
Expanding the product, we get the following expression.
\[
    S_{m,n,\ell}(q,t,a)=\sum_{\pi\in\D(m,n,\ell)} q^{\area(\pi)}t^{\dinv(\pi)} \sum_{L\subset\AB(\pi)}a^{|L|} t^{-\sum_{\Box\in L} \xi(\pi,\Box)} \nonumber
\]
Rearranging proves the Theorem.
\end{proof}

For each fixed $k$, the summation over the pairs $(\pi,L)$ in \eqref{eq:Schroder-dyckpaths} can be thought of as the summation over paths $\pi\in\D(m,n,\ell)$ with exactly $k$ marked addable boxes. One can also modify the marked addable boxes by replacing these with northwest diagonal steps. Lattice paths under the main diagonal consisting of north, west, and northwest diagonal steps are known as \newword{Schr\"oder paths} \cite{S1870}, hence the name of the polynomials $S_{m,n,\ell}(q,t,a).$ Schr\"oder polynomials in the classical square case were introduced in \cite{EHKK03} and \cite{Ha04}.

\begin{figure}[ht]
\[
\begin{array}{l|l|l}
 \pi\in\D(8,5,26) & \Aa(\pi)\in\Inv_{8,5,26}^0\ \ \ \ \  &\AD(\pi)\in\Inv_{0000100101001}\ \ \ \ \ \\
\hline
\emptyset & \Gamma_{8,5} & \{4,7,9,12\}\cup \mathbb{Z}_{\ge 13}\setminus \{13,16,18,21,26\}\\
\hline
(1) & \Gamma_{8,5}\cup \{27\} & \{4,7,9,12\}\cup \mathbb{Z}_{\ge 13}\setminus \{13,16,18,21\}\\
\hline
(1,1) & \Gamma_{8,5}\cup \{19,27\} & \{4,7,9,12\}\cup \mathbb{Z}_{\ge 13}\setminus \{13,16,21\}\\
\hline
(2) & \Gamma_{8,5}\cup \{22,27\} & \{4,7,9,12\}\cup \mathbb{Z}_{\ge 13}\setminus \{13,16,18\}\\
\hline
(2,1) & \Gamma_{8,5}\cup \{19,22,27\} & \{4,7,9,12\}\cup \mathbb{Z}_{\ge 13}\setminus \{13,16\}\\
\hline
(3) & \Gamma_{8,5}\cup \{17,22,27\} & \{4,7,9,12\}\cup \mathbb{Z}_{\ge 13}\setminus \{13,18\}\\
\hline
(2,2) & \Gamma_{8,5}\cup \{14,19,22,27\} & \{4,7,9,12\}\cup \mathbb{Z}_{\ge 13}\setminus \{16\}\\
\hline
(3,1) & \Gamma_{8,5}\cup \{17,19,22,27\} & \{4,7,9,12\}\cup \mathbb{Z}_{\ge 13}\setminus \{13\}\\
\hline
(3,2) & \Gamma_{8,5}\cup \{14,17,19,22,27\} & \{4,7,9,12\}\cup \mathbb{Z}_{\ge 13}\\
\end{array}
\]
\caption{The nine subparitions of $\tau_{8,5,26}$ and their combinatorial data.} \label{ex:8,5,26}
\end{figure}

\begin{example}\label{ex: (5.28,3.3)} Let $(m,n,\ell) = (8,5,26)$ so that $\tau_{8,5,26} = \tau_{5.2,3.25}=(3,2)$. The partition $\tau_{8,5,26}$, Anderson labels, and shifted (solid) and main (dotted) diagonals can be seen in Figure \ref{fig:5-8}. 

The set $W_{8,5,26} =[1,13] $ (shaded yellow in Figure \ref{fig:5-8}) consists of the window of all boxes that lie on the shifted diagonal. Notice that since $U_{8,5,26} = \{ 5,8,10,13\}$ and the gaps of $(\tau_{8,5,26},\tau_{8,5})$ within $W_{8,5,26}$ are precisely $\{ 1,2,3,4,6,7,9,11,12\}$, then: 
\begin{align*}
\Inv_{8,5,26}^0=&\{\Delta\in \Inv_{8,5}^0: \{1,2,3,4,6,7,9,11,12\}\cap\Delta=\emptyset\}\\
=&\{\Delta\in \Inv_{8,5}^0: \Delta\cap [1,13]=\{5,8,10,13\}\}.
\end{align*}
By Proposition \ref{prop:shiftedInv}, we obtain $${\bf w}(8,5,26) = 0000100101001,\ {\bf x}(8,5,26)=0010\bullet00\bullet,\ \text{and}\ {\bf y}(8,5,26)=0100\bullet.$$ Hence, 
\[
\Inv_{0000100101001} = \Inv_{0010\bullet00\bullet,0100\bullet} = \{ \Delta \in \Inv_{8,5} \; |\; \Delta \cap [0,12] = \{ 4,7,9,12\} \}. 
\]
In particular, by directly comparing all the subsets in $I_{8,5,26}^0$ and $I_{0000100101001}$ corresponding to each of the nine subpartitions of $\tau_{8,5,26}=(3,2)$: $\emptyset,(1),(1,1),(2),(2,1),(3),(2,2),(3,1),$ and $(3,2)$, we can see that indeed, $I_{0000100101001}$ can be obtained from $I_{8,5,26}^0$ by shifting each set down by one and then intersecting with $\mathbf{Z}_{\ge 0}$ (thus discarding $-1$). See Figure \ref{ex:8,5,26}.

See Figure \ref{figure: triangular Schroder 8-5} for the computation of the contributions of the nine elements of $I_{0010\bullet00\bullet,0100\bullet}$ towards $Q_{0010\bullet00\bullet,0100\bullet}(q,t,a)$. Putting it all together, we obtain
\begin{equation}\label{formula: schroder via R, example}
S_{8,5,26}(q,t,a)=t^5Q_{0010\bullet00\bullet,0100\bullet}(q,t,a). 
\end{equation}
Expanded, this yields:
\begin{align*}
S_{8,5,26}(q,t,a)=&q^5(1+a)+q^4t(1+a)(1+at^{-1})+q^3t^2(1+a)(1+at^{-1})\\
+&q^3t(1+a)(1+at^{-1})+q^2t^3(1+a)(1+at^{-1})(1+at^{-2})\\
+&q^2t^2(1+a)(1+at^{-1})+qt^3(1+a)(1+at^{-1})\\
+&qt^4(1+a)(1+at^{-1})(1+at^{-2})+t^5(1+a)(1+at^{-1})(1+at^{-2}).    
\end{align*}
\end{example}

\begin{figure}[ht]
    \centering
\begin{tabular}{|l|l|l|}
\hline
$\Gaps\cap\mathbb{Z}_{\ge 13}=\{13,16,18,21,26\}$ & $5$-$\gen=\{4,7,15,23,31\}$& $\cogennn=\{26\}$  \\
$\area'=5$&$\codinv'=5,$ $\dinv=0$& $1+a$  \\
\hline
$\Gaps\cap\mathbb{Z}_{\ge 13}=\{13,16,18,21\}$ & $5$-$\gen=\{4,7,15,23,26\}$& $\cogennn=\{18,21\}$  \\
$\area'=4$&$\codinv'=4,$ $\dinv=1$& $(1+a)(1+at^{-1})$  \\
\hline
$\Gaps\cap\mathbb{Z}_{\ge 13}=\{13,16,21\}$ & $5$-$\gen=\{4,7,15,18,26\}$& $\cogennn=\{10,21\}$  \\
$\area'=3$&$\codinv'=3,$ $\dinv=2$& $(1+a)(1+at^{-1})$  \\
\hline
$\Gaps\cap\mathbb{Z}_{\ge 13}=\{13,16,18\}$ & $5$-$\gen=\{4,7,15,21,23\}$& $\cogennn=\{16,18\}$  \\
$\area'=3$&$\codinv'=4,$ $\dinv=1$& $(1+a)(1+at^{-1})$  \\
\hline
$\Gaps\cap\mathbb{Z}_{\ge 13}=\{13,16\}$ & $5$-$\gen=\{4,7,15,18,21\}$& $\cogennn=\{10,13,16\}$  \\
$\area'=2$&$\codinv'=2,$ $\dinv=3$& $(1+a)(1+at^{-1})(1+at^{-2})$  \\
\hline
$\Gaps\cap\mathbb{Z}_{\ge 13}=\{13,18\}$ & $5$-$\gen=\{4,7,15,16,23\}$& $\cogennn=\{11,18\}$  \\
$\area'=2$&$\codinv'=3,$ $\dinv=2$& $(1+a)(1+at^{-1})$  \\
\hline
$\Gaps\cap\mathbb{Z}_{\ge 13}=\{16\}$ & $5$-$\gen=\{4,7,13,15,21\}$& $\cogennn=\{10,16\}$  \\
$\area'=1$&$\codinv'=2,$ $\dinv=3$& $(1+a)(1+at^{-1})$  \\
\hline
$\Gaps\cap\mathbb{Z}_{\ge 13}=\{13\}$ & $5$-$\gen=\{4,7,15,16,18\}$& $\cogennn=\{10,11,13\}$  \\
$\area'=1$&$\codinv'=1,$ $\dinv=4$& $(1+a)(1+at^{-1})(1+at^{-2})$  \\
\hline
$\Gaps\cap\mathbb{Z}_{\ge 13}=\emptyset$ & $5$-$\gen=\{4,7,13,15,16\}$& $\cogennn=\{8,10,11\}$  \\
$\area'=0$&$\codinv'=0,$ $\dinv=5$& $(1+a)(1+at^{-1})(1+at^{-2})$  \\
\hline
\end{tabular}
    \caption{For the nine elements of $I_{0010\bullet00\bullet,0100\bullet}$ we record the gaps that are greater than or equal to $13,$ since only those contribute to $\area'$ and $\codinv'$. We also record the non-negative co-generators and the corresponding Schr\"oder factors.}
    \label{figure: triangular Schroder 8-5}
\end{figure}

\section{The Knot Associated to a Triangular Partition}\label{sec:trian knot}

As we shall see, for $(m,n,\ell)$ as in Section \ref{subsec:m,n-triangular}, the Schr\"oder polynomial $S_{m,n,\ell}(q,t,a)$ depends only on the triangular partition $\tau=\tau_{m,n,\ell}$ cut by the diagonal line and not on the choice of triple. From a combinatorial perspective this is far from obvious (although it will also follow directly from the shuffle theorem under any path, see \S\ref{sec:qtSchroderTheorem}). In this section, we provide an independent and purely topological proof; namely, we associate a knot to each triple $(m,n,\ell)$ and prove the following:
\begin{enumerate}
    \item The KR series of $K_{\ell/n,\ell/m}$ equals $S_{m,n,\ell}$ (up to normalization).
    \item Up to isotopy, the knot $K_{\ell/n,\ell/m}$ depends only on the triangular partition $\tau=\tau_{m,n,\ell}$.
\end{enumerate}
The construction of $K_{\ell/n,\ell/m}$ uses a construction from Galashin-Lam \cite{GL23}, and we refer the reader to that article for further details.

\subsection{Knots from Young Diagrams}
\label{ss:knots and partitions}

\begin{figure}[ht]
\definecolor{shade}{rgb}{.9,1,1}
\begin{tikzpicture}[scale=.8,anchorbase]
\draw[fill=shade] (0,0) rectangle (2,1);
\draw[fill=shade] (0,1) rectangle (1,2);
\draw[fill=shade] (0,2) rectangle (1,3);
\foreach\j in {0,1,...,5}{\draw[black] (0,\j)--(4,\j);}
\foreach\i in {0,1,...,4}{\draw[black] (\i,0)--(\i,5);}
\draw[blue, thick] (0,5)--({2/23},{4+11/23})--({3+7/37},{5/37})--(4,0);
\end{tikzpicture}
\ \ \ $\mapsto$ \ \ \ 
\begin{tikzpicture}[scale=3,anchorbase]
\def\m{5}
\def\n{7}
\def\r{2} 
\def\I{3} 
\def\J{2} 
%
\draw[white,line width=2mm] ({2/23},{11/23})--({(3-.1*\m)/\n},.1);
\draw[blue, thick] (0,{(\m)/\m})--({2/23},{11/23})--({(3)/\n},0);
\foreach \i in {1,4}{ 
\draw[white, line width=2mm] (0,{(\i)/\m}) --  ({(\i-.1*\m)/\n},.1);
\draw[blue,  thick] (0,{(\i)/\m})--  ({(\i)/\n},0);  
}
\foreach \j in {1}{ 
\draw[white, line width=2mm] ({(\j+.1*\m)/\n},1-.1) -- ({(\j+\m-.1*\m)/\n},0.1); 
\draw[blue,  thick] ({(\j)/\n},1) -- ({(\j+\m)/\n},0); 
}
\foreach \j in {1,2,4}{ 
\draw[white, line width=2mm]  (1,{(\j)/\m}) -- ({(\j+\r+.1*\m)/\n},1-.1);
\draw[blue,  thick]  (1,{(\j)/\m}) -- ({(\j+\r)/\n},1);
}
\draw[white,line width=2mm] ({(\n)/\n},0)--({7/37},{5/37})--(0,{(2)/\m}); 
\draw[blue, thick] ({(\n)/\n},0)--({7/37},{5/37})--(0,{(2)/\m});
\draw[gray] (0,0) rectangle (1,1);
\end{tikzpicture}
\caption{A primitive $(4,5)$ curve  $C$ (left) and the square diagram representing the corresponding knot $\KT_C$ (right).}
\label{fig:monotone curve}
\end{figure}
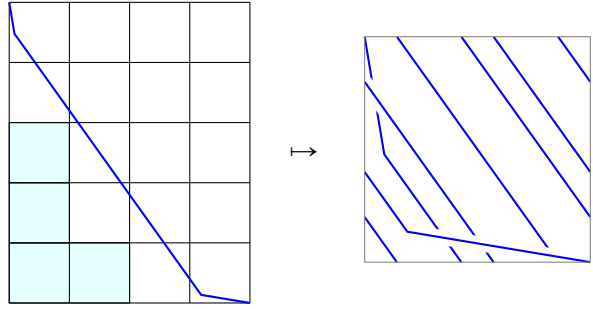

In this section we recall the notion of Coxeter knots.  These generalize torus knots, and are particularly well-suited to connections with the flag Hilbert scheme \cite{GNR21, OR}.   For our purposes, it is most convenient to regard the family of (positive) Coxeter knots as parametrized by Young diagrams.

\begin{definition}\label{def:partition to coxeter}
Let $\tau$ be a partition, and let $M$ be an integer with $M>\tau_1$.  Define a sequence $N>\mu_1\geq \cdots \geq \mu_M=0$ by letting $\mu_j$ be the number of indices $i$ with $\tau\geq j$ (note that $\mu$ is the transpose partition of $\tau$, extended by zero to obtain a sequence of length $M$).

Associated to the pair $(M,\tau)$, we have the braid $\beta^{cox}_{M,\tau}\in \Br_M$ defined by
\begin{equation}\label{eq:JM coxeter braid}
\beta^{cox}_{M,\tau} := (\one_{M-2}\sqcup \JM_2)^{\mu_1-\mu_2}(\one_{M-3}\sqcup \JM_3)^{\mu_2-\mu_3}\dots \JM_{M}^{\mu_{M-1}-\mu_M} \sigma_1\sigma_2\cdots\sigma_{M-1},
\end{equation}
where $\JM_i\in \Br_i$ is the \newword{Jucys-Murphy} braid $\JM_i=(\sigma_1\sigma_2\cdots\sigma_{i-1})(\sigma_{i-1}\cdots\sigma_2\sigma_1)$.
Knots which are closures of braids of this form are called \newword{(positive) Coxeter knots}.  
\end{definition}

\begin{proposition}\label{prop:Ktau}
The closure of $\b^{\cox}_{M,\tau}$ depends only on $\tau$ up to isotopy.  Furthermore, the normalizing exponent \eqref{eq:normalizing exponent} is given by $\d(\b^{\cox}_{M,\tau})=|\tau|$.
\end{proposition}
\begin{proof}
Increasing $M$ by one has the following effect: $\b^{\cox}_{M+1,\tau} =  (\one_1\sqcup \b^{\cox}_{M,\tau})\sigma_1$, which represents the same knot as $\b^{\cox}_{M,\tau}$.  This proves the first statement.

For the second statement, note that the number of crossings in the braid \eqref{eq:JM coxeter braid} is
\begin{align*}
    e
    &=2(\mu_1-\mu_2)+4(\mu_2-\mu_3)+\cdots+2M(\mu_{M-1}-0) +M-1\\
    &=2\mu_1+\cdots+2\mu_{M-1}+M-1\\
    &=2|\tau|+M-1,
\end{align*}
since $\mu_M=0$.
The number of components is 1, so by definition we have
\[
\d(\b^{\cox}_{M,\tau}) = \frac{e+1-M}{2} = |\tau|,
\]
as claimed.
\end{proof}

This proposition justifies the following definition.

\begin{definition}\label{def:Ktau}
Given any partition $\tau$ and any integers $M,N\geq 1$ with $\tau\subset [M-1]\times[N-1]$, we we will let $K_\tau$ denote the Coxeter knot obtained as the closure of $\b^{\cox}_{M,\tau}$, which is well-defined by Proposition \ref{prop:Ktau}. In the special case when $\tau=\tau_{m,n,\ell}$ is a triangular partition, we will write $K_{m,n,\ell}:=K_{\tau_{m,n,\ell}}$.  Such knots will be called \newword{broken torus knots}.
\end{definition}
\begin{remark}
We chose the name \emph{broken torus knot} to emphasize that $K_\tau$ is obtained from a torus knot by removing an arc and re-gluing a different arc, as in Figures \ref{fig:monotone curve} and \ref{fig:binary knot}.
\end{remark}

As mentioned already, one of our main goals is to show that the Schr\"oder polynomial $S_{m,n,\ell}(q,t,a)$ coincides with the the KR series of $K_{{m,n,\ell}}$.  In order to prove this, we will need a more topologically friendly description of $K_{\tau}$ as a monotone knot.

\subsection{Monotone Curves and Knots} 
\label{ss:monotone}

Galashin-Lam \cite{GL23} introduced the notion of a monotone knot and showed that the set of monotone knots (up to isotopy in $\R^3$) coincides with the set of positive Coxeter knots.   In this section we recall their definition, taking some liberties.

\begin{remark}
The work of Galashin-Lam holds for links as well, but we will restrict to the case of knots for simplicity.
\end{remark}

Let $M,N$ be positive integers. A primitive $(M,N)$-\newword{monotone curve} $C$ is the graph of a strictly decreasing continuous function $f:[0,M] \to [0,N]$ satisfying $f(0)=N$, $f(M)= 0$, and such that $C\cap \Z\times \Z = \{(0,M),(N,0)\}$.

\begin{definition}\label{def:C link}
Let $\T= \R^2 / \Z^2$ denote the torus and $\pi:\R^2 \to \T$ be the quotient map. Then the \newword{knot $\KT_C$ associated to a primitive monotone curve} $C$ is the knot $\KT_C\subset \T\times [0,1]$ defined to be concatenation of the curve $\{(\pi(x,f(x)), x/M) \: :\:  x \in [0,M]\}$ and the vertical line segment from $((0,0),1)$ to $((0,0),0)$. 

Let $\KnR_C$ denote the image of $\KT_C$ under the standard embedding $\T\times [0,1]\hookrightarrow \R^3$.  A knot $K\subset \R^3$ is \newword{monotone} if it is isotopic to $\KnR_C$ for some primitive monotone curve $C$.
\end{definition}

\begin{example}
If $m,n\in \Z_{\geq 1}$ are coprime and $C$ is the graph of $\{nx+my=mn\}$ for $0\leq x\leq m$, then $\KnR_C$ is the $(m,n)$ torus knot. Thus, positive torus knots are monotone.
\end{example}

Below we setup some language for discussings knots in the thickened torus and the diagrams we use to represent them.  In the definition below, note that for a surface $\Sigma$ we will distinguish between knots in $\Sigma\times [0,1]$ and their diagrams in $\Sigma$.  Strictly speaking, a \emph{knot diagram} for $K\subset \Sigma\times [0,1]$ is an immersed curve $G$ in $\Sigma$ with double point singularities, together with overcrossing / undercrossing data at the double points.  By a (mostly harmless) abuse, will write $G\subset \Sigma$ and regard $G$ as nothing more than a subset of $\Sigma$.

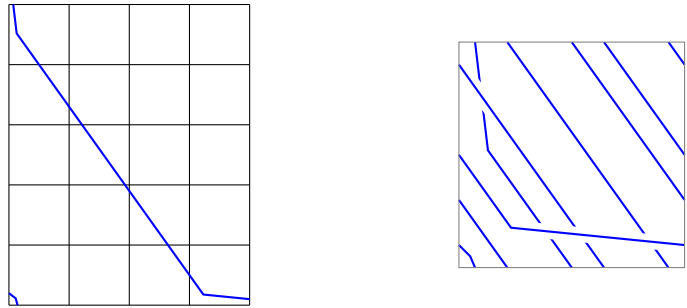
\begin{figure}[ht]
\begin{tikzpicture}[scale=.8,anchorbase]
\foreach\j in {0,1,...,5}{\draw[black] (0,\j)--(4,\j);}
\foreach\i in {0,1,...,4}{\draw[black] (\i,0)--(\i,5);}
\draw[blue, thick] ({.5/7},5)--({2/23+1/24},{4+11/23+1/24})--({3+7/37+1/24},{5/37+1/24})--(4,{.5/5});
\draw[blue, thick] (1/7,0)--(2.7/24,2.7/24)--(0,1/5);
\end{tikzpicture}
\hskip 1in
\begin{tikzpicture}[scale=3,anchorbase]
\def\m{5}
\def\n{7}
\def\r{2} 
\def\I{3} 
\def\J{2} 
%
\draw[white,line width=2mm] ({2/23+1/24},{11/23+1/24})--({(3+.5-.1*\m)/\n},.1);
\draw[blue, thick] ({(.5)/\n},{(\m)/\m})--({2/23+1/24},{11/23+1/24})--({(3+.5)/\n},0);
\foreach \i in {1,4}{ 
\draw[white, line width=2mm] (0,{(\i+.5)/\m}) --  ({(\i+.5-.1*\m)/\n},.1);
\draw[blue,  thick] (0,{(\i+.5)/\m})--  ({(\i+.5)/\n},0);  
}
\foreach \j in {1}{ 
\draw[white, line width=2mm] ({(\j+.5+.1*\m)/\n},1-.1) -- ({(\j+\m+.5-.1*\m)/\n},0.1); 
\draw[blue,  thick] ({(\j+.5)/\n},1) -- ({(\j+\m+.5)/\n},0); 
}
\foreach \j in {1,2,4}{ 
\draw[white, line width=2mm]  (1,{(\j+.5)/\m}) -- ({(\j+\r+.5+.1*\m)/\n},1-.1);
\draw[blue,  thick]  (1,{(\j+.5)/\m}) -- ({(\j+\r+.5)/\n},1);
}
\draw[white,line width=2mm] ({(\n)/\n},{.5/\m})--({7/37+1/24},{5/37+1/24})--(0,{(2+.5)/\m}); 
\draw[blue, thick] ({(\n)/\n},{.5/\m})--({7/37+1/24},{5/37+1/24})--(0,{(2+.5)/\m});
\draw[blue, thick] (.5/\n,0)--(1.2/24,1.2/24)--(0,.5/\m);  
\draw[gray] (0,0) rectangle (1,1);
\end{tikzpicture}
\caption{A shift by $+(\e,\e)$ of the curve $C$ from Figure \ref{fig:monotone curve}, drawn in $\R^2/(4\Z\times 5\Z)$ (left), and the corresponding shifted square diagram $G^{I\times I}_+$ (right).}
\label{fig:shiftedcurve}
\end{figure}

\begin{definition}\label{def:torus diagrams}
Throughout this definition, let $I:=[0,1]$.  Recall that we identify $\T$ as $\R^2/\Z^2$ or, equivalently, as the unit square $[0,1]^2$ with opposite edges identified. The image of $\Z^2$ (respectively $\{0,1\}^2$) in $\T$ will be referred to as the \newword{basepoint of $\T$}.

Let $\KT$ be a knot (or link) in $\T\times I$.  A \newword{square diagram} for $\KT$ is a diagram (with boundary) $G^{I\times I}$ in the unit square $I\times I$ which becomes a knot diagram for $\KT$ after identifying opposite edges. 
\end{definition}
Note that all printed diagrams of knots in the thickened torus are square diagrams in the above sense.
\begin{remark}\label{rmk:monotone diagram}
To obtain a square diagram for a monotone knot $\KT_C$, we draw $G^{I\times I}$ so that for two arcs $\a=C\cap ([i-1,i]\times [j-1,j])$ and $\a'=C\cap([i'-1,i']\times [j'-1,j'])$ with $i<i'$, the image of $\a'$ in $[0,1]^2$ is drawn \emph{above} the image of $\a$. See Figure \ref{fig:monotone curve} for an example.
\end{remark}

\begin{construction}\label{const:KT to KR}
Let $\KT\subset \T\times I$ be a knot (or link) in the thickened torus, and let $\KnR\subset \R^3$ denote its image  under the standard embedding $\T\times I\hookrightarrow \R^3$.  Let $G^\T\subset \T$ be a knot diagram for $\KT$, and let $G^{I\times I}$ be the preimage of $G^\T$ under the projection $I^2\to \T$, so that $G^{I\times I}$ is a square diagram for $\KT$.  We describe how to obtain a diagram for $\KnR$ from $G^{I\times I}$

Assume that $G^\T$ is disjoint from the basepoint (if not, see below), so that $G^{I\times I}$ is disjoint from the corners $\{0,1\}\times \{0,1\}$.  The boundary of $G^{I\times I}$ will consist of pairs of matching points of the form $(X\times \{0,1\}) \ \sqcup (\{0,1\}\times Y)$, where $X,Y$ are finite subsets of $(0,1)$.  To obtain a diagram of $\KnR$ from this $G$, we connect matching points $\{x\}\times\{0,1\}$ by arcs below the rest of the diagram, and we connect matching points $\{0,1\}\times \{y\}$ by any disjoint family of arcs in $\R^2\setminus [0,1]^2$.  The result is a knot diagram, denoted $G^{\R\times \R}$, which represents $\KnR$. See Figure \ref{fig:KT to KR} for an example.

If $G^\T$ meets the basepoint, then we let $G^\T_{\pm}$ denote the diagrams obtained by adding $\pm(\e,\e)$ for $\e>0$ small, we let $G^{I\times I}_{\pm}$ denote their lifts to the unit square, and we let $G^{\R\times \R}_{\pm}$ the diagrams for $\KnR$ obtained as above. 
\end{construction}

\begin{figure}[ht]
\begin{tikzpicture}[scale=3,anchorbase]
\def\m{5}
\def\n{7}
\def\r{2} 
\def\I{3} 
\def\J{2} 
\foreach \i in {2,4,5,7}{
\draw[blue,  thick] ({(\i-.5)/\n},0)--({(\i-.5)/\n},1);
}
{\def\i{1}
\draw[orn,  thick] ({(\i-.5)/\n},0)--({(\i-.5)/\n},1);
}
\foreach \i in {2,3,5}{  
\draw[white, line width=2mm]
(0,{(\i-.5)/\m})..controls++(-.3,0)and++(0,{.1*(\i+1)})..
({-.3-\i/(1.5*\m)},-.2)..controls++(0,{-.1*(\i+1)})and++(-.2,0)..
(.5,{-\i/(2*\m)-.3})..controls++(.2,0)and++(0,{-.1*(\i+1)})..
({1.3+\i/(1.5*\m)},-.2)..controls++(0,{.1*(\i+1)})and++(.3,0)..
(1,{(\i-.5)/\m});
\draw[blue,  thick]
(0,{(\i-.5)/\m})..controls++(-.3,0)and++(0,{.1*(\i+1)})..
({-.3-\i/(1.5*\m)},-.2)..controls++(0,{-.1*(\i+1)})and++(-.2,0)..
(.5,{-\i/(2*\m)-.3})..controls++(.2,0)and++(0,{-.1*(\i+1)})..
({1.3+\i/(1.5*\m)},-.2)..controls++(0,{.1*(\i+1)})and++(.3,0)..
(1,{(\i-.5)/\m});
}
{
\def\i{1}
\draw[white, line width=2mm]
(0,{(\i-.5)/\m})..controls++(-.3,0)and++(0,{.1*(\i+1)})..
({-.3-\i/(1.5*\m)},-.2)..controls++(0,{-.1*(\i+1)})and++(-.2,0)..
(.5,{-\i/(2*\m)-.3})..controls++(.2,0)and++(0,{-.1*(\i+1)})..
({1.3+\i/(1.5*\m)},-.2)..controls++(0,{.1*(\i+1)})and++(.3,0)..
(1,{(\i-.5)/\m});
\draw[orn,  thick]
(0,{(\i-.5)/\m})..controls++(-.3,0)and++(0,{.1*(\i+1)})..
({-.3-\i/(1.5*\m)},-.2)..controls++(0,{-.1*(\i+1)})and++(-.2,0)..
(.5,{-\i/(2*\m)-.3})..controls++(.2,0)and++(0,{-.1*(\i+1)})..
({1.3+\i/(1.5*\m)},-.2)..controls++(0,{.1*(\i+1)})and++(.3,0)..
(1,{(\i-.5)/\m});
}
\draw[white,line width=2mm] ({2/23+1/24},{11/23+1/24})--({(3+.5-.1*\m)/\n},.1);
\draw[orn,very thick] ({(.5)/\n},{(\m)/\m})--({2/23+1/24},{11/23+1/24});
\draw[blue, thick] ({2/23+1/24},{11/23+1/24})--({(3+.5)/\n},0);
%
\foreach \i in {1,4}{ 
\draw[white, line width=2mm] (0,{(\i+.5)/\m}) --  ({(\i+.5-.1*\m)/\n},.1);
\draw[blue,  thick] (0,{(\i+.5)/\m})--  ({(\i+.5)/\n},0);  
}
{\def\i{0}
\draw[orn,  thick] (0,{(\i+.5)/\m})--  ({(\i+.5)/\n},0);  
}
\foreach \j in {1}{ 
\draw[white, line width=2mm] ({(\j+.5+.1*\m)/\n},1-.1) -- ({(\j+\m+.5-.1*\m)/\n},0.1); 
\draw[blue,  thick] ({(\j+.5)/\n},1) -- ({(\j+\m+.5)/\n},0); 
}
\foreach \j in {1,2,4}{ 
\draw[white, line width=2mm]  (1,{(\j+.5)/\m}) -- ({(\j+\r+.5+.1*\m)/\n},1-.1);
\draw[blue,  thick]  (1,{(\j+.5)/\m}) -- ({(\j+\r+.5)/\n},1);
}
\draw[white,line width=2mm] ({(\n)/\n},{.5/\m})--({7/37+1/24},{5/37+1/24})--(0,{(2+.5)/\m}); 
\draw[orn, thick] ({(\n)/\n},{.5/\m})--({7/37+1/24},{5/37+1/24});
\draw[blue, thick] ({7/37+1/24},{5/37+1/24})--(0,{(2+.5)/\m});
\end{tikzpicture}
\caption{The knot diagram $G^{\R\times \R}_+$ obtained by closing up the shifted square diagram $G^{I\times I}_+$ from Figure \ref{fig:shiftedcurve}, as in Construction \ref{const:KT to KR}. The orange coloring will play a role when we compare to Figure \ref{fig:binary knot}, and may be ignored for now.}
\label{fig:KT to KR}
\end{figure}
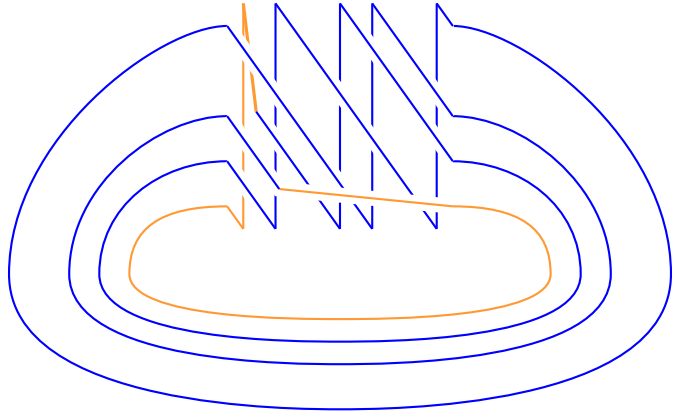
\begin{definition}\label{def:monotone curve to partition}
Let $C$ be a primitive $(M,N)$-monotone curve $C$.  Let $\tau_C$ denote the maximal partition
contained in $\{0\leq y\leq f(x)\:|\: 0\leq x\leq M\}$.

\end{definition}
In other words, a box $(x,y)\in [m]\times[n]$ is in $\tau_C$ if and only if $y<f(x)$.  The connection between monotone knots and Coxeter knots is provided by the following. 

We refer the reader to \cite[Prop. 7.5]{GL23} for a restatement and proof of the following fact.
\begin{proposition} \label{prop:CoxConjugate}
We have $\KnR_C=K_{\tau_C}$, where $K_{\tau_C}$ is as in Definition \ref{def:partition to coxeter}.
In particular, all monotone knots are Coxeter knots, and conversely.
\end{proposition}

We note a few very important special cases of $K_\tau$.

\begin{example}\label{ex:torus knot as Ktau}
Let $m,n\geq 1$ be coprime integers, and set $\tau=\tau_{m,n}$.  Then $K_\tau=T(m,n)$ is the $(m,n)$-torus knot.  We may present $T(m,n)$ in the standard way, as the monotone knot associated to the straight line segment $C$ from $(0,n)$ to $(m,0)$.  Alternatively, we may deform $C$ by an isotopy in $[0,m]\times[0,n]\setminus \Z\times \Z$, obtaining various non-standard presentations.

The associated sequence $N>\mu_1\geq \cdots \geq \mu_M=0$ is defined by $\mu_i = \lfloor n-\frac{in}{m}\rfloor=\lfloor \frac{(m-i)n}{m}\rfloor$, and from Proposition \ref{prop:CoxConjugate} one obtains a description of $T(m,n)$ as the closure of the braid
\[
\b^{\cox}_{m,\tau} = \JM_2^{\lfloor \frac{(m-1)n}{m}\rfloor-\lfloor \frac{(m-2)n}{m}\rfloor}\JM_3^{\lfloor \frac{(m-2)n}{m}\rfloor-\lfloor \frac{(m-3)n}{m}\rfloor}\cdots \JM_m^{\lfloor \frac{n}{m}\rfloor-0}\sigma_1\cdots\sigma_{m-1},
\]
(where we are writing $\JM_i=\one_{M-i}\sqcup \JM_i$, by a mild abuse of notation).  This braid representing $T(m,n)$ also appears, for instance, in \cite[Conjecture 3.28]{GNR21}.
\end{example}

\begin{example}\label{ex:iterated torus knot as Ktau}
Let $m,n,d\geq 1$ be integers with $m,n$ coprime, and set $\tau=\tau_{md,nd}$.  Then $K_\tau$ is the $(d,mnd+1)$ cable of $T(m,n)$ torus knot.  This was proven by Galashin and Lam (\cite[Lemma 8.1]{GL23}). We give an alternate proof in Appendix \ref{sec:AppendixB}.
\end{example}

\subsection{Knots from Binary Sequences; Redux}
\label{ss:KR of Krs}
In this section we show that every knot $K_{\tau}$ with $\tau$ a triangular partition is of the form $K_{\mathbf{u},\mathbf{v}}$, for some pair of binary sequences (and conversely).  The notion of monotone knots (and the flexibility afforded by being able to perform isotopies in $\R^2\setminus \Z^2$) will help facilitate this connection. To do so, we begin by first reconstructing the trinary sequences ${\bf x}(m,n,\ell)$ and ${\bf y}(m,n,\ell)$ from \eqref{eq:trinarysequences}. 

\begin{figure}[ht]
\begin{tikzpicture}[scale=.9,anchorbase]
\def\m{4}
\def\n{5}
\dRarc{4}{5}{0}{4}
\dLarc{4}{5}{0}{3}
\dRarc{4}{5}{1}{3}
\dVarc{4}{5}{1}{2}
\dLarc{4}{5}{1}{1}
\dRarc{4}{5}{2}{1}
\dLarc{4}{5}{2}{0}
\dRarc{4}{5}{3}{0}
\dFarc{4}{5}
\foreach \i in {0,1,...,3}{
	\foreach \j in {0,1,...,4}{
	\begin{scope}[shift={(\i,\j)}]
		\draw (0,0) rectangle (1,1);
		\foreach \p in {0,1,...,3}{
			\draw[thick] (-.05,{(\p+.5)/\m})--(.05,{(\p+.5)/\m})
				(.95,{(\p+.5)/\m})--(1.05,{(\p+.5)/\m});
		}
		\foreach \l in {0,1,...,4}{
			\draw[thick] ({(\l+.5)/\n},-.05)--({(\l+.5)/\n},.05)
				({(\l+.5)/\n},.95)--({(\l+.5)/\n},1.05);
		}
	\end{scope}
	}
}
\end{tikzpicture}
\hskip 1in
\begin{tikzpicture}[scale=3,anchorbase]
\def\m{4}
\def\n{5}
\dRaarc{4}{5}{0}{4}
\dLaarc{4}{5}{0}{3}
\dRaarc{4}{5}{1}{3}
\dVaarc{4}{5}{1}{2}
\dLaarc{4}{5}{1}{1}
\dRaarc{4}{5}{2}{1}
\dLaarc{4}{5}{2}{0}
\dRaarc{4}{5}{3}{0}
\dFaarc{4}{5}
		\draw (0,0) rectangle (1,1);
\end{tikzpicture}
\caption{  We continue the example depicted in Figures \ref{fig:monotone curve} and \ref{fig:shiftedcurve}. In order to facilitate the Coxeter knot description $\KT_C$ we perform an isotopy and shift by $-(\e,\e)$ (instead of $+(\e,\e)$ as we did in Figure \ref{fig:shiftedcurve}).  Closing up as in Construction \ref{const:KT to KR} yields a diagram which is readily seen to be equivalent to the the closure of $\b^\cox_{(2,1,1)}$.} \label{fig:isotopy}
\end{figure}
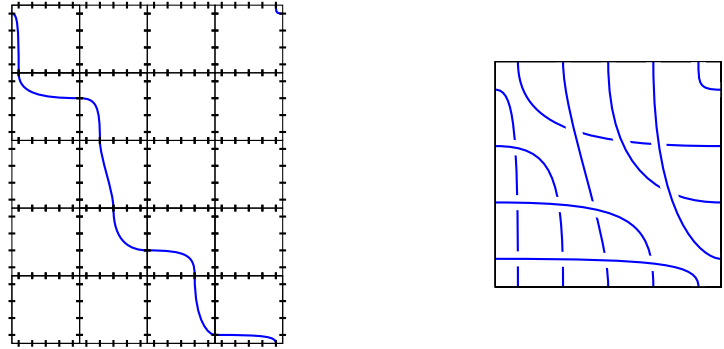

\begin{lemma}\label{lemma: x and y from m,n,l}
Let $(m,n)$ be positive relatively prime integers and $0<\ell<mn,$ $\ell\in\mathbb{Z}.$ Consider the corresponding trinary sequences ${\bf x}(m,n,\ell)=(x_0,\ldots,x_{m-1})$ and ${\bf y}(m,n,\ell)=(y_0,\ldots,y_{n-1}).$ Then,  
\begin{align*}
x_i &= \begin{cases}
    1 & \text{ if } i\equiv \ell \text{ modulo } m \\
    0 & \text{ if } i\equiv \ell - nx \text{ modulo } m \text{ for some } x\in \{1,\ldots,\lfloor \ell/n\rfloor\}\\
    \bullet & \text{ otherwise}
    \end{cases}\\
    y_j &= \begin{cases}
    1 & \text{ if } j\equiv \ell \text{ modulo } n \\
    0 & \text{ if } j\equiv \ell - my \text{ modulo } n \text{ for some } y\in \{1,\ldots,\lfloor \ell/m\rfloor\}\\
    \bullet & \text{ otherwise}
    \end{cases}
\end{align*}

\end{lemma}
\begin{proof}
Recall that the binary sequence ${\bf w}(m,n,\ell)$ from Definition \ref{def:from slope to binary} is given by the condition that $w_i(m,n,\ell)=1$ if and only if
$$
i+mn-\ell-(n+m)\in \Gamma_{m,n}.
$$

By the definition of ${\bf x}(m,n,\ell)$ one gets $x_i=1$ if and only if $w_{i+n}=1$ and $w_i=0,$ which is equivalent to $i+mn-\ell-(n+m)\notin\Gamma_{m,n}$ and $i+mn-\ell-m\in\Gamma_{m,n}.$ In other words, $i+mn-\ell-m$ is an $n$-generator of $\Gamma_{m,n}.$ It is not hard to see that the set of $n$-generators of $\Gamma_{m,n}$ is $\{0,m,2m,\ldots,(n-1)m\}.$ Since for $i\in\{0,1,\ldots,m-1\}$ and $\ell\in\{0,1,\ldots,mn-1\}$ one gets $-m<i+mn-\ell-m<mn,$ we conclude that $x_i=1$ if and only if $i+mn-\ell-m$ is divisible by $m,$ which is equivalent to $i\equiv \ell$ modulo $m.$ 

Similarly, $x_i=0$ if and only if $w_{i+n}=0,$ or $i+mn-\ell-m\notin\Gamma_{m,n}.$ It is well known that the complement to the semigroup $\Gamma_{m,n}$ is given by
\begin{equation*}
\mathbb{Z}\setminus\Gamma_{m,n}=mn-(m+n)-\Gamma_{m,n}.    
\end{equation*}
Therefore, one obtains that $i+mn-\ell-m\notin\Gamma_{m,n}$ if and only if
\begin{equation*}
mn-(m+n)-i-mn+\ell+m=\ell-n-i\in\Gamma_{m,n}.
\end{equation*}
Equivalently, $x_i=0$ if and only if $i$ is the remainder modulo $m$ of $\ell-nx$, for some $x\in\{1,\ldots,\lfloor \ell/n\rfloor\}.$

The case for ${\bf y}(m,n,\ell)$ follows directly by exchanging $m$ and $n$ in the argument above. 

\end{proof}

\begin{example}
Recall Example \ref{example:l=mn-1 first part}. Then in that case, for $\ell=mn-1,$ we have:
$$
\mathbf{x}(m,n,mn-1)=(0^{m-1}1)\text{ and }\mathbf{y}(m,n,mn-1)=(0^{n-1}1).
$$ 
\end{example}

\begin{construction}\label{constr:triang to binary seq}
Given any triangular partition $\tau$, choose a triple $(m,n,\ell)$ with $\tau=\tau_{m,n,\ell}$. As in \S\ref{ss:rs invariant subsets} let $\mathbf{u}(m,n,\ell)\in\{0,1\}^{m'}$ and $\mathbf{v}(m,n,\ell)\in \{0,1\}^{n'}$ be the binary sequences obtained from $\mathbf{x}(m,n,\ell)\in \{0,1,\bullet\}^m$ and $\mathbf{y}(m,n,\ell)\in \{0,1,\bullet\}^n$ in Lemma \ref{lemma: x and y from m,n,l} by deleting the occurences of $\bullet$, where $m' = \lfloor \ell/n\rfloor +1$ and $n'=\lfloor \ell/m\rfloor +1$.
\end{construction}

\begin{lemma}\label{lem:K-L isotopic}
For any triangular partition $\tau=\tau_{m,n,\ell}$, we have
\[
K_{\tau} = K_{\mathbf{u}(m,n,\ell),\mathbf{v}(m,n,\ell)}.
\]
Conversely, given $\mathbf{u}\in \{0,1\}^{m'}$, $\mathbf{v}\in \{0,1\}^{n'}$ with $|\mathbf{u}|=1=|\mathbf{v}|$ such that $K_{\mathbf{u},\mathbf{v}}$ is a knot, we have $K_{\mathbf{u},\mathbf{v}}=K_{\tau}$, for some triangular partition.
\end{lemma}
\begin{proof}
Throughout the proof, we will denote the image of $x\in \R$ under the quotient $\R\twoheadrightarrow \R/\Z$ by $x+\Z$.  Then points in the 2d torus $\T$ are pairs $(x+\Z,y+\Z)$.

Fix a triple $(m,n,\ell)$ with $\tau=\tau_{m,n,\ell}$, and let $M=\left\lceil\frac{\ell}{n}\right\rceil$ and $N=\left\lceil\frac{\ell}{m}\right\rceil$.  Choose $\e\in \R_{>0}$ to be as small and generic as needed (in a sense that will be made precise in a moment) and let $r:=\frac{\ell+\e}{n}$ and $s:=\frac{\ell+\e}{m}$.  Let $C$ be the piecewise linear curve $C=C_{\infty}\cup C_{n/m}\cup C_{0}$ from $(0,N)$ to $(0,s)$ to $(r,0)$ to $(M,0)$.  Assume $\e$ is generic and small enough so that $C$ is primitive, and $\tau=\tau_C$.

Let
\begin{align*}
    X&:=\Big\{x\in (0,1)\:\Big|\: x\equiv r-\frac{jm}{n} \text{ (mod $\Z$)} \text{ with } 0\leq j\leq \ell/m\Big\},\\
    Y&:=\Big\{y\in (0,1)\:\Big|\: y\equiv s-\frac{in}{m} \text{ (mod $\Z$)}  \text{ with } 0\leq i\leq \ell/n\Big\}.
\end{align*}
Note that each element of $X$ is the $x$-coordinate (mod $\Z$) of a point of intersection between $C_{n/m}$ and $\R\times\Z$. Similarly, each element of $Y$ is the $y$-coordinate (mod $\Z$) of a point of intersection between $C_{n/m}$ and $\Z\times \R$.

The number of elements of $X$ is the number of integers $j$ with $0\leq mj\leq \ell+\e$.  Since $\e$ is arbitrarily small, this is equivalent to $0\leq mj\leq \ell$.  Thus, $|X|=n'=\lfloor \ell/m\rfloor +1$.  Similarly, $|Y|=m':=\lfloor \ell/n\rfloor +1$.

Using the natural order on the interval $(0,1)$, we order the points in $X$ and $Y$, writing them as $X=\{x_1<\cdots<x_{n'}\}$ and $Y=\{y_1<\cdots<y_{m'}\}$.

Let $a\in \{1,\ldots,m'\}$ and $b\in\{1,\ldots,n'\}$ be the indices such that $y_a\equiv s$ (mod $\Z$) and $x_b\equiv r$ (mod $\Z$).  Now, define binary sequences $\mathbf{u}=(u_1,\ldots,u_{m'})\in \{0,1\}^{m'}$ and $\mathbf{v}=(v_1,\ldots,v_{m'})\in \{0,1\}^{n'}$ by $u_a=1=v_b$, and $u_i=0=v_j$ for all indices $i\neq a$, $j\neq b$.  Lemma \ref{lemma: x and y from m,n,l} tells us that $\mathbf{u}=\mathbf{u}(m,n,\ell)$ and $\mathbf{v}=\mathbf{v}(m,n,\ell)$.

Now, consider the Construction \ref{constr:Kuv}, which tells us how to draw a diagram for $K_{\mathbf{u},\mathbf{v}}$.  In that construction, with the above choices of $x_j$ and $y_j$, the union of all arcs of the form $A_k,H_i,V_j$ coincides exactly with the contribution of $C_{n/m}$ to the monotone knot $K_C=K_{C_{\infty}\cup C_{n/m}\cup C_0}$.

It remains only to observe that $B=B_-\cup B_+$ (shown in orange in Figure \ref{fig:binary knot}) is isotopic to the contribution of $C_\infty$ and $C_0$ to $K_C$ (shown in orange in Figure \ref{fig:KT to KR}).

For the converse, let $K=K_{\mathbf{u},\mathbf{v}}$ be given, with $\mathbf{u}\in \{0,1\}^{m'}$, $\mathbf{v}\in \{0,1\}^{n'}$. Let $C'$ denote the union of all arcs of the form $A_k,H_i,V_j$ from Construction \ref{constr:Kuv}, so that $K=C'\cup B$.  We may regard $C'$ as embedded in the standard 2d torus $\T$.  Since $C'$ is embedded in $\T$, we may regard $C'$ as a piece of a torus knot $T(m,n)$, for some $m,n$ coprime.  More precisely, up to an isotopy we can arrange that $C'$ be the image of a straight line segment from $(0,s)$ to $(r,0)$, for some real numbers $r,s$ with $s/r = n/m$ and $\lfloor r\rfloor +1 = m'$, and $\lfloor s\rfloor+1= n'$ (these last two equations come from counting the number of intersections of $C'$ with the meridian and longitude of $\T$).  The remaining arc $B$ is isotopic to the image of line segments from $(0,n')$ to $(0,s)$ and from $(r,0)$ to $(m',0)$, as we saw above. 
Thus $K$ is a broken torus knot.
\end{proof}

\subsection{The Invariance Theorem for $S_{m,n,\ell}(q,t,a)$}
\label{ss:invariance thm}

We are finally ready to state and prove our theorem on the independence of the Schr\"oder polynomials $S_{m,n,\ell}(q,t,a)$ on the choice of triple.

\begin{theorem}\label{thm: schroder=poincare}
For any triangular partition $\tau=\tau_{m,n,\ell}$, the corresponding Schr\"oder polynomial $S_{m,n,\ell}(q,t,a)$ equals the KR series of the knot $K_\tau$, up to normalization:
\[ S_{m,n,\ell}(q,t,a)= (1-q)(at^{-1/2}q^{-1/2})^{-|\tau|} P_{K_{\tau}}^{\KR}(q,t,a).\]
In particular, $S_{m,n,\ell}(q,t,a)$ depends only on the triangular partition $\tau$.
\end{theorem}
\begin{proof}
Let $m,n,\ell$ and $\tau=\tau_{m,n,\ell}$ be as in the statement.  Abbreviate by writing $\mathbf{u}=\mathbf{u}(m,n,\ell)$ and $\mathbf{v}=\mathbf{v}(m,n,\ell)$.  Lemma \ref{lem:K-L isotopic} gives us that
\[
\KRH(K_{\tau}) = \KRH(K_{\mathbf{u},\mathbf{v}}).
\] 
Using Proposition \ref{prop:Ktau} to compute the appropriate normalizing factor $\d(\b)$, we observe the following
\[
\PC_{K_{\tau}}^{\KR}(q,t,a)= \PC_{K_{\mathbf{u},\mathbf{v}}}^{\KR}(q,t,a) = \frac{(at^{-1/2}q^{-1/2})^{|\tau|}}{1-q} t^{|\tau|} R_{\mathbf{u},\mathbf{v}}(q,t,a)= \frac{(at^{-1/2}q^{-1/2})^{|\tau|}}{1-q} S_{m,n,\ell}(q,t,a),
\]
as claimed. In the first equality we used Lemma \ref{lem:K-L isotopic}, in the second we used Theorem \ref{thm:RHHH}, and in the final equality we used Theorem \ref{thm:Qrecursion}.
\end{proof}

The above theorem justifies the following definition.

\begin{definition}\label{def:S tau}
For $\tau$ a triangular partition, we define 
\begin{equation}
S_\tau(q,t,a) := S_{m,n,\ell}(q,t,a), 
\end{equation}
where $(m,n,\ell)$ is any triple as in Section \ref{subsec:m,n-triangular} for which $\tau_{m,n,\ell} = \tau$.
\end{definition}
Thus, as discussed in Remark \ref{rem:Catalan}, specializing at $a=0$ yields the triangular $(q,t)$-Catalan polynomials of \cite{BM,BHMPS}:
\[S_{\tau}(q,t,0)=C_\tau(q,t).\]
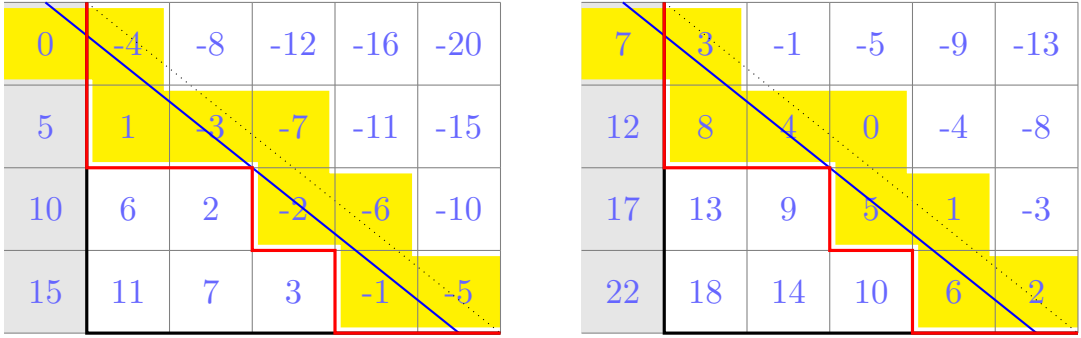
\begin{figure}[ht]
\center
\begin{tikzpicture}[scale=1.1]
\fill [fill=gray!20!white] (-1,0) rectangle (0,4);
\draw[line width = 27pt, yellow](-1,3.5)--(.5,3.5)--(.5,2.5)--(2.5,2.5)--(2.5,1.5)--(3.5,1.5)--(3.5,.5)--(5,.5);\
%
 \pgfmathtruncatemacro{\m}{4} 
  \pgfmathtruncatemacro{\n}{5} 
 \pgfmathtruncatemacro{\mi}{\m-1} 
    \pgfmathtruncatemacro{\ni}{\n-1} 
  \foreach \x in {0,...,\n} 
    \foreach \y in {1,...,\m}     
       {\pgfmathtruncatemacro{\label}{ - \m * \x - \n *  \y +\m*\n} 
       \node  at (\x-.5,\y-.5) [blue!60!white, scale=1.2]{\label};} 
\draw[dotted] (0,\m)--(\n,0); 
  \foreach \x in {0,...,\n} 
    \foreach \y  [count=\yi] in {0,...,\mi}   
      \draw[gray] (\x,\y)--(\x,\yi);
        \foreach \y in {0,...,\m}     
       \foreach \x  [count=\xi] in {-1,...,\ni}   
         \draw[gray] (\x,\y)--(\xi-1,\y);
\draw[ thick, blue] (-.5,4)--(4.5,0); 
\draw [very thick] (5,0)--(0,0)--(0,4); 
\draw[ very thick,red](0,4)--(0,2)--(2,2)--(2,1)--(3,1)--(3,0)--(5,0); 
\end{tikzpicture}
\begin{tikzpicture}[scale=1.1]
\fill [fill=gray!20!white] (-1,0) rectangle (0,4);
\draw[line width = 27pt, yellow](-1,3.5)--(.5,3.5)--(.5,2.5)--(2.5,2.5)--(2.5,1.5)--(3.5,1.5)--(3.5,.5)--(5,.5);\
%
 \pgfmathtruncatemacro{\m}{4} 
  \pgfmathtruncatemacro{\n}{5} 
 \pgfmathtruncatemacro{\mi}{\m-1} 
    \pgfmathtruncatemacro{\ni}{\n-1} 
  \foreach \x in {0,...,\n} 
    \foreach \y in {1,...,\m}     
       {\pgfmathtruncatemacro{\label}{ - \m * \x - \n *  \y +\m*\n+7} 
       \node  at (\x-.5,\y-.5) [blue!60!white, scale=1.2]{\label};} 
\draw[dotted] (0,\m)--(\n,0); 
  \foreach \x in {0,...,\n} 
    \foreach \y  [count=\yi] in {0,...,\mi}   
      \draw[gray] (\x,\y)--(\x,\yi);
        \foreach \y in {0,...,\m}     
       \foreach \x  [count=\xi] in {-1,...,\ni}   
         \draw[gray] (\x,\y)--(\xi-1,\y);
\draw[ thick, blue] (-.5,4)--(4.5,0); 
\draw [very thick] (5,0)--(0,0)--(0,4); 
\draw[ very thick,red](0,4)--(0,2)--(2,2)--(2,1)--(3,1)--(3,0)--(5,0); 
\end{tikzpicture}
\caption{ The triangular partition $\tau=(3,2)=\tau_{5,4,18}=\tau_{18/4,18/5}=\tau_{8,5,26}$ and the Anderson labels denoted in blue (left) and the same diagram with the labels shifted up by $7$ (right).}\label{fig:5-4}
\end{figure}

\begin{example}
Consider $(m,n,\ell)=(5,4,18)$ and notice that $\tau_{5,4,18}=(3,2) = \tau_{8,5,26}$ (See Example \ref{ex: (5.28,3.3)}). Then, $W_{5,4,18} = [-7,1]$ so that ${\bf w}(5,4,18) = 000000010$ and ${\bf x}(5,4,18)=00010$, ${\bf y}(5,4,18)=0010$. As pictured in Figure \ref{fig:5-4}, shifting the sets in $\Inv_{5,4,18}$ does not generate negative numbers, hence:
\begin{align*}
\Inv_{00010,0010} &= \{ \Delta \in \Inv_{5,4} \; | \; \Delta \in (\Inv_{5,4,18}^0 +7) \cap \Z_{\geq 0} \} \\
&= \{ \Delta \in \Inv_{5,4} \; | \; \Delta \in \Inv_{5,4,18}^0 +7 \}\\
&= \Inv_{5,4,18}^0 +7.
\end{align*}
Hence, the nine subdiagrams of $\tau_{5,4,18}$ give rise to the following sets in $\Inv_{00010,0010} $:
\[
\begin{array}{lllccl}
& \underline{\cogennn} & \underline{4-\gen} & \underline{\area'} &\underline{\codinv'} & \\
\Z_{\geq7} \setminus \{8,9,10,13,14,18\} & \{18\} & \{7,12,17,22\}& 5 &5 &(1+a) \\
 \Z_{\geq7} \setminus \{8,9,10,13,14\} &\{13,14\} & \{7,12,17,18\}&4&4 &(1+a)(1+at^{-1}) \\
  \Z_{\geq7} \setminus \{8,9,10,13\} &\{10,13\} & \{7,12,14,17\}&3&3&(1+a)(1+at^{-1})\\
 \Z_{\geq7} \setminus \{8,9,10,14\} &\{8,14\} & \{7,12,13,18\}&3&4&(1+a)(1+at^{-1})\\
  \Z_{\geq7} \setminus \{8,9,10\} &\{8,9,10\}& \{7,12,13,14\}&2&2&(1+a)(1+at^{-1})(1+at^{-2})\\
   \Z_{\geq7} \setminus \{8,9,13\}&\{6,13\}& \{7,10,12,17\}&2&3&(1+a)(1+at^{-1})\\
\Z_{\geq7} \setminus \{8,10\} &\{ 8,10\}& \{7,9,12,14\}&1&2&(1+a)(1+at^{-1})\\
 \Z_{\geq7} \setminus \{8,9\} &\{6,8,9\}& \{7,10,12,13\}&1&1&(1+a)(1+at^{-1})(1+at^{-2})\\
   \Z_{\geq7} \setminus \{8\} & \{5,6,8\}  & \{7,9,10,12\}&0&0&(1+a)(1+at^{-1})(1+at^{-2}).
   \end{array}
\]    
A straightforward computation confirms that $Q_{00010,0010} (q,t,a) = Q_{01000\bullet 0, 0010}(q,t,a)$. Thus, comparing with Example \ref{ex: (5.28,3.3)}, we indeed obtain that 
\[S_{(3,2)}(q,t,a)=S_{5,4,18}(q,t,a) = S_{8,5,26} (q,t,a).\]
In particular, specializing at $a=0$ yields the triangular $(q,t)$-Catalan polynomial:
\begin{align*}
C_{(3,2)}(q,t)=S_{(3,2)}(q,t,0)=q^5+q^4t+q^3t^2
+q^3t+q^2t^3
+q^2t^2+qt^3
+qt^4+t^5.    
\end{align*}
\end{example}

\subsection{Compactified Jacobians and the Oblomkov-Rasmussen-Shende Conjecture}\label{subsec:ORS}

Let $Z\subset\mathbf{C}^2$ be a complex curve given by an analytic equation $f(x,y)=0.$ Assume that $(0,0)\in Z.$ It is well known that for a small enough $\epsilon>0$, the sphere $S_\epsilon:=\{\|(x,y)\|=\epsilon\}$  intersects $Z$ transversely. The intersection is a smooth compact real curve inside the three dimensional sphere $S_\epsilon,$ in other words, a link. The links that appear this way are called \newword{algebraic links}. The study of algebraic links goes back almost a century to the works of Brauner \cite{B28} and K\"ahler \cite{K29}. For a more modern treatment of the subject, see \cite{EN85}. 

In the last several decades, there has been considerable interest in relating various invariants of the planar curve singularities to the invariants of the corresponding algebraic links. In \cite{ORS} Oblomkov, Rasmussen, and Shende conjectured a formula expressing the Poincar\'e  series of the Khovanov--Rozansky homology of an algebraic link in terms of the weight polynomials of the Hilbert schemes of points of the singularity. Migliorini and Shende \cite{MS13}, and independently Maulik and Yun \cite{MY14}, related the cohomology of the Hilbert scheme of a singular complex curve to the cohomology of its \newword{compactified Jacobian} equipped with a certain canonical filtration, called the \newword{perverse filtration}. Combined with these results, the \newword{Oblomkov-Rasmussen-Shende conjecture} predicts that the minimal $a$-degree part of the Poincar\'e  series of the Khovanov--Rozansky homology of a link of a singularity can be explicitly recovered from the cohomology of the compactified Jacobian of the singularity, equipped with the perverse filtration. The perverse filtration, which is responsible for the `$q$' degree, is notoriously hard to compute.  For this reason, it is simpler to forget the perverse filtration and specialize $q=1$.  However, Khovanov--Rozansky series have denominators of the form $\frac{1}{1-q}$, which must be cleared before specializing.  The modified version of the Oblomkov-Rasmussen-Shende conjecture is stated below.

\begin{conjecture}\label{conj:ors at q=1}
Let $Z=\{f(x,y)=0\}\subset \C^2$ be a curve with a unibranch singularity at $0$, and let $K=K_Z$ denote its link (actually a knot since $0\in Z$ is a unibranch singularity).  Let $\d$ be the Milnor number of $0\in Z$, $\overline{JC}$ denote its compactified Jacobian, and let $\PC_{\overline{JC}}(u)$ be its Poincar\'e  polynomial.  Then
\[
\Big((1-q)\PC^{\KR}_K(q,t,a)\Big)|_{\substack{q=1\\ t=u^{-2}}} = a^\d u^{-\d} \PC_{\overline{JC}}(u) + (\text{terms with $a$-degree $>\d$}).
\]
\end{conjecture}

We refer the reader to \cite{GM13,MS13,MY14,ORS} and the references therein for the definitions of compactified Jacobians, Hilbert schemes, perverse filtrations, and all the relevant constructions.

Recently, in \cite{GMO}, the fourth author together with Gorsky and Oblomkov studied the compactified Jacobians of the planar curves that can be parametrized by
\begin{equation}\label{eq:parametrization}
 (x(t),y(t))=  (t^{nd}, t^{md}+ At^{md+1}+ \dots),
\end{equation}
where $n$ and $m$ are positive relatively prime integers, $d$ is an integer greater than $1,$ and $A\neq 0$ is a non-zero complex number. They proved that the compactified Jacobian in this case admits an affine paving with the cells enumerated by the $dm,dn$-Dyck paths.  Assembling their computation with some well-known facts in conjunction with our Theorem \ref{thm: schroder=poincare}, establishes the following.

\begin{theorem}\label{thm: ORS for gen curves}
    Given integers $m,n,d\geq 1$ with $m,n$ coprime, let $Z$ be the curve with parametrization as in \eqref{eq:parametrization}, and let $\overline{JC}$ be the compactified Jacobian of the singularity $0\in Z$, and let $\PC_{\overline{JC}}(u)$ be its Poincar\'e  polynomial.  Then:
    \begin{enumerate}
        \item the Milnor number of $0\in Z$ is $\d=\d(md,nd)$,
        \item the link of $0\in Z$ is $K_{\tau_{md,nd}}$, which equals the $(d,mnd+1)$-cable of $T(m,n)$,
        \item Conjecture \ref{conj:ors at q=1} is true for the $(d,mnd+1)$-cable of $T(m,n)$.
    \end{enumerate} 
\end{theorem}
\begin{proof}
Statement (1) is well-known.  The fact that the link of $0\in Z$ is the $(d,dnm+1)$-cable of the torus knot $T(m,n)$ is well-known (see \cite{EN85}).  The fact that this knot coincides with $\tau_{md,nd}$ was first established by Galashin-Lam \cite{GL23}, but it also follows from Lemma \ref{lem:K-L isotopic}. 

From statement (2), we have the computation (letting $\tau=\tau_{md,nd}$) from Theorem \ref{thm: schroder=poincare}
\[
\Big((1-q)\PC_{K_{\tau}}(q,t,a)\Big)|_{\substack{q=1\\t=u^{-2}}}
= (au)^\d C_{\tau}(1,u^{-2}) +(\text{h.o.t.}),
\]
where $C_\tau(q,t)=S_\tau(q,t,0)$ as in Corollary \ref{cor:Catalan}.  Comparing this with the computation in \cite[Theorem 1.4]{GMO}
\begin{equation*}
        \PC_{\overline{JC}}(u)=\sum_{\pi\in \D(md,nd)}u^{2\codinv(\pi)}=u^{2\delta}C_{\tau}(1,u^{-2}),
    \end{equation*}
proves statement (3).
\end{proof}

\section{A $(q,t)$-Schr\"oder Theorem for Paths Under Any Line}\label{sec:qtSchroderTheorem}

The \emph{classical shuffle theorem}, conjectured by Haglund et al \cite{HHLRU} and proven by Carlsson and Mellit \cite{CM18}, gives a combinatorial formula in terms of $(n,n+1)$-Dyck paths for the action of the operator $\nabla$ on the $n^{th}$ elementary symmetric function $e_n$, which had previously been proven by Haiman \cite{Haiman, Haiman-Vanishing} to equal the Frobenius character of the ring of diagonal coinvariants. It was conjectured by Egge et. al. \cite{EHKK03} and proven by Haglund \cite{Ha04} that the \newword{$(q,t)$-Schr\"oder theorem} holds, 
\[
S_{n,n+1}(q,t,a)= \sum_{k\geq 0} S_{n,k}(q,t) a^k= \sum_{k\geq 0} \langle \nabla e_n, h_ke_{n-k} \rangle \; a^k, 
\]
where $S_{n,k}$ are the classical Schr\"oder polynomials defined in terms of so-called Schr\"oder paths \cite{S1870, EHKK03}. 

More recently, building upon the proof of the shuffle theorem and its rational generalization \cite{GorskyNegut, Mellit-Shuffle}, Blasiak-Haiman-Morse-Pun-Seelinger \cite{BHMPS} proved a vast generalization of the shuffle theorem for Dyck paths under lines of any real slope. In what follows we prove a complete generalization of the $(q,t)$-Schr\"oder Theorem by showing that our triangular Schr\"oder polynomial $S_{\tau}(q,t,a)$ arises as the sum of the hook components of the shuffle theorem under any line. Namely, we will prove that 
\begin{equation}\label{eq:qt-Schroder}
S_{\tau}(q,t,a)= \sum_{k\geq 0} \langle \Hik_{\tau}(X;q,t), h_ke_{\lfloor s \rfloor-k} \rangle \;a^k,  
\end{equation}
where $\Hik_\tau(X;q,t)$ is the symmetric function obtained by the action of particular elliptic Hall algebra elements indexed by a triangular partition $\tau$ on $1$.

\subsection{The Shuffle Theorem Under Any Line} We begin by recalling some essential definitions from \cite{BHMPS} and refer the reader to this source for additional details. 

Let $\nu=(\nu_{1}, \dots,\nu_{\ell})$ be a $\ell$-tuple of skew diagrams. For a fixed $\epsilon >0$ such that $\ell\epsilon <1$, we define the \emph{content} of a cell $\Box \in \nu_{j}$ with northeast corner $(x,y)$ to be $c(\Box) = x-y+j\epsilon$. A \emph{reading order} on $\nu$ is any total ordering of the boxes of $\nu$ in which $c$ is weakly increasing. Two cells $\Box_1$ and $\Box_2$ form an \newword{attacking pair} if $0 < |c(\Box_1) - c(\Box_2)|<1$. Let 
\[
\I(T):= \# \text{attacking pairs of } \nu.
\]
As usual, we denote by $\SSYT(\nu)$ and $\SYT(\nu)$ the sets of semistandard and standard Young tableaux of shape $\nu$. 

Given a semistandard filling $T \in \SSYT(\nu)$ with entries in $\Z_{\geq 0}$, we say an attacking pair $\Box_1 \in \nu_i$ and $\Box_2 \in \nu_j$ form an \newword{attacking inversion} if $\Box_1$ preceeds $\Box_2$ in the reading order and $T(\Box_1)>T(\Box_2)$. For $T \in \SSYT(\nu)$, set
\[
\inv(T):= \# \text{attacking inversions of } T.
\]
\begin{definition}
The \newword{LLT-polynomial} indexed by a tuple of skew shapes $\nu$ is the generating function, 
\[
\G_\nu(X;t):= \sum_{T \in \SSYT(\nu)} t^{\inv(T)}X^T. 
\]
\end{definition}
It was originally shown in \cite{LLT97} that the LLT polynomials are symmetric functions. Later, a more elementary proof was provided in \cite{HHL05}. 

\noindent\underline{\textbf{Notation:}} Given a triangular partition, we have so far chosen to index the associated objects by triples $(m,n,\ell)$. This was relevant given the importance of embedding our triangular partitions inside larger ones. In this section this choice of embedding plays no role, thus in order to facilitate comparisons with the work of the existing combinatorial literature \cite{Ha04,HHL05, HHLRU, BM, BHMPS}, we revert to the initial notation of \S\ref{sec:GMV-Recursions} and label our triangular partitions and associated objects by tuples of real numbers $(r,s)$ where $r=z/n$ and $s= z/m$, for some $z$ with $\lfloor z \rfloor = \ell$. 
\medskip

So then, for any $\pi = (\lambda, \tau_{r,s}) \in \D(r,s)$ with $\tr = \lfloor r \rfloor +1$, 
one can associate an $\tr$-tuple of skew-rows $\nu(\pi)$ as follows. 

Let $\tau_{r,s}' = (\tau'_1, \dots, \tau'_{\tr-1})$ and $\lambda'=(\lambda'_1 , \dots , \lambda'_{\tr-1})$ denote the transpose partitions of $\tau_{r,s}$ and $\lambda$ respectively. Set:
\begin{align*}
\tilde{\tau}=(\tilde{\tau}_1,\ldots,\tilde{\tau}_{\tilde{r}})&:= (\ts,\tau'_1,\dots,\tau'_{\tr-1}) \\
\tilde{\lambda}=(\tilde{\lambda}_1,\ldots,\tilde{\lambda}_{\tilde{r}})&:=(\ts,\lambda'_1,\dots,\lambda'_{\tr-1})
\end{align*}
and for each $1 \leq i \leq \tr$, let $h_{i} = s - \frac{s}{r}(i-1)- \tilde{\tau}_{i}$ and set $\gamma$ to be the $\tr$-tuple of skew \emph{rows} with parts, 
\[
 \gamma_i:= (\tilde{\tau}_i-\tilde{\lambda}_{i+1})\setminus (\tilde{\tau}_i-\tilde{\lambda}_{i}) \qquad  \text{and} \qquad \gamma_\tr= (\tilde{\tau}_\tr)\setminus (\tilde{\tau}_\tr-\tilde{\lambda}_{\tr})\qquad \text{ for } \qquad 1\leq i \leq \tr-1, 
\]
where $(\ell)$ denotes the row partition of size $\ell$. 

Let $\sigma$ be the permutation in $S_\tr$ that sorts $(h_\tr, \dots, h_1)$ into an increasing sequence, and define 
\[\nu(\pi):= \sigma\omega_0 (\gamma)= (\gamma_{\omega_0 \sigma^{-1}(1)}, \dots , \gamma_{\omega_0 \sigma^{-1}(\tr)}), \]
where $\omega_0$ is the longest word in $S_\tr$. That is, $\nu(\pi)$ is a tuple of (potentially overlapping) skew-rows all at height one obtained by taking the parts $\gamma_i$ of $\gamma$ and sorting them according to the relative size of $h_i$.  

\begin{remark}
Notice that the permutation $\sigma$ depends only on the triangular partition $\tau_{r,s}$ and the cutting line $L_{r,s},$ and not on the choice of subdiagram $\lambda \subseteq \tau_{r,s}$.
\end{remark}

\begin{example}
Consider $\tau_{4.5,3.6} =(3,2,0)$ from Figure \ref{fig:5-4} and suppose $\lambda = (3,1,0)$. Then, $\tilde{\tau} = (3,2,2,1,0)$ and $\tilde\lambda=(3,2,1,1,0)$. Hence, $\gamma=((1)\setminus (0),(1)\setminus (0),(1)\setminus (1),(1)\setminus (0),(0)\setminus (0))=((1),(1),(0),(1),(0))$ and $h=(0.6,0.8,0,0.2,0.4)$.
Consequently, (in one-line notation) $\sigma=(3\;2 \;1\;5 \;4)$ and $\nu(\pi)=
(4\;5\;1\;2\;3)((1),(1),(0),(1),(0)) = ((0),(1),(0),(1),(1))$.
\end{example}

Blasiak et al. also introduced the rotated tuple $\nu(\pi)^R$ obtained by taking each skew-row $\nu(\pi)_i$, transposing it, then rotating it by $180$ degrees, and finally shifting it so that for each box $\Box$ in the column the content $c(\Box)$ is the same as it was in $\nu(\pi)_i= \gamma_{\omega_0\sigma^{-1}(i)}$. Thus, $\nu(\pi)^R$ is a $\tr$-tuple of \emph{columns} and $\nu(\pi)$ a $\tr$-tuple of \emph{rows}. It follows from the definition that $I(\nu(\pi))=I(\nu(\pi)^R)$.

The \newword{shuffle theorem under any line} \cite[Thm. 5.5.1]{BHMPS} gives a combinatorial formula in terms of $\tau$-Dyck paths, with $\tau$ a triangular partition, for the action of a family of operators $\mathsf{D}_{\tau}$ in the elliptic Hall algebra on the element $1 \in \C[q,t][x_1,x_2,\dots]^{S_\infty}$. Combined with \cite[Prop. 4.1.6]{BHMPS} the shuffle theorem states:
\begin{equation}\label{eq:shuffletheorem}
\mathsf{D}_{\tau_{r,s}}(1) = \sum_{\pi \in \D(r,s)} q^{\area(\pi)} t^{\dinv(\pi)} \omega(\G_{\nu(\pi)}(X;t^{-1}))
= \sum_{\pi \in \D(r,s)} q^{\area(\pi)} t^{\dinv(\pi)} t^{-I(\nu(\pi))}\G_{\nu(\pi)^R}(X;t),
\end{equation}
where $\omega$ is the standard Weyl involution on symmetric functions.

\begin{definition}\label{def:Qsym}
For a subset $J\subseteq \{1,\ldots,n-1\}$, Gessel's \newword{fundamental quasisymmetric function} is given by
\begin{equation}
    Q_{n,J}(X)=\sum_{\substack{a_1\le\ldots\le a_n\in\mathbb{Z}_{\ge 0}\\ a_i=a_{i+1}\Rightarrow i\notin J}} x_{a_1}x_{a_2}\ldots x_{a_n}.
\end{equation}

\end{definition}
It was shown in the Appendix of \cite{HHL05} that for $S \in \SYT(\nu)$ with \emph{descent set} 
\[
d(S):= \{ i \in \Z_{\geq 0} \;|\; S(\Box_1) = i+1,\; S(\Box_2)=i\; \text{ and } \Box_1 \text{ preceeds } \Box_2 \text{ in reading order}\}
\]
we can write the following:
\begin{equation}\label{eq:llt-gessel}
\G_{\nu}(X;t)= \sum_{S \in \SYT(\nu)} q^{\inv(S)} Q_{n, d(S)}(X). 
\end{equation}
Thus, denoting by $\Hik_{\tau_{r,s}}(X;q,t)$ the combinatorial side of the shuffle theorem \eqref{eq:shuffletheorem}, from \eqref{eq:llt-gessel} we have, 
\begin{align}\label{eq:H-Gessel}
\Hik_{\tau_{r,s}}(X;q,t) = \sum_{\pi \in \D(r,s)} q^{\area(\pi)} t^{\dinv(\pi)} t^{-I(\nu(\pi))} \sum_{S \in \SYT(\nu(\pi)^R)}t^{\inv(S)}Q_{n,d(S)}(X).
\end{align}

\subsection{Superization}
Consider the alphabet $\aA=\aA_X\sqcup \aA_Y$ consisting of two copies of the set of positive integers: $\aA_X=\{1_X<2_X<\ldots\}$ and $\aA_Y=\{1_Y>2_Y>\ldots\}$, ordered so that $k_X<k_Y$ for all $k \in \Z_{>0}$.  
\begin{definition}
A \newword{semistandard super tableaux} $T \in \SSYT_{\pm}(\nu)$ is a filling of $\nu$ with entries in $\aA$, such that the elements of $\aA_X$ are strictly increasing in columns and the elements of $\aA_Y$ are strictly increasing in rows, with the entries weakly increasing in rows and columns otherwise. Denote by $\SSYT_{\pm}(\nu; \mu,\eta)$ the set of all semistandard super tableaux of shape $\nu$, with $\aA_X$-weight equal to $\mu$ and $\aA_Y$-weight equal to $\eta$. 
\end{definition}

\begin{remark}
    When $\mu=(a)$ and $\eta=(b)$ are one part partitions, we will omit the parenthesis and write $\SSYT_{\pm}(\nu; a,b)$ instead of $\SSYT_{\pm}(\nu; (a),(b))$ to avoid cumbersome notation. 
\end{remark}

Two attacking cells $\Box_1$ and $\Box_2$ in $\nu$ form an \newword{attacking inversion} in $T \in \SSYT_{\pm}(\nu)$ if $\Box_1$ preceeds $\Box_2$ in reading order and either $T(\Box_1)>T(\Box_2)$ or $T(\Box_1)=T(\Box_2) \in \aA_Y$. As before $\inv(T)$ counts the number of attacking inversions of $T$. 
\smallskip

In the usual way, to every super tableaux $T\in\SSYT(\nu;\mu,\eta)$ we associate a monomial $Z^T:=X^\mu Y^\eta$. Given any symmetric function $f(X)$, its \newword{superization} is the following function, 
\begin{equation}
     \widetilde{f}(Z)=\widetilde{f}(X,Y):=\omega_Y f[X+Y],
\end{equation}
where $f[X+Y]$ denotes the plethystic evaluation of $f$ on the formal series $X+Y$ and $\omega_Y$ is Weyl involution on $\C[Y]^{S_\infty}$. While we omit all details of plethystic substitution in this article, we refer the reader to \cite{HaglundCatalan} for an introductory treatment of the subject.  

The superization of the LLT-polynomial is given by (see the Appendix of \cite{HHL05}):
\begin{equation}\label{eq:superLLT}
\widetilde{\G}_\nu(Z;t):= \sum_{T \in \SSYT_\pm(\nu)} t^{\inv(T)}Z^T. 
\end{equation}
Similarly, given $J\subseteq\{1,\ldots,n-1\}$, the \emph{super fundamental quasisymmetric functions} are defined as, 
\begin{equation}\label{equation:super Gessel}
    \widetilde{Q}_{n,J}(Z):=\sum_{\substack{a_1\leq\ldots\leq a_n\in\aA\\ a_k=a_{k+1}\in\aA_X\Rightarrow i\notin J\\
    a_k=a_{k+1}\in\aA_Y\Rightarrow i\in J}} z_{a_1}z_{a_2}\ldots z_{a_n}.
\end{equation}

In \cite[eq. (17)]{HHLRU}, Haglund et al. derive the following identity, 
\begin{equation}\label{eq:scalar prod to coef of sup}
\langle f(X) , h_\mu e_\eta \rangle = \widetilde{f}(X,Y)|_{X^\mu Y^\eta}.
\end{equation}
Hence, in order to compute the right hand side of \eqref{eq:qt-Schroder}, we must first compute $\widetilde{\Hik}_{\tau_{r,s}}(X;q,t)$.

\begin{proposition}\cite[Cor. 2.4.3]{HHLRU} \label{prop:super}
Let $f$ be any degree $n$ homogeneous symmetric function with expansion into fundamental quasisymmetric functions given by
\begin{equation}
    f(X)=\sum_J f_J Q_{n,J}(X).
\end{equation}
Then, its superization $\widetilde{f}$ expands into super fundamental quasisymmetric functions as follows: 
\begin{equation}
    \widetilde{f}(Z)=\sum_J f_J \widetilde{Q}_{n,J}(Z).
\end{equation}
\end{proposition}

Consequently, applying the previous Proposition to both \eqref{eq:H-Gessel} and \eqref{eq:llt-gessel}, we derive that, 
\begin{align*}
\widetilde{\Hik}_{\tau_{r,s}}(X;q,t) & = \sum_{\pi\in \D(r,s)} q^{\area(\pi)} t^{\dinv(\pi)} t^{-I(\nu(\pi))} \sum_{S \in \SYT(\nu(\pi)^R)}t^{\inv(S)}\widetilde{Q}_{n,d(S)}(X)
\\
&\overset{\eqref{eq:llt-gessel}}{=} \sum_{\pi\in \D(r,s)} q^{\area(\pi)} t^{\dinv(\pi)} t^{-I(\nu(\pi))}\widetilde{\G}_{\nu(\pi)^R}(Z;t)\\
\\
&\overset{\eqref{eq:superLLT}}{=} \sum_{\pi\in \D(r,s)} q^{\area(\pi)} t^{\dinv(\pi)} t^{-I(\nu(\pi))} \sum_{T \in \SSYT_{\pm}(\nu(\pi)^R)}t^{\inv(T)}Z^T.
\end{align*}
Thus, from \eqref{eq:scalar prod to coef of sup} we have,
\[
\langle \Hik_{\tau_{r,s}}(X;q,t), h_{\mu}e_{\eta}\rangle = \widetilde{\Hik}_{\tau_{r,s}}(X;q,t)|_{X^\mu Y^\eta} = \sum_{\pi\in \D(r,s)} q^{\area(\pi)} t^{\dinv(\pi)} t^{-I(\nu(\pi))} \sum_{T \in \SSYT_{\pm}(\nu(\pi)^R;\mu,\eta)}t^{\inv(T)},
\]
which specialized at $\mu = (k)$ and $\eta = (\ts-k)$ yields,
\begin{equation}\label{eq:RHS}
\langle \Hik_{\tau_{r,s}}(X;q,t), h_{k}e_{\ts-k}\rangle = \sum_{\pi\in \D(r,s)} q^{\area(\pi)} t^{\dinv(\pi)} \sum_{T \in \SSYT_{\pm}(\nu(\pi)^R;k,\ts-k)}t^{\inv(T)-I(\nu(\pi))}.
\end{equation}

\subsection{The $(q,t)$-Schr\"oder Theorem}
In order to relate equality \eqref{eq:RHS} to our Schr\"oder polynomials, we need to express it in terms of fillings of $\rho_\pi:=(\lambda+1^{\ts})\setminus \lambda$, where $\pi=(\lambda, \tau_{r,s})$. To do so, we recall the bijection \cite[\S 6]{BHMPS}: 
\begin{align}
\Psi: \SSYT_{\pm}(\nu(\pi)^R) &\to \SSYT_{\pm}(\rho_\pi) \label{eq:bijSSYT}\\
T &\mapsto P_T, \nonumber
\end{align}
which is obtained by mapping the $i^{th}$ component $\nu(\pi)^R_i$ of $\nu(\pi)^R$ to the $\sigma\omega_0(i)^{th}$ column of $\rho_\pi$ and preserving the corresponding labels. 

To understand how the associated statistics are modified under this bijection, we introduce the following definitions.

\begin{definition}\label{def:phi}
Given a cell $\Box \in \rho_\pi$ with \emph{northwest} coordinate $(a,b)$, define $\phi(\Box)$ to be the vertical distance between $(a,b)$ and the line $L_{r,s}$, so that 
\[
\phi(\Box) := s-s/r(a)-b.
\]
\end{definition}
\begin{definition} \label{def:invSSTY}
We say two cells $\Box_1$ and $\Box_2$ in $\rho_\pi$ are an \newword{attacking pair} if $0<|\phi(\Box_1)-\phi(\Box_2)|<1$, and denote by $\I(\pi)$ the number of attacking pairs in $\rho_\pi$.  

Similarly, given $P \in \SSYT_\pm(\rho_\pi)$, we say an attacking pair $\Box_1, \Box_2 \in \rho_\pi$ is an \newword{attacking inversion} if $\phi(\Box_1)<\phi(\Box_2)$ and either $P(\Box_1)>P(\Box_2)$ or $P(\Box_1)=P(\Box_2) \in \aA_Y$. Denote by $\inv(P)$ the number of attacking inversions of $P$. 
\end{definition}

\begin{lemma}\label{lem:inv}
For any $T \in \SSYT_{\pm}(\nu(\pi)^R)$ and $P_T = \Psi(T) \in \SSYT_{\pm}(\rho_\pi)$, we have the following: 
\[
\I(\nu(\pi)^R) = \I(\pi)\qquad \text{and} \qquad \inv(T) = \inv(P_T).
\]
\end{lemma}
\begin{proof}
Since $I(\nu(\pi))=I(\nu(\pi)^R)$ it suffices to consider attacking pairs in $I(\nu(\pi))$.

We begin by noting that from the bijection \ref{eq:bijSSYT} it follows that to any box $\Box$ in $\nu(\pi)_{j}$ with NE corner $(x,1)$ the corresponding box $\widetilde{\Box}$ in $\rho_\pi$ will have $\phi(\widetilde{\Box}) = x-1+h_{\omega_0\sigma^{-1}(j)}$. Moreover, given any two components $\nu_i$ and $\nu_j$ then $i<j$ if and only if $h_{\omega_0\sigma^{-1}(i)}<h_{\omega_0\sigma^{-1}(j)}$. 

Let $\Box_1$ and $\Box_2$ be a pair of boxes in $\nu(\pi)$ and let $\widetilde{\Box}_1$ and $\widetilde{\Box}_2$ be the corresponding boxes in $\rho_\pi.$ Let $(x_1,1)$ be the NE corner of $\Box_1 \in \nu(\pi)_{j_1}$ and $(x_2,1)$ be the NE corner of $\Box_2 \in \nu(\pi)_{j_2}.$ Boxes $\Box_1$ and $\Box_2$ form an attacking pair if and only if
\begin{equation*}\label{eq:XX}
|c(\Box_1)-c(\Box_2)|=|(x_1-x_2)+(j_1-j_2)\epsilon| <1, 
\end{equation*}
while boxes $\widetilde{\Box}_1$ and $\widetilde{\Box}_2$ for an attacking pair if and only if
\begin{equation*}\label{eq:attacking Y}
|\phi(\widetilde{\Box}_1)-\phi(\widetilde{\Box}_2)|=|(x_1-x_2)+(h_{\omega_0\sigma^{-1}(j_1)}-h_{\omega_0\sigma^{-1}(j_2)})| <1. 
\end{equation*}
To show that $\I(\nu(\pi)^R) = \I(\pi)$ it remains to notice that the two conditions are equivalent. Indeed, $(x_1-x_2)$ is an integer, and the correction terms $(j_1-j_2)\epsilon$ and $(h_{\omega_0\sigma^{-1}(j_1)}-h_{\omega_0\sigma^{-1}(j_2)})$ always have the same sign and are less than $1$ by absolute value.

For the second claim we recall the bijection in Proposition 4.1.6 in \cite{BHMPS}:
\[ \SSYT_\pm(\nu(\pi)) \to \SSYT_\pm(\nu(\pi)^R) \qquad\qquad U \mapsto U^R\]
obtained by reflecting a tableau and exchanging positive entries in $\aA_X$ with negative entries $\aA_Y$. In particular, Blasiak et. al. show that under this bijection $\inv(U^R) = I(\nu(\pi)) - \inv(U)$. Thus, it suffices to prove that attacking inversions in $\SSYT_\pm(\rho_\pi)$ correspond to attacking non-inversions of $\SSYT_\pm(\nu(\pi))$. 

So then, given a pair of attacking cells $\Box_1$ and $\Box_2$ in $\nu(\pi)$ notice that the attacking condition ensures that $c(\Box_1) \neq c(\Box_2)$. Hence, these cells form an attacking non-inversion if $c(\Box_1)<c(\Box_2)$  with either $U(\Box_1)<U(\Box_2)$ or $U(\Box_1)=U(\Box_2) \in \aA_X$. Since by definition the alphabets are ordered such that $\{1_X<2_X<\dots\}$ and $\{1_Y>2_Y>\dots\}$ with $ k_X<k_Y$ for any $k \in \N$, then under under $\Psi$, these conditions correspond to $\phi(\widetilde{\Box}_1)<\phi(\widetilde{\Box}_2)$  with either $P_{U^R}(\widetilde{\Box}_1)>P_{U^R}(\widetilde{\Box}_2)$ or $P_{U^R}(\widetilde{\Box}_1)=P_{U^R}(\widetilde{\Box}_2) \in \aA_Y$, as desired.

\end{proof}
Thus, in light of Lemma \ref{lem:inv}, we may rewrite \eqref{eq:RHS} as follows:
\begin{equation}\label{eq:hook}
\langle \Hik_{\tau_{r,s}}(X;q,t), h_{k}e_{\ts-k}\rangle = \sum_{\pi \in \D(r,s)} q^{\area(\pi)} t^{\dinv(\pi)} \sum_{P \in \SSYT_{\pm}(\rho_\pi;k,\ts-k)}t^{\inv(P)-I(\pi)} .
\end{equation}
Now, recall from \eqref{eq:Schroder-dyckpaths} that
\begin{align*}
S_{\tau_{r,s}}(q,t,a)&=\sum_{\pi \in \D(r,s)}
q^{\area(\pi)}t^{\dinv(\pi)} \sum_{L\subset\AB(\pi)}
a^{|L|} t^{-\sum_{\Box\in L} \xi(\pi,\Box)}\\
&= \sum_{k\geq 0} a^k \left(\sum_{\pi \in \D(r,s)}
q^{\area(\pi)}t^{\dinv(\pi)} \sum_{\substack{
L\subset\AB(\pi)\\ |L|=k}}
t^{-\sum_{\Box\in L} \xi(\pi,\Box)}\right).
\end{align*}
Thus, \eqref{eq:qt-Schroder} is an immediate consequence of the subsequent result. 
\begin{theorem} The following equality holds:
\[
 \sum_{P \in \SSYT_{\pm}(\rho_\pi;k,\ts-k)}t^{\inv(P)-I(\pi)} = 
\sum_{\substack{
L\subset\AB(\pi)\\ |L|=k}}
t^{-\sum_{\Box\in L} \xi(\pi,\Box)}.
 \]
\end{theorem}

\begin{proof}
Consider any $P \in \SSYT_{\pm}(\rho_\pi; k , \ts -k)$. Then by definition, since entries in $\aA_X$ are strictly increasing on columns and $P$ has exactly $k$ boxes with entries $1_X$ and $\ts -k$ boxes with entries $1_Y$, then any column of $\rho_\pi$ can have at most one cell with value $1_X$. Thus, any choice of $P \in \SSYT_{\pm}(\rho_\pi; k , \ts -k)$ is equivalent to a choice of $k$ columns of $\rho_\pi$ on which the bottom-most entry is filled with $1_X$, with all remaining entries filled with $1_Y$. Evidently then, since the set bottom-most cells in $\rho_\pi$ are precisely the addable boxes of $\lambda$, there is an obvious bijection between 
\[
\SSYT_{\pm}(\rho_\pi; k , \ts -k) \to \mathcal{P}(\AB(\pi))
\]
that maps each tableaux $P$ to the subset $L_P$ of boxes of $\rho_\pi$ labeled by $1_X$ in $P$. 

Thus, it remains to show that for any such $P$ we have $I(\pi) - \inv(P) = \sum_{\Box \in L_P}\xi(\pi, \Box)$. Given the restriction on the fillings it follows that the set attacking non-inversions of $P$ corresponds to attacking cells $\Box_1, \Box_2 \in \rho_\pi$ satisfying $\phi(\Box_1) < \phi(\Box_2)$ with $P(\Box_1) = 1_X$. 

Fix a $P \in \SSYT_\pm(\rho_\pi)$ and recall from Definition \ref{def:xi-AB} that for any positive integers $m$ and $n>s$ with $n/m=s/r$ if $\Box_1 \in L_P \subset \AB(\pi)$, we have, 
\[
 \xi(\pi,\Box_1)=\left|\{\Box\in\EB(\pi)\;|\;\gamma(\Box_1)+n<\gamma(\Box)\le \gamma(\Box_1)+n+m\}\right|.
\]

However, since the cells of $\rho_\pi$ are in bijection with those in $\EB(\pi)$, i.e. for every cell $\Box \in \EB(\pi)$ with NE corner $(a,b)$ there is a cell $\Box' \in \rho_\pi$ with NW corner $(a,b)$ satisfying $\gamma(\Box) = \gamma(\Box')+n$. Then,
\[
 \xi(\pi,\Box_1)=\left|\{\Box_2\in\rho_\pi\;|\;\gamma(\Box_1)<\gamma(\Box_2) \le \gamma(\Box_1)+m\}\right|.
\]
So, suppose $\Box_1 \in L_P$ with NW corner $(x_1,y_1)$ and $\Box_2 \in \rho_\pi$ with NW corner $(x_2,y_2)$. Then, if $\Box_2 \in \rho_\pi$ satisfies $\gamma(\Box_1)<\gamma(\Box_2) \leq \gamma(\Box_1)+m$,
we have the following:
\[\begin{array}{rcccl}
mn-my_1-n(x_1+1)&<&mn-my_2-n(x_2+1) &\leq& mn-my_1-n(x_1+1)+m \\
-my_1-nx_1&<&-my_2-nx_2 &\leq&-my_1-nx_1+m\\
 0&<&m(y_1-y_2) + n(x_1-x_2)&\leq& m \\
 0&<&(y_1-y_2) + \frac{n}{m}(x_1-x_2)&\leq& 1\\
0&<& \phi(\Box_2) - \phi(\Box_1)&\leq& 1.
\end{array}\]
Note that the equality on the right can never occur. Indeed, since $n$ and $m$ are relatively prime and $|x_1-x_2|<m,$ that would require $x_1=x_2$ and $y_1=y_2+1,$ which means that $\Box_1$ lies directly on top of $\Box_2,$ but $\Box_2\in\rho_\pi$ and $\Box_1$ is an addable box. Thus, $\Box_1$ and $\Box_2$ are attacking with $\phi(\Box_1) < \phi(\Box_2)$. Since by construction $P(\Box_1)=1_X$, this completes the proof. 
\end{proof}

Consequently, we immediately obtain a $(q,t)$-Schr\"oder Theorem for paths under any line. 

\begin{corollary}\label{cor:Schroder-Shuffle}
The triangular $(q,t)$-Schr\"oder polynomial computes the hook components of the shuffle theorem for paths under any line. That is, 
\[
S_{\tau_{r,s}}(q,t,a)= \sum_{k\geq 0} \langle \Hik_{\tau_{r,s}}(X;q,t), h_ke_{\ts-k} \rangle \;a^k. 
\]
\end{corollary}

In particular, since $\Hik_{\tau_{r,s}}(X;q,t)$ depends solely on the choice of triangular partition and not the choice of slope, Corollary \ref{cor:Schroder-Shuffle} once again implies that the triangular Schr\"oder polynomial depends only on $\tau$.

\appendix

\section{An Alternative Derivation of the Hook Component}\label{sec:AppendixA}

In this section we provide an alternative derivation of \eqref{eq:hook} where the superization procedure is performed after the bijection from \eqref{eq:bijSSYT} has been applied, rather than before. This allows us to provide a self contained argument that does not rely on the proofs of \cite{HHLRU} or the appendix in \cite{HHL05}. 

Given a triangular partition $\tau = \tau_{r,s}$ with $\pi = (\lambda, \tau_{r,s}) \in \D(r,s)$ and $\rho_\pi=\left(\lambda+1^\ts\right)/\lambda$, let $\inv(T)$ with $T \in \SSYT(\rho_\pi)$ be given as in Definition \ref{def:invSSTY}. 

\begin{definition}
    The LLT polynomial corresponding to $\rho_\pi=\left(\lambda+1^\ts\right)/\lambda$ is given by the formula, 
    \begin{equation}
        \mathcal{L}_\pi(X,t):=\sum_{T\in\SSYT(\rho_\lambda)}t^{\inv(T)}X^T.
    \end{equation}
\end{definition}

It follows from Lemma \ref{lem:inv} that $\mathcal{L}_\pi(X,t) = \mathcal{G}_{\nu(\pi)^R}(X;t)$. Thus, we have:

\begin{align*}
\Hik_{\tau_{r,s}}(X;q,t):=&\sum_{\pi \in \D(r,s)} q^{\area(\pi)}t^{\dinv(\pi)}\omega(\mathcal{G}_{\nu(\pi)}(X;t^{-1}))\\
=&\sum_{\pi \in \D(r,s)} q^{\area(\pi)}t^{\dinv(\pi)-I(\nu(\pi))}\mathcal{G}_{\nu(\pi)^R}(X;t)\\
=&\sum_{\pi \in \D(r,s)} q^{\area(\pi)}t^{\dinv(\pi)-I(\pi)}\mathcal{L}_{\pi}(X;t)\\
=&\sum_{\pi \in \D(r,s)} q^{\area(\pi)}t^{\dinv(\pi)-I(\pi)}\sum_{T\in\SSYT(\rho_\pi)}t^{\inv(T)}X^T.
\end{align*}

\begin{remark}
    In the special case when $r=n$ is an integer and $s=n+\epsilon$, one obtains that $\dinv(\pi)=I(\pi)$, for all $\pi \in \D(n,n+\epsilon)$ and the formula above specializes into the right hand side of the classical Shuffle Theorem (see Formula (29) in \cite{HHLRU} and \cite{CM18}).
\end{remark}

\subsection{Superization}

Our goal now is to evaluate the Hall scalar product $\langle \Hik_{\tau_{r,s}}(X,q,t),h_ke_{\ts-k}\rangle$ using the formula above and the superization methods from \cite{HHLRU}. 

The Pieri rule immediately implies that for a triple of partitions $\lambda,\mu,\eta$, one obtains
\begin{equation}\label{formula: <s,eh>}
    \langle s_\lambda,h_\mu e_\eta\rangle=|\SSYT_\pm(\lambda;\mu,\eta)|.
\end{equation}

\begin{lemma}\label{lemma:<f,gh>}
    For any three symmetric functions $f_1,\ f_2,$ and $f_3$, one has
    \begin{equation*}
        \langle f_1,f_2 f_3\rangle=\langle\langle f_1[X+Y],f_2[X]\rangle_X, f_3[Y]\rangle_Y.
    \end{equation*}
\end{lemma}

\begin{proof}
    Both sides are linear with respect to all three functions, hence it is enough to check the equation for $f_1=m_\lambda,$ $f_2=h_\mu,$ and $f_3=h_\eta.$ Then for the left hand side, one gets
    \begin{equation*}
        \langle m_\lambda, h_\mu h_\eta\rangle=\langle m_\lambda, h_{\mu\sqcup\eta}\rangle=\delta_{\lambda,\mu\sqcup\eta},
    \end{equation*}
    where $\mu\sqcup\eta:=\text{sort}(\mu_1,\mu_2,\ldots,\eta_1,\eta_2,\ldots),$
    while for the right hand side, one also gets
    \begin{equation*}
        \left\langle\langle m_\lambda[X+Y],h_\mu[X]\rangle_X, h_\eta[Y]\right\rangle_Y=\left\langle \sum_{\nu\sqcup\mu=\lambda} m_\nu[Y],h_\eta[Y]\right\rangle_Y=\delta_{\lambda,\mu\sqcup\eta}.
    \end{equation*}
    Hence, the statement holds.
\end{proof}

\begin{lemma}\label{lemma: <f,eh>=coef}
    For any triple of partitions $\lambda,\mu,\eta$, the following holds: 
\begin{equation}
    \langle f,h_\mu e_\eta\rangle=\widetilde{f}(X,Y)\vert_{X^\mu Y^\eta}. 
\end{equation}
\end{lemma}

\begin{proof}
Indeed,
\begin{align*}
    \widetilde{f}(X,Y)\vert_{X^\mu Y^\eta}&=\langle\langle \omega_Y f[X+Y],h_\mu[X]\rangle_X,h_\eta[Y]\rangle_Y\\
    &=\langle\omega_Y\langle f[X+Y],h_\mu[X]\rangle_X,h_\eta[Y]\rangle_Y\\
    &=\langle\langle f[X+Y],h_\mu[X]\rangle_X,e_\eta[Y]\rangle_Y\\
    &=\langle f,h_\mu e_\eta\rangle,
\end{align*}
where on the last step we use Lemma \ref{lemma:<f,gh>}.
\end{proof}

\begin{corollary}
    Formula \ref{formula: <s,eh>} and Lemma \ref{lemma: <f,eh>=coef} imply
    \begin{equation}
        \tilde{s}_\lambda(Z)=\sum_{T\in\SSYT_\pm(\lambda)} Z^T.
    \end{equation}
\end{corollary}

\subsection{Quasisymmetric Functions}
Recall Gessel's fundamental quasisymmetric functions from Definition \ref{def:Qsym}. It is well known that Schur functions can be expressed in terms of these functions.

\begin{definition}
Let $S\in \SYT(\lambda)$. For every $k\in\{1,\ldots n\}$, let $\Box_k=(x_k,y_k)\in\lambda$ be the box of $\lambda$ such that $S(\Box_k)=k.$ We say that $k\in\{1,\ldots,n-1\}$ is a \textit{right step} if $\Box_{k+1}$ is to the right of $\Box_k$ (i.e. $x_{k+1}>x_k$), and we say that $k$ is an \textit{up step} if $\Box_{k+1}$ is higher than $\Box_k$ (i.e. $y_{k+1}>y_k$). Note that every $k\in\{1,\ldots,n-1\}$ is either a right step, or an up step, but not both. Then,
\begin{equation*}
    \{1,\ldots,n-1\}=R(S)\sqcup U(S),
\end{equation*}
where $R(S)$ (resp. $U(S)$) is the set of right (resp. up) steps.
\end{definition}

\begin{lemma}[\cite{G84}, Theorem $7$]\label{lemma: schur via Gessel}
For any partition $\lambda$,
$$
s_\lambda(X)=\sum_{S\in\SYT(\lambda)}Q_{|\lambda|,R(S)}(X).
$$
\end{lemma}

We will prove Lemma \ref{lemma: schur via Gessel}, as we will need a modification of this argument later on. 

\begin{proof}
The \textit{standardization map} $\St:\SSYT(\lambda)\to \SYT(\lambda)$ is constructed as follows. Let $T\in\SSYT(\lambda)$ be a semistandard tableau. Then the standard tableau $S:=\St(T)$ is defined by the following conditions:
\begin{enumerate}
    \item If $T(\Box_1)<T(\Box_2)$, then $S(\Box_1)<S(\Box_2),$
    \item If $T(\Box_1)=T(\Box_2)$ and $\Box_1=(x_1,y_1)$ is to the left of $\Box_2=(x_2,y_2)$ (i.e. $x_1<x_2$), then $S(\Box_1)<S(\Box_2).$
\end{enumerate}

Indeed, since if $T(\Box_1)=T(\Box_2)$ then $\Box_1$ and $\Box_2$ cannot be in the same column, the above conditions determine a total order on the values of $S$ on the squares of $\lambda,$ such that the values in the rows increase left to right and values in the columns increase down to up, which uniquely determines the tableau $S=\St(T)\in\SYT(\lambda).$ 

Let us reverse the construction in order to understand the fiber $\St^{-1}(S)\subset\SSYT(\lambda).$ As before, let
$\Box_1,\ldots,\Box_n$ be the boxes of $\lambda$ such that $S(\Box_k)=k$, for all $k.$ Then $T'\in\St^{-1}(S)$ if and only if for all $k\in\{1,\ldots,n-1\}$, and therefore,
\begin{enumerate}
    \item $T'(\Box_k)\le T'(\Box_{k+1}),$
    \item If $T'(\Box_k)= T'(\Box_{k+1})$ then $k\in R(S).$
\end{enumerate}
Hence
\begin{align*}
    s_\lambda(X)=\sum_{T\in\SSYT(\lambda)} X^T=\sum_{S\in\SYT(\lambda)}\sum_{T\in\St^{-1}(S)} X^T=\sum_{S\in\SYT(\lambda)}Q_{|\lambda|,R(S)}(X).
\end{align*}
\end{proof}

Now, recall the super quasi-symmetric functions $\tilde{Q}_{n,J}(Z)$ from Definition \eqref{equation:super Gessel}. In \cite{BHMPS}, the following generalization of Lemma \ref{lemma: schur via Gessel} is proven.

\begin{lemma}[\cite{BHMPS}, Proposition $2.4.2$]\label{lemma:super schur viaa super Gessel}
    One can express the superization of Schur functions in terms of the superization of the Gessel's functions:
\begin{equation}\label{eq: schur via Gessel, super}
    \tilde{s}_\lambda(Z)=\sum_{S\in\SYT(\lambda)}\tilde{Q}_{|\lambda|,R(S)}(Z).
\end{equation}
\end{lemma}

One can prove Lemma \ref{lemma:super schur viaa super Gessel} by extending the standardization map $\St:\SSYT(\lambda)\to\SYT(\lambda)$ to the set of super Young tableaux $\SSYT_\pm(\lambda)$, as we now explain.

Let $T\in\SSYT_\pm(\lambda)$ be a super Young tableau. Then the standard tableau $S:=\St_\pm(T)$ is defined by the conditions:
\begin{enumerate}
    \item If $T(\Box_1)<T(\Box_2)$, then $S(\Box_1)<S(\Box_2),$
    \item If $T(\Box_1)=T(\Box_2)\in \aA_X$ and $\Box_1$ is to the left of $\Box_2$, then $S(\Box_1)<S(\Box_2).$    
    \item If $T(\Box_1)=T(\Box_2)\in \aA_Y$ and $\Box_1=(x_1,y_1)$ is lower than $\Box_2=(x_2,y_2)$ (i.e. $y_1<y_2$), then $S(\Box_1)<S(\Box_2).$
\end{enumerate}  
Similarly to before, these conditions uniquely determine the standard tableau $S:=\St_\pm(T).$ Furthermore, the fibers of the map $\St_\pm:\SSYT_\pm(\lambda)\to\SYT(\lambda)$ can be described as follows. Hence, $T'\in\St_\pm(\lambda)$ if and only if for all $k\in\{1,\ldots, n-1\}$ and
\begin{enumerate}
    \item $T'(\Box_k)\le T'(\Box_{k+1}),$
    \item If $T'(\Box_k)= T'(\Box_{k+1})\in\aA_X$, then $k\in R(S),$
    \item If $T'(\Box_k)= T'(\Box_{k+1})\in\aA_Y$, then $k\in U(S).$
\end{enumerate}
Hence
\begin{align*}
    \tilde{s}_\lambda(Z)=\sum_{T\in\SSYT_\pm(\lambda)} Z^T=\sum_{S\in\SYT(\lambda)}\sum_{T\in\St_\pm^{-1}(S)} Z^T=\sum_{S\in\SYT(\lambda)}\tilde{Q}_{|\lambda|,R(S)}(Z).
\end{align*}

Extending Lemmas \ref{lemma: schur via Gessel} and\ref{lemma:super schur viaa super Gessel} by linearity, we obtain Proposition \ref{prop:super}.

The proofs of Lemmas \ref{lemma: schur via Gessel} and \ref{lemma:super schur viaa super Gessel} also work for skew Schur functions, so that for any skew shape $\nu=\lambda/\mu$, the following equalities hold:
\begin{equation}
    s_\nu(X)=\sum_{T\in\SSYT(\nu)} X^T=\sum_{S\in\SYT(\nu)} Q_{n,R(S)}(X),
\end{equation}
and
\begin{equation}
    \tilde{s}_\nu(Z)=\sum_{S\in\SYT(\nu)} \widetilde{Q}_{n,R(S)}(Z)=\sum_{T\in\SSYT_\pm(\nu)} Z^T,
\end{equation}
where for the first step we use Proposition \ref{prop:super}.

\subsection{Standardization of LLT polynomials}
 Following \cite{HHLRU}, it is now our goal to apply the standardization construction to the polynomials $\mathcal{L}_\pi(X,t)=\sum_{T\in\SSYT(\rho_\pi)}t^{\inv(T)}X^T$ and express them as linear combinations of Gessel's quasi-symmetric functions. The challenge is to modify the standardization map $\St:\SSYT(\rho_\pi)\to \SYT(\rho_\lambda)$ so that it preserves the $\inv$ statistic. Luckily, the skew shape $\rho_\pi$ is a vertical strip, which enables us to do the required modifications.

Consider the map $\ST:\SSYT(\rho_\pi)\to\SYT(\rho_\pi)$ defined as follows: Given $T\in\SSYT(\rho_\pi)$, its standardization $S:=\ST(T)$ is uniquely defined by the conditions:
\begin{enumerate}
    \item If $T(\Box_1)<T(\Box_2)$, then $S(\Box_1)<S(\Box_2),$
    \item If $T(\Box_1)=T(\Box_2)$ and $\phi(\Box_1)<\phi(\Box_2)$, then $S(\Box_1)<S(\Box_2),$
\end{enumerate}
where $\phi$ is given as in Definition \ref{def:phi}. 

Note that the condition on the slope implies that $\phi(\Box_1)\neq\phi(\Box_2)$ for any two distinct squares of $\rho_\pi,$ therefore, the above conditions define a total order on the values of $S$ on the boxes of $\rho_\pi.$ Since $\rho_\pi$ is a vertical strip, to show that this order defines a standard tableau, one only needs to check that it is increasing in columns down to up. But $T$ is strictly increasing in columns, hence by condition ($1$) so does $S$.

Reversing the construction, one gets that for a standard tableau $S\in\SYT(\rho_\pi)$, the preimage $\ST^{-1}(S)\subset\SSYT(\rho_\pi)$ can be described as follows. Let $\Box_1,\ldots,\Box_n$ be the boxes of $\rho_\pi$ such that $S(\Box_k)=k$, for all $k.$ Then $T'\in\ST^{-1}(S)$ if and only if for all $k\in\{1,\ldots,n-1\}$
\begin{enumerate}
    \item $T'(\Box_k)\le T'(\Box_{k+1})$ and
    \item If $T'(\Box_k)= T'(\Box_{k+1})$ then $\phi(\Box_k)<\phi(\Box_{k+1}).$
\end{enumerate}

\begin{definition}
    The set $d_\phi(S)$ of \newword{$\phi$-descents} of a standard tableau $S\in\SYT(\rho_\pi)$ is given by
    \begin{equation*}
        d_\phi(S):=\left\{k\in\{1,\ldots,n-1\}|\phi(\Box_k)>\phi(\Box_{k+1})\right\}.
    \end{equation*}
\end{definition}

\begin{lemma}\label{lemma: st preserves inv}
    For any semistandard tableau $T\in \SSYT(\rho_\pi)$ one gets
    \begin{equation*}
        \inv(T)=\inv(\ST(T)).
    \end{equation*}
\end{lemma}

\begin{proof}
    Denote $S:=\ST(T).$ Let $(\Box_1,\Box_2)$ be an attacking pair of boxes of $\rho_\pi$ (i.e. $|\phi(\Box_1)-\phi(\Box_2)|<1$). Without loss of generality, assume that $\phi(\Box_1)>\phi(\Box_2).$ The pair $(\Box_1,\Box_2)$ forms an attacking inversion of $T$ if and only if $T(\Box_1)<T(\Box_2),$ in which case one also has $S(\Box_1)<S(\Box_2),$ so the pair $(\Box_1,\Box_2)$ also forms an attacking inversion of $S.$

    Conversely, the pair $(\Box_1,\Box_2)$ forms an attacking inversion of $S$ if and only if $S(\Box_1)<S(\Box_2).$ Since $\phi(\Box_1)>\phi(\Box_2),$ this implies that $T(\Box_1)<T(\Box_2),$ so the pair $(\Box_1,\Box_2)$ also forms an attacking inversion for $T.$

    We conclude that the attacking inversions of $T$ and of $S=\ST(T)$ are in bijection, hence $\inv(T)=\inv(\ST(T)).$
\end{proof}

Lemma \ref{lemma: st preserves inv} and the above construction imply

\begin{lemma}\label{lemma:LLT via Gessel}
    For any $\pi \in \D(r,s)$,
    \begin{equation*}
        \mathcal{L}_\pi(X,t)=\sum_{S\in\SYT(\rho_\pi)} t^{\inv(S)}Q_{n,d_\phi(S)}(X).
    \end{equation*}
\end{lemma}

Applying Proposition \ref{prop:super}, we also obtain expansions for the superizations of the LLT polynomials:
\begin{equation}\label{equation: super LLT via super Gessel}   \tilde{\mathcal{L}}_\pi(Z,t)=\sum_{S\in\SYT(\rho_\pi)} t^{\inv(S)}\tilde{Q}_{n,d_\phi(S)}(Z).
\end{equation}

Now, consider an extension of the map $\ST:\SSYT(\rho_\pi)\to\SYT(\rho_\pi)$ to the set $\SSYT_\pm(\rho_\pi)$ of super tableaux, defined as follows. Given a super tableau $T\in\SSYT_\pm(\rho_\pi)$, its standardization $S:=\ST_\pm(T)\in\SYT(\rho_\pi)$ is uniquely determined by the following conditions:

\begin{enumerate}
    \item If $T(\Box_1)<T(\Box_2)$, then $S(\Box_1)<S(\Box_2),$
    \item If $T(\Box_1)=T(\Box_2)\in \mathcal{A}_X$ and $\phi(\Box_1)<\phi(\Box_2)$, then $S(\Box_1)<S(\Box_2),$
    \item If $T(\Box_1)=T(\Box_2)\in \mathcal{A}_Y$ and $\phi(\Box_1)<\phi(\Box_2)$, then $S(\Box_1)>S(\Box_2).$
\end{enumerate}

Similar to before, these conditions uniquely define the standard tableau $S=\ST_\pm(T).$ Furthermore, one can reverse the construction to compute the fiber $\ST_\pm^{-1}(S).$ As before, let $\Box_1,\ldots,\Box_n$ be the boxes of $\rho_\pi,$ such that $S(\Box_k)=k$ for all $k.$ Then $T'\in\ST_\pm^{-1}(S)$ if and only if for all $k\in\{1,\ldots,n-1\}$, and
\begin{enumerate}
    \item $T'(\Box_k)\le T'(\Box_{k+1}),$
    \item If $T'(\Box_k)= T'(\Box_{k+1})\in\mathcal{A}_X$, then $\phi(\Box_k)<\phi(\Box_{k+1}),$ and
    \item If $T'(\Box_k)= T'(\Box_{k+1})\in\mathcal{A}_Y$, then $\phi(\Box_k)>\phi(\Box_{k+1}).$
\end{enumerate}

In particular, for any $S\in\SYT(\rho_\pi)$, it follows that

\begin{equation*}
    \sum_{T\in\ST_\pm^{-1}(S)} Z^T=\sum_{\substack{a_1\le\ldots\le a_n\in\aA\\ a_i=a_{i+1}\in\aA_+\Rightarrow \phi(\Box_k)<\phi(\Box_{k+1})\\
    a_i=a_{i+1}\in\aA_-\Rightarrow \phi(\Box_k)<\phi(\Box_{k+1})}} Z^T=\sum_{\substack{a_1\le\ldots\le a_n\in\aA\\ a_i=a_{i+1}\in\aA_+\Rightarrow k\notin d_\phi(S)\\
    a_i=a_{i+1}\in\aA_-\Rightarrow k\in d_\phi(S)}} Z^T=\widetilde{Q}_{n,d_\phi(S)}(Z).
\end{equation*}
The definition of $\inv$ statistic can be extended to the super tableaux using $\ST_\pm:$
\begin{definition}
    Let $T\in\SSYT_\pm(\rho_\pi)$ and set
    \begin{equation*}
        \inv(T):=\inv(\ST_\pm(T)).
    \end{equation*}
\end{definition}
Then the following holds:
\begin{equation*}
\tilde{\mathcal{L}}_\pi(Z,t)=\sum_{S\in\SYT(\rho_\pi)} t^{\inv(S)}\tilde{Q}_{n,d_\phi(S)}(Z)=\sum_{T\in\SSYT_\pm(\rho_\pi)} t^{\inv(T)} Z^T.    
\end{equation*}
Now, recall the notion of attacking inversions from Definition \ref{def:invSSTY}. Applying Lemma \ref{lemma:LLT via Gessel} to the RHS of the shuffle theorem under any line one implies:
\begin{align*}
\Hik_{\tau_{r,s}}(X;q,t)=&\sum_{\pi \in \D(r,s)} q^{\area(\pi)}t^{\dinv(\pi)-I(\pi)}\mathcal{L}_{\pi}(X;t)\\
=&\sum_{\pi \in \D(r,s)} q^{\area(\pi)}t^{\dinv(\pi)-I(\pi)}\sum_{S\in\SYT(\rho_\pi)}t^{\inv(S)}Q_{n,d_\phi(S)}(X).
\end{align*}
Applying Proposition \ref{prop:super} and \eqref{equation:super Gessel}, one can extend these formula to the superizations.
\begin{align*}
    \widetilde{\Hik}_{\tau_{r,s}}(Z;q,t)=&\sum_{\pi \in \D(r,s)} q^{\area(\pi)}t^{\dinv(\pi)-I(\pi)}\sum_{S\in\SYT(\rho_\pi)}t^{\inv(S)}\widetilde{Q}_{n,D_\phi(S)}(Z)\\
    =&\sum_{\pi \in \D(r,s)} q^{\area(\pi)}t^{\dinv(\pi)-I(\pi)}\sum_{T\in\SSYT_\pm(\rho_\pi)} t^{\inv(T)}Z^T.    
\end{align*}
We now can apply Lemma \ref{lemma: <f,eh>=coef} to get the following formula:
\begin{align*}
    \langle {\Hik}_{\tau_{r,s}}(X;q,t),h_\mu e_\eta\rangle=&\;\widetilde{\Hik}_{\tau_{r,s}}(X,Y;q,t) \vert_{X^\mu Y^\eta}\\
    =&\sum_{\pi \in \D(r,s)} q^{\area(\pi)}t^{\dinv(\pi)-I(\pi)}\sum_{T\in\SSYT_\pm(\rho_\pi;\mu,\eta)}t^{\inv(T)}.
\end{align*}
Specializing, the result now follows. 

\section{Sequences for Links}\label{sec:AppendixB}

Let $a$ and $b$ be positive relatively prime integers, and $g$ be an integer greater than $1.$ Consider $\tau=\tau_{ga,gb}.$ Let also $c$ and $d$ be the unique positive integers such that $c\le a,$ $d\le b,$ and $ad-bc=1.$ Set $m:=ga+c,$ $n:=gb+d,$ and $\ell=mn-dn-1=(gb+d)ga-1.$
Finally, set $(r,s):=(\ell/n,\ell/m)=\left(\frac{(gb+d)ga-1}{gb+d},\frac{(gb+d)ga-1}{ga+c}\right).$

\begin{lemma}\label{lemma: Kdmdn as Luv}
    In the notations as above one obtains
    \begin{enumerate}
        \item $\tau_{r,s}=\tau=\tau_{ga,gb},$
        \item ${\bf u}(m,n,\ell)=0^{g-1}10^{g(a-1)}$ and ${\bf v}(m,n,\ell)=0^{gb}1.$
    \end{enumerate}
\end{lemma}

\begin{proof}
    According to Lemma \ref{lemma: x and y from m,n,l} one has
\begin{align*}
x_i &= \begin{cases}
    1 & \text{ if } i\equiv (gb+d)ga-1 \text{ mod } (ga+c) \\
    0 & \text{ if } i\equiv (gb+d)ga - (gb+d)x -1 \text{ mod } (ga+c) \text{ for } x\in \left\{1,\ldots,\left\lfloor \frac{(gb+d)ga-1}{gb+d}\right\rfloor\right\}\\
    \bullet & \text{ otherwise}
    \end{cases}\\
    y_j &= \begin{cases}
    1 & \text{ if } j\equiv (gb+d)ga-1 \text{ mod } (gb+d) \\
    0 & \text{ if } j\equiv (gb+d)ga - (ga+c)y-1 \text{ mod } (gb+d) \text{ for } y\in \left\{1,\ldots,\left\lfloor \frac{(gb+d)ga-1}{ga+c}\right\rfloor\right\}\\
    \bullet & \text{ otherwise}
    \end{cases}
\end{align*}

Note that $gb<\frac{(gb+d)ga-1}{ga+c}<gb+1.$ Indeed, cross-multiplying and using $ad-bc=1$ one gets

\begin{equation*}
    g^2ab+gda-1-g^2ab-gbc=g-1>0,
\end{equation*}
for the left inequality, and

\begin{equation*}
    g^2ab+gbc+ga+c-g^2ab-gad+1=g(a-1)+c+1>0
\end{equation*}
for the right inequality. Note also that 
\begin{equation*}
 (gb+d)a=gab+ad=gab+bc+1=b(ga+c)+1\equiv 1\text{ mod }ga+c.   
\end{equation*}

Using this and simplifying, one gets

\begin{align*}
x_i &= \begin{cases}
    1 & \text{ if } i\equiv g-1 \text{ mod } (ga+c) \\
    0 & \text{ if } i\equiv g - (gb+d)x - 1 \text{ mod } (ga+c) \text{ for } x\in \{1,\ldots,ga-1\}\\
    \bullet & \text{ otherwise}
    \end{cases}\\
    y_j &= \begin{cases}
    1 & \text{ if } j\equiv -1 \text{ mod } (gb+d) \\
    0 & \text{ if } j\equiv (gb+d)ga - (ga+c)y - 1 \text{ mod } (gb+d) \text{ for } y\in \{1,\ldots,gb\}\\
    \bullet & \text{ otherwise}
    \end{cases}
\end{align*}

Taking $x=ka$ with $k\in\{1,2,\ldots,g-1\}$ it follows that $g-(gb+d)ak-1\equiv g-k-1$ mod $(ga+c).$ One concludes that $x_0=x_1=\ldots=x_{g-2}=0,$ $x_{g-1}=1,$ and there are $ga-1$ zeros in ${\bf x}.$ Hence ${\bf u}=0^{g-1}10^{g(a-1)}.$ 

Also, there are $gb$ zeros in ${\bf y}$ and the its last entry is $1.$ Hence ${\bf v}=0^{gb}1.$
\end{proof}

\end{document}